\documentclass[a4paper,12pt,english,english]{article}
\usepackage[latin1]{inputenc}
\usepackage[T1]{fontenc}
\usepackage{amsfonts}
\usepackage{amsmath}
\usepackage{amscd}
\usepackage{lipsum}
\usepackage{graphicx}
\usepackage{mathrsfs}
\usepackage{titletoc}
\usepackage{babel}
\usepackage{amssymb}
\usepackage{bbold}
\usepackage[colorlinks=true]{hyperref}
\hypersetup{urlcolor=blue, linkcolor=black}
\usepackage{fancyhdr}
\usepackage{textcomp}
\usepackage{amsthm}
\usepackage[toc]{appendix}
\usepackage[a4paper,left=2cm,right=2cm,top=2cm,bottom=2cm]{geometry}
\newtheorem{thm}{Theorem}[section]
\newtheorem{dfn}[thm]{Definition}
\newtheorem{lem}[thm]{Lemma}
\newtheorem{cor}[thm]{Corollary}
\newtheorem{rmq}[thm]{Remark}
\newtheorem{prp}[thm]{Proposition}

\makeatletter
\renewcommand{\thesection}{\@arabic\c@section}
\makeatother
\makeatletter
\makeatletter
\def\toclevel@paragraph{5}
\def\toclevel@subparagraph{6}
\makeatother
\usepackage{setspace}

\begin{document}
\vspace*{1cm}
\thispagestyle{empty}
\begin{center}
{\Large{ A Spatial Stochastic Epidemic Model: \vspace*{0.21cm}\\Law of Large Numbers and Central Limit Theorem }}\vspace*{0.5cm}\\
S. Bowong$^{\ast}$\quad A. Emakoua$^{\dagger}$\quad E. Pardoux$^{\ddagger}$ \vspace*{0.21cm}\\
\today
\end{center}
\vspace*{0.8cm}
\begin{center}
\textbf{\large{Abstract}}
\end{center}
\par We consider an individual-based SIR  stochastic epidemic  model   in continuous space. The evolution of the epidemic involves the rates of infection and cure of individuals. We assume that individuals move randomly on the  two-dimensional torus according to   independent Brownian motions. We define the empirical measures  $\mu^{S,N}$, $\mu^{I,N}$ and  $\mu^{R,N}$ which describe the evolution of the position
of the susceptible, infected and  removed individuals. We prove the  convergence in propbability, as $N\rightarrow \infty$, of the sequence $(\mu^{S,N},\mu^{I,N})$  towards  $(\mu^{S},\mu^{I})$  solution of a system of parabolic PDEs. We show that the sequence  $(U^{N}=\sqrt{N}(\mu^{S,N}-\mu^{S}),V^{N}=\sqrt{N}(\mu^{I,N}-\mu^{I}))$ converges in law, as $N\rightarrow\infty$, towards a Gaussian distribution valued process, solution of a system of linear PDEs with highly singular Gaussian driving processes. In the case where the individuals do not move  we also define and study the law of large numbers and central limit theorem  of the sequence  $(\mu^{S,N},\mu^{I,N})$.\vspace*{0.2cm}\\
\textit{\textbf{Keywords}}: Measure-valued process; Stochastic Epidemic   model; Law of large numbers;  Central limit theorem. 
\section{Introduction}
\hspace*{0.6cm}Recent studies on homogeneous stochastic models of epidemics of Anderson and Britton [\ref{dc}], and Britton and Pardoux [\ref{hc}] show that deterministic models of epidemics are the law of large numbers limits (as the size of the population tends to $\infty$) of stochastic models, while the central limit theorem and moderate and large deviations (see [\ref{hc}] and [\ref{vc}])  give tools to accurately describe the gap between stochastic  and deterministic models. However, the homogeneity hypothesis made in these models is not realistic since it means that when an infected individual infects a susceptible individual, this latter is uniformly chosen in the population, while in  real life an infected individual infects a susceptible person who is geographically close to him. Hence the interest of modelling epidemics for a population distributed in a continuous space with or without displacement of the population. The case without displacement is important for populations of plants. The corresponding deterministic models are beginning to be well studied in the literature, see for example [\ref{cc}]. \vspace*{0.1cm}\\\hspace*{0.6cm}In this paper we study the law of large numbers and the central limit theorem of two stochastic SIR epidemic models for a population distributed on the two dimentional torus. The only difference between the two models is that in the first one  the population is moving and in the second it does not move.
More precisely we consider a population of size N distributed on the  torus ($\mathbb{T}^{2}=\mathbb{R}^{2}/\mathbb{Z}^{2}$) such that at any time each individual is either susceptible(S), infectious(I) or recovered(R). Let S(t), I(t) and R(t) denote the number of individuals in the different states at time t.\newpage
During the epidemic an individual $i$  moves on the torus  according to the processes \\ 
$\{X^{i}_{t}=\varPi(\widetilde{X}^{i}_{t}),\hspace*{0.1cm}t\geq 0\}$, where $\varPi$ is the canonical projection from $\mathbb{R}^{2}$ to $\mathbb{T}^{2}$, $\widetilde{X}^{i}_{t}=X^{i}+\sqrt{2\gamma}B_{t}^{i}$, with $\{ X^{i},1\leq i\leq N \}$  an independent and identically distributed family of random variables, globally independent of  $\{ B^{i},1\leq i\leq N \},$ which in turn is  a family of  independent standard   Brownian motions on the torus ($\gamma$ is a positive constant in the first model and is zero in the second). 
We assume   at time t=0 (for $\gamma\geq 0$), that  the population consists of two classes S(0) and  I(0) described as follows.\\
Let $ A$ be an arbitrary Borel subset  of $ \mathbb{T}^{2}$ and $ 0 <p \le 1 $, each individual $i $ is placed in $\mathbb{T}^{2}$ independently of the others at the position $X^{i}$. If  $X^{i}\in A^{c}$  then the individual $i$   is susceptible and  if $X^{i}\in A $, the individual $i$ is infected with  probability $p$ and susceptible with  probability $1-p$. This situation is modelled by empirical measures $
                \hspace*{3cm} \\
                  	       \hspace*{5cm}\displaystyle  \mu_{0}^{S,N}=\frac{1}{N}\sum_{i=1}^{N}\{1_{A}(X^{i})(1-\xi_{i})+1_{A^{c}}(X^{i})\}\delta_{X^{i}}
                         \\\hspace*{5cm}
                         \displaystyle  \mu_{0}^{I,N}=\frac{1}{N}\sum_{i=1}^{N}1_{A}(X^{i})\xi_{i}\delta_{X^{i}}
                         \\\hspace*{5cm}
                         \displaystyle \mu_{0}^{N}=\mu_{0}^{S,N}+\mu_{0}^{I,N}=\frac{1}{N}\sum_{i=1}^{N}\delta_{X^{i}}, 
                         $\vspace*{0.07cm}\\
where  $\{\xi_{i},1\leq i\leq N \}$ is a mutually independent family of Ber($p$) random variables, globally independent of $\{X^{i},1\leq i\leq N \}$.\vspace*{0.08cm}\\Let K be a function defined  by \vspace*{-0.4cm}
\begin{align*}
   K:\hspace*{0.1cm}& \mathbb{T}^{2}\times \mathbb{T}^{2} \rightarrow \mathbb{R}_{+}\\
     & (x,y)\mapsto k(d_{\mathbb{T}^{2}}^{2}(x,y))
   \end{align*} 
   \vspace*{-0.1cm}
 where $k:$ $\mathbb{R}_{+}\rightarrow \mathbb{R}_{+}$ is a non-increasing function which is non zero only in a neighbourhood of zero in such a way that by considering for any $x,y\in \mathbb{T}^{2}$, $d_{\mathbb{T}^{2}}(x,y)=\underset{a\in \mathbb{Z}^{2}}{\inf}\{\lVert x-y-a\lVert\}$, one has \vspace*{-0.12cm}\[ \tag{1.1}\forall x\in \mathbb{T}^{2},\hspace*{0.03cm} y \in \textrm{support}\{K(x,.)\} \hspace*{0.14cm}\textrm{iff}\hspace*{0.14cm} d_{\mathbb{T}^{2}}(x,y)=\lVert x-y\lVert.\label{p1}\]
Let $ E^{i}_{t}$ be the state of the individual i at time t, $ E^{i}_{t}\in \{S, I, R\}$ and 
 $\beta $ a positive constant. In the homogeneous model [\ref{hc}], the rate of infectious contacts  can be thought of as a product of a rate $c'$ at which each infectious individual has contacts with others, and the probability $p'$ that such a contact results in an infection given that the other person is susceptible, which happens with the probability  $S_{t} / N$,  where N is the total population size, because all the individuals  have the same probability to be in contact with an infectious. However, in our case the rate  $P_{j,i}$ at which an infectious $j$ has a contact with a susceptible $i$ depends upon  the  distance that separates them, so $ P_{j,i}$  is proportional to $ K(X_{t}^{i}, X_{t} ^{j})$, thus \\
 $\hspace*{0.5cm}\bullet$ $P_{j,i}= \frac{\beta K(X_{t}^{i},X_{t}^{j})}{\sum_{l=1}^{N}K(X_{t}^{l},X_{t}^{j})} $\\
   $\hspace*{0.5cm}\bullet$ An infectious individual $j$ has an infectious contact with  susceptible individuals at the \hspace*{0.8cm} rate $  \beta\frac{\sum_{i=1}^{N}K(X_{t}^{i},X_{t}^{j})1_{\{E_{t}^{i}=S\}}}{\sum_{l=1}^{N}K(X_{t}^{l},X_{t}^{j})}. $\vspace*{0.1cm} \\
Hence  infectious individuals  have  infectious contacts with  susceptible individuals  at the rate $\beta \sum_{j=1}^{N}\frac{\sum_{i=1}^{N}K(X_{t}^{i},X_{t}^{j})1_{\{E_{t}^{i}=S\}}}{\sum_{l=1}^{N}K(X_{t}^{l},X_{t}^{j})} 1_{\{E_{t}^{j}=I\}}.$\vspace*{0.2cm}\\
The epidemic evolves according to  the following rules:\\
 $\hspace*{0.5cm} \bullet$ A susceptible $i$ becomes infected at time t at the rate $\beta 1_{\{E_{t}^{i}=S\}}\sum_{j=1}^{N}\frac{K(X_{t}^{i},X_{t}^{j})}{\sum_{l=1}^{N}K(X_{t}^{l},X_{t}^{j})}1_{\{E_{t}^{j}=I\}} $\\
 $\hspace*{0.5cm} \bullet$ Each infected individual cures at rate $\alpha$  independently of the others, of the number of \hspace*{0.8cm} infected and of the respective positions of the individuals.\vspace*{0.1cm}\\
The evolution of the numbers of susceptible, infected and removed individuals is described by the following equations.\\
 $
	       \displaystyle  S(t)=S(0)-P_{inf}\left( \beta \int_{0}^{t}  \sum_{j=1}^{N}\frac{\sum_{i=1}^{N}K(X_{r}^{i},X_{r}^{j})1_{\{E_{r}^{i}=S\}}}{\sum_{l=1}^{N}K(X_{r}^{l},X_{r}^{j})} 1_{\{E_{r}^{j}=I\}} dr\right)
       \\
       \displaystyle  I(t)=I(0)+P_{inf}\left( \beta \int_{0}^{t}  \sum_{j=1}^{N}\frac{\sum_{i=1}^{N}K(X_{r}^{i},X_{r}^{j})1_{\{E_{r}^{i}=S\}}}{\sum_{l=1}^{N}K(X_{r}^{l},X_{r}^{j})} 1_{\{E_{r}^{j}=I\}} dr\right)-P_{cu}\left(\alpha \int_{0}^{t}\sum_{j=1}^{N}1_{\{E_{r}^{j}=I\}} dr\right)
       \\
       \displaystyle R(t)=R(0) + P_{cu}\left(\alpha \int_{0}^{t}\sum_{j=1}^{N}1_{\{E_{r}^{j}=I\}} dr\right)
       $\\
       where $P_{inf}$ and $P_{cu}$ are two independent standard Poisson processes.\vspace*{0.2cm}\\
       We now define the renormalized point processes,
      $ \forall t>0, $ \\\hspace*{5cm}
            $ \displaystyle  \mu_{t}^{S,N}=\frac{1}{N}\sum_{i=1}^{N}1_{\{E_{t}^{i}=S\}}\delta_{X_{t}^{i}}
                \\\hspace*{5cm}
                \displaystyle  \mu_{t}^{I,N}=\frac{1}{N}\sum_{i=1}^{N}1_{\{E_{t}^{i}=I\}}\delta_{X_{t}^{i}}
                \\\hspace*{5cm}
                 \displaystyle  \mu_{t}^{R,N}=\frac{1}{N}\sum_{i=1}^{N}1_{\{E_{t}^{i}=R\}}\delta_{X_{t}^{i}}
                 \\\hspace*{5cm}
                 \displaystyle  \mu_{t}^{N}=\frac{1}{N}\sum_{i=1}^{N}\delta_{X_{t}^{i}}=\mu_{t}^{S,N}+\mu_{t}^{I,N}+\mu_{t}^{R,N}. $\vspace*{0.2cm}\\ \hspace*{0.6cm}We first study the law of large numbers and the central limit theorem of the sequence $(\mu_{0}^{S,N},\mu_{0}^{I,N},\mu_{0}^{N})_{N \geq 1 },$ then we study the law of large numbers and the central limit theorem of the sequences  $(\mu^{N})_{N\geq 1}$ and ($\mu^{S,N},\mu^{I,N})_{N\geq1} $ when $\gamma$ is a positive constant, and finally  when $\gamma$ is zero.\vspace*{0.15cm}\\
\hspace*{0.6cm}The paper is  organized  as follows.
In section 2 we recall some results that will be useful in the sequel. In  sections 3 and 4   we   study the law of large numbers and the central  limit theorem of the sequence $(\mu_{0}^{S,N},\mu_{0}^{I,N},\mu_{0}^{N})_{N \geq 1 }$. In section 5, for $\gamma>0$, we  first establish the evolution equations of the measure-valued process   $\mu^{S,N}$ and $\mu^{I,N} $  then we  show that $\mu^{N}$ converges in probability as $N\rightarrow \infty$ towards the processes $\{\mu_{t}, t\geq 0 \}$ where for each $t\geq 0$, $\mu_{t}$ is the law $ X_{t}^{1}$ and  finally we prove that    $\{(\mu_{t}^{S,N},\mu_{t}^{I,N}) ,t\geq0 \} $ converges in probability as $N\rightarrow \infty $  towards $(\mu^{S},\mu^{I})$ solution of a system of parabolic PDEs. In section 6 we study the convergence of the  sequences  $Z^{N}=\sqrt{N}(\mu^{N}-\mu)$ and ($U^{N}=\sqrt{N}(\mu^{S,N}-\mu)$,$V^{N}=\sqrt{N}(\mu^{I,N}-\mu^{I})$). Finally in section 7 we assume that the individuals do not move ($\gamma = 0$) then we study the  law of large numbers and the central limit theorem of the sequence  $\{(\mu_{t}^{S,N},\mu_{t}^{I,N}),N\geq1, t\geq0 \} $. \vspace*{0.15cm}\\
  \text{\textbf{Notation:}}\vspace*{0.1cm}\\
 \hspace*{0.5cm}-- $\mathcal{M}_{F}(\mathbb{T}^{2})$ denotes the space of finite measures on $\mathbb{T}^{2}$. \\
  \hspace*{0.5cm}--  $C(\mathbb{T}^{2})$ denotes the space of continuous functions on $\mathbb{T}^{2}$. \\
  \hspace*{0.5cm}-- $\forall \mu \in \mathcal{M}_{F}(\mathbb{T}^{2}) $ and $ \varphi \in C(\mathbb{T}^{2})$,  we denote the integral  $\int_{\mathbb{T}^{2}} \varphi(x)\mu(dx)$ by $(\mu,\varphi)$.\vspace*{0.2cm}  \\ \hspace*{0.5cm}-- We define  the Fortet distance on $\mathcal{M}_{F}(\mathbb{T}^{2})$ by $\forall \mu , \nu \in \mathcal{M}_{F}(\mathbb{T}^{2})$\\\hspace*{4cm}
                      $d_{F}(\mu , \nu)=\underset{\underset{\lVert f\lVert_{\infty\leq 1,\parallel f\parallel_{L}\leq 1}}{f\in C(\mathbb{T}^{2})}}{\sup} \mid (\mu,f)-(\nu,f)\mid, $\\\hspace*{0.8cm} where $ \parallel f\parallel_{L}=\underset{x\neq y}{\sup}(\mid f(x)-f(y)\mid)/d_{\mathbb{T}^{2}}(x,y). $\\\hspace*{0.8cm} This distance induces the topology of  weak convergence, in other words the sequence of \\\hspace*{0.8cm} measures  $ \mu_{n}$ converges weakly towards $ \mu $ if and only if  $ \lim\limits_{n\rightarrow \infty}d_{F}(\mu_{n},\mu)= 0$ \\\hspace*{0.5cm}-- In  the following, the letter $C$ will denote a (constant) positive  real number which can \hspace*{0.9cm}change from line to line. \\\hspace*{0.5cm}-- We equip  $\mathcal{M}_{F}(\mathbb{T}^{2})$   with the topology of weak convergence.
\section{Preliminaries \label{pr}}
        \begin{dfn}
              Sobolev spaces  (see [\ref{ac}])\vspace*{0.2cm}\\
               \hspace*{0.2cm}1) Let m $\in \mathbb{N}, p\in \mathbb{R}_{+}$ the Sobolev space $W^{m,p}(\mathbb{T}^{2})$ is defined by:\vspace*{0.1cm}\\ 
                                \hspace*{1cm}$ W^{m,p}(\mathbb{T}^{2})=\{\varphi\in L^{p}(\mathbb{T}^{2}):D^{\eta}\varphi\in L^{p}(\mathbb{T}^{2}),\forall \eta=(\eta_{1},\eta_{2})\in \mathbb{N}^{2}$   such that  $  \lvert \eta\lvert=\eta_{1}+\eta_{2} \leq m\}$\\
                                \hspace*{1cm}where $D^{\eta}\varphi$ is the weak derivative  of the function $\varphi$ with respect to the multi-index $\eta.$ \vspace*{0.1cm}\\
                             \hspace*{0.1cm} 2) Let $s=n+ \sigma$ with $n \in \mathbb{N}$ and $\sigma\in ]0,1[ $,  the Sobolev space $ W^{s,p}(\mathbb{T}^{2}) $ is defined as follows \vspace*{0.1cm}\\
                              \hspace*{1cm} $ W^{s,p}(\mathbb{T}^{2})=\{\varphi \in  W^{m,p}(\mathbb{T}^{2}) / \sum\limits_{\lvert \eta \lvert=n}\int_{\mathbb{T}^{2}\times \mathbb{T}^{2}}\frac{\lvert D^{\eta}(\varphi(x)-D^{\eta}\varphi(x')\lvert^{2}}{(d_{\mathbb{T}^{2}}(x,x'))^{2(1+\sigma)}}dxdx'<\infty \}. $\\
                              \hspace*{0.4cm} Notice that for p=2 and $s\in \mathbb{R}_{+}$, $W^{s,2}(\mathbb{T}^{2})$  is denoted $ H^{s}(\mathbb{T}^{2})$ and is a Hilbert space.
                  \end{dfn}
\begin{prp}(see [\ref{gc}])\\\hspace*{0.1cm} Consider the following sets of functions.
     \begin{align*}
          \mathcal{A}_{1}&=\{f_{n_{1},n_{2}}^{1}(x_{1},x_{2})=2sin(\pi n_{1} x_{1})cos(\pi n_{2} x_{2}),n_{1}>0  ,n_{2}>0\quad even \}\\
              \mathcal{A}_{2}&=\{f_{n_{1},n_{2}}^{2}(x_{1},x_{2})=2sin(\pi n_{1} x_{1})sin(\pi n_{2} x_{2}),n_{1}>0  ,n_{2}>0\quad even \}\\
              \mathcal{A}_{3}&=\{f_{n_{1},n_{2}}^{3}(x_{1},x_{2})=2cos(\pi n_{1} x_{1})cos(\pi n_{2} x_{2}),n_{1}>0 ,n_{2}>0\quad even \}\\
              \mathcal{A}_{4}&=\{f_{n_{1},n_{2}}^{4}(x_{1},x_{2})=2cos(\pi n_{1} x_{1})sin(\pi n_{2} x_{2}),n_{1}>0  ,n_{2}>0\quad even \}\\
              \mathcal{A}_{5}&=\{f_{n_{1},0}^{5}(x_{1},x_{2})=\sqrt{2}cos(\pi n_{1} x_{1}),f_{n_{1},0}^{6}(x_{1},x_{2})=\sqrt{2}sin(\pi n_{1} x_{1}),n_{1}>0  \hspace*{0.1cm} even \}\\
              \mathcal{A}_{6}&=\{f_{0,n_{2}}^{7}(x_{1},x_{2})=\sqrt{2}cos(\pi n_{2} x_{2}),f_{0,n_{2}}^{8}(x_{1},x_{2})=\sqrt{2}sin(\pi n_{2} x_{2}),n_{2}>0 \hspace*{0.1cm}   even \}.
              \end{align*}
                              \hspace*{0.2cm}For any $\gamma >0$,\vspace*{0.1cm}\\
             1)  $\{f^{0}=1$, $ (\mathcal{A}_{i}$, $i\in \{1,2,3,4\})$, $  \mathcal{A}_{5}$, $\mathcal{A}_{6}\} $ are eigenfunctions of the operator $\gamma\bigtriangleup$ on $\mathbb{T}^{2},$ with eigenvalues $\{\lambda_{0}=0$, $ -\lambda_{n_{1},n_{2}}=-\gamma\pi^{2}(n_{1}^{2}+n_{2}^{2})$, $-\lambda_{n_{1}}=-\gamma\pi^{2}n_{1}^{2}$, $-\lambda_{n_{2}}=-\gamma\pi^{2}n_{2}^{2}\}$ respectively, they   form  an orthonormal basis of $\mathbb{L}^{2}(\mathbb{T}^{2})$. \\
                2) $\Big\{ \rho^{0}= 1, \{\rho_{n_{1},n_{2}}^{i,s}= \frac{f_{n_{1},n_{2}}^{i}}{(1+\gamma\pi^{2}(n_{1}^{2}+n_{2}^{2}))^{\frac{s}{2}}} , i\in \{1,2,3,4 \} \},  \{\rho_{n_{1},0}^{i,s}= \frac{f_{n_{1},0}^{i}}{(1+\gamma\pi^{2}n_{1}^{2})^{\frac{s}{2}}}i\in \{5,6 \}\},\\\hspace*{0.6cm}\{\rho_{0,n_{2}}^{i,s}= \frac{f_{0,n_{2}}^{i}}{(1+\gamma\pi^{2}n_{2}^{2})^{\frac{s}{2}}},i\in \{7,8 \}\} \Big\} $ is an orthonormal basis of $H^{s}(\mathbb{T}^{2}).$\label{b2}
       \end{prp}
   \begin{prp}
        \hspace*{0.1cm}Parseval identity (see the lemma 6.52 [\ref{ac}]) \\
        For any $ s>0$,  $\varphi \in H^{s}(\mathbb{T}^{2})$,  $A \in H^{-s}(\mathbb{T}^{2}) $\\
              $ \lVert A\lVert_{H^{-s}}=\underset{\varphi\neq 0,\varphi\in H^{s}}{\sup}\frac{\lvert (A, \varphi)\lvert}{\lVert\varphi\lVert_{H^{s}}}\\\lVert A\lVert_{H^{-s}}^{2}=\lvert( A, \rho^{0})\lvert^{2}+\sum\limits_{\underset{n_{1}>0,even}{i\in \{5,6\}}}\lvert( A, \rho_{n_{1},0}^{i,s})\lvert^{2}+\sum\limits_{\underset{n_{2}>0,even}{i\in \{7,8\}}}\lvert (A, \rho_{0,n_{2}}^{i,s})\lvert^{2}+\sum\limits_{\underset{n_{1}>0,n_{2}>0,even}{i\in \{1,2,3,4\}}}\lvert( A, \rho_{n_{1},n_{2}}^{i,s})\lvert^{2}\\\hspace*{1.3cm} =
               \sum\limits_{i,n_{1},n_{2}}\lvert( A, \rho_{n_{1},n_{2}}^{i,s})\lvert^{2}  $\vspace*{0.3cm}\\
       $  \lVert \varphi\lVert_{H^{s}}^{2}=(\varphi,f^{0})_{L^{2}}^{2}+\sum\limits_{\underset{n_{1}>0, even}{i\in\{5,6\}}}(1+\gamma\pi^{2}n_{1}^{2})^{s}(\varphi,f_{n_{1},0}^{i})_{L^{2}}^{2}+\sum\limits_{\underset{n_{2}>0, even}{i\in\{7,8\}}}(1+\gamma\pi^{2}n_{2}^{2})^{s}(\varphi,f_{0,n_{2}}^{i})_{L^{2}}^{2}\\\hspace*{1.5cm}+
      \sum\limits_{\underset{n_{1}>0,n_{2}>0, even}{i\in\{1,2,3,4\}}}(1+\gamma\pi^{2}(n_{1}^{2}+n_{2}^{2}))^{s}(\varphi,f_{n_{1},n_{2}}^{i})_{L^{2}}^{2} $\\$\hspace*{1cm}=\sum\limits_{i,n_{1},n_{2}}(1+\gamma\pi^{2}(n_{1}^{2}+n_{2}^{2}))^{s}(\varphi,f_{n_{1},n_{2}}^{i})_{L^{2}}^{2},$\\
                where $ (.,.)_{L^{2}}$ denote the scalar product in $L^{2}(\mathbb{T}^{2})$ and (.,.) is the duality product.\label{b3}
        \end{prp}
\begin{prp}
      Sobolev injection (see [\ref{yc}] page 22)\vspace*{0.1cm}\\
      1) If $s>k+1$ with    k $\in \mathbb{N}$   then $H^{s}(\mathbb{T}^{2}) \subset C^{k}(\mathbb{T}^{2})$  and  $ \forall \varphi \in H^{s}(\mathbb{T}^{2}), $ $\exists $ $ C(s)>0 $ such  that \\\hspace*{5cm}$ \lVert D^{\eta}\varphi \lVert_{\infty}\leq C(s) \lVert \varphi \lVert_{H^{s}}$, $\forall \lvert \eta \lvert\leq $ k. \\
      2) If s>1 then  $H^{s}(\mathbb{T}^{2})$ is a Banach algebra, ie
      $\exists C>0$; $\forall u $, $v \in H^{s}(\mathbb{T}^{2}) $, $uv \in H^{s}(\mathbb{T}^{2})$ and \hspace*{5cm} $\lVert uv \lVert_{H^{s}}\leq C \lVert u \lVert_{H^{s}} \lVert v \lVert_{H^{s}}.$\label{b4}
      \end{prp} 
  \begin{prp}
       Description of the Contraction Semigroup (see [\ref{gc}] and [\ref{oc}]). Let $\gamma>0$,\vspace*{0.1cm}\\
      1) The operator $  \gamma\bigtriangleup $ is  selfadjoint, unbounded on $ L^{2}(\mathbb{T}^{2})$ and it is   the infinitesimal generator of  \hspace*{0.4cm} a semigroup  $ \Upsilon(t)=e^{\gamma t\bigtriangleup} $, furthemore  $\forall \varphi \in L^{2}(\mathbb{T}^{2}),$\vspace*{-0.3cm}
        \[\tag{2.1}\Upsilon(t)\varphi =\underset{i,n_{1},n_{2}}{\sum} exp(-\lambda_{n_{1},n_{2}}t)(\varphi
      ,f_{n_{1},n_{2}}^{i})_{L^{2}}f_{n_{1},n_{2}}^{i}.\label{pp}\vspace*{-0.2cm}\]2)
       $\forall s<0$ (resp. s>0) $\Upsilon(t)$ and $\gamma\bigtriangleup $ have extension (resp. restriction ) to $H^{s}(\mathbb{T}^{2})$ such that $\Upsilon(t)$ is a  contraction semigroup, stronly-continuous, and bounded on $H^{s}(\mathbb{T}^{2})$ with  $\lvert \Upsilon(t) \lvert_{\mathcal{L}(H^{s}(\mathbb{T}^{2}))}\leq 1 $, where  $\mathcal{L}(H^{s}(\mathbb{T}^{2})) $ is the space of linear continuous   operator on $H^{s}(\mathbb{T}^{2})$.\label{b5}
      \end{prp}
      \begin{lem}(see [\ref{oc}])
       \hspace*{0.2cm}$\forall s>0$, $\forall \varphi \in H^{s}(\mathbb{T}^{2})$, resp. $\forall \varphi \in L^{\infty}(\mathbb{T}^{2})$,  from (\ref{pp}), we have \vspace*{0.12cm}\\\hspace*{4cm} $\lVert \Upsilon(t)\varphi \lVert_{H^{s}}\leq \lVert \varphi \lVert_{H^{s}}$ and $\lVert \Upsilon(t)\varphi \lVert_{\infty}\leq \lVert \varphi \lVert_{\infty}.$ \label{b6}
      \end{lem}  
      \begin{lem}(Lemma 2.5 of [\ref{ic}])\\
      Suppose that $(\mathbb{D},d) $ is a separable metric space and let $(\hat{P}_{n})_{n}$ denote a sequence of random probability measures on $(\mathbb{D},\mathcal{D})$ defined on a probability space $(\Omega, \mathcal{A},\mathbb{P})$. Then \\
      \hspace*{6cm} $\displaystyle \int_{\mathbb{D}}f(x)d\hat{P}_{n}(x)\xrightarrow{\mathbb{P}}\int_{\mathbb{D}}f(x)dP(x),$\\ for any $f$ bounded and  Lipschitz if and only if $d_{F}(\hat{P}_{n},P)\xrightarrow{\mathbb{P}}0.$ \label{b7}     
      \end{lem}
      Let  $T>0$. Let  $F$ be  a separable and reflexive Banach space included (with density and continuous injection) in a Hilbert space H. We identify H with its dual. So $F\subset H\subset F'$, where $F'$ is the dual of $F$.  Let $(A(t,.))_{\{t\in ]0,T[\}}$ be a family of linear  operators  from $F$ to $F'$, such that:\vspace*{0.25cm}\\
            \hspace*{1cm}(1) $\theta\longrightarrow (A(t,u+\theta v),w)$ is continuous from $\mathbb{R}$ to $\mathbb{R}$, $
            \forall u,v,w \in F $\vspace*{0.17cm}\\
            \hspace*{1cm}(2) $\exists$  $ \delta>0, \lVert A(t,u)\lVert_{F'}\leq \delta\lVert u\lVert_{F}, \hspace*{0.1cm}\forall u \in F$ \vspace*{0.18cm} \\
            \hspace*{1cm}(3) $\exists$  $ \sigma_{1}>0,$ $\sigma_{2}\in \mathbb{R},$ $(A(t,u),u)+\sigma_{2}\lVert u\lVert_{H}^{2}\geq \sigma_{1}\lVert u\lVert ^{2}_{F}$, $\hspace*{0.1cm} \forall u \in F$\vspace*{0.16cm}\\
           \hspace*{1cm}(4) $\forall u\in F,$ $t\longrightarrow A(t,u)$ est Lebesgue-mesurable with values in $F'$,\vspace*{0.1cm}\\
            where $(.,.)$ is a duality product between $F'$ and $F$.
            \begin{prp} (see Theorem 1.1 page 81 in [\ref{tc}])\\
            Let $(\Omega,\mathcal{F}_{0},\mathbb{P})$ be a probability space, let  $(A(t,.))_{\{t\in ]0,T[\}}$ be a family of  operators    from $F$ to $F'$ which  satisfy (1), (2), (3) and (4).  For $u(0)\in L^{2}(\Omega,\mathcal{F}_{0},\mathbb{P},H),$  $f=f_{1}+f_{2}$ with\vspace*{0.13cm} \\\hspace*{2.5cm} $-$ $ f_{1}\in L^{2}(\Omega,L^{1}(0,T,H))$, non-anticipative, \vspace*{0.12cm}\\\hspace*{2.5cm}  $-$ $f_{2} \in  L^{2}(\Omega\times ]0,T[, F),$ non-anticipative, \\ and  $(M_{t})_{0\leq t \leq T}$ a continuous square-integrable  martingale with values in H, the equation \vspace*{0.1cm}\\\hspace*{2.6cm} $du(t)+A(u(t))dt=f(t)dt+dM(t),$ $t\in [0,T]$ with $u(0)=u_{0},$\vspace*{0.15cm}\\ admits a unique solution  $u\in  L^{2}(\Omega\times ]0,T[, F )\cap L^{2}(\Omega, C([0,T],H)).\label{b8} 
            $\vspace*{-0.17cm}\end{prp}

 \section{Law of Large Numbers  of   Initial Measures \label{se8} }
 Recall that  
  $\mu_{0}^{S,N}=\frac{1}{N}\sum_{i=1}^{N}\{1_{A}(X^{i})(1-\xi_{i})+1_{A^{c}}(X^{i})\}\delta_{X^{i}}$,  
  $\mu_{0}^{I,N}=\frac{1}{N}\sum_{i=1}^{N}1_{A}(X^{i})\xi_{i}\delta_{X^{i}}$ and \\ $\mu_{0}^{N}=\mu_{0}^{S,N}+\mu_{0}^{I,N}=\frac{1}{N}\sum_{i=1}^{N}\delta_{X^{i}} $   where $\{X^{i},1\leq i\leq N \}$ is  an independent an identically distributed family of random variables, globally independent of $\{\xi_{i},1\leq i\leq N \}$, which in turn is a mutually independent family of Ber($p$).\vspace*{0.1cm}\\The following is assumed to be hold throughout this paper.
   \vspace*{0.17cm}\\\textbf{Assumption (H0):} The law $\nu$  of $X^{1}$ is abosuletly continuous with respect to the Lebesgue measure and its density $g$ satisfies:\\
   \hspace*{2cm} there exists $\delta_{1}>0$, $\delta_{2}>0$ such that $\delta_{1}\leq g(x)\leq \delta_{2},$ $\forall x \in \mathbb{T}^{2}.$
 \begin{thm}
The sequence $(\mu_{0}^{S,N},\mu_{0}^{I,N},\mu_{0}^{N})_{N\geq 1}$ converges a.s towards  $(\mu_{0}^{S},\mu_{0}^{I},\mu_{0})$\\ in $(\mathcal{M}_{F}(\mathbb{T}^{2}))^{3},$ where
$\mu_{0}^{S}(dx)=\{(1-p)1_{A}(x)+1_{A^{c}}(x)\}\nu(dx),$ $ \mu_{0}^{I}(dx)=p1_{A}(x)\nu(dx)$ and $ \mu_{0}(dx)=\nu(dx)$.\label{kl}
 \end{thm}
 \begin{proof} 
  All we need is to prove that for any  $\varphi,\psi,\phi\in C(\mathbb{T}^{2})$  the sequence \\$ \left((\mu_{0}^{S,N},\varphi),(\mu_{0}^{I,N},\psi),(\mu_{0}^{N},\phi)\right)_{N\geq 1}$ converges a.s towards  $\left((\mu_{0}^{S},\varphi),(\mu_{0}^{I},\psi),(\mu_{0},\phi)\right).$\\
 Let $\varphi,\psi,\phi\in C(\mathbb{T}_{2}),$ we have\\
 \hspace*{4cm}$(\mu_{0}^{S,N},\varphi)=\frac{1}{N}\sum_{i=1}^{N}[1_{A}(X^{i})(1-\xi_{i})+1_{A^{c}}(X^{i})]\varphi({X^{i}})$\\
 \hspace*{4cm}$(\mu_{0}^{I,N},\psi)=\frac{1}{N}\sum_{i=1}^{N}1_{A}(X^{i})\xi_{i}\psi({X^{i}})$\vspace*{0.1cm}\\
 \hspace*{4.2cm}$(\mu_{0}^{N},\phi)=\frac{1}{N}\sum_{i=1}^{N}\phi({X^{i}}).$\vspace*{0.12cm}\\
 Furthemore, according to the law of large numbers\vspace*{0.1cm}\\
 $(\mu_{0}^{S,N},\varphi)\xrightarrow{a.s}\mathbb{E}((1_{A}(X^{1})(1-\xi_{1})+1_{A^{c}}(X^{1}))\varphi({X^{1}}))=\displaystyle\int_{\mathbb{T}^{2}}\{(1-p)1_{A}(x)+1_{A^{c}}(x)\}\varphi (x)\nu(dx),$\\
  $(\mu_{0}^{I,N},\psi)\xrightarrow{a.s}\mathbb{E}(1_{A}(X^{1})\xi_{1}\psi({X^{1}}))=p\displaystyle\int_{A}\psi (x)\nu(dx),$\\
$(\mu_{0}^{N},\phi)\xrightarrow{a.s}\mathbb{E}(\phi({X^{1}}))=\displaystyle\int_{\mathbb{T}^{2}}\phi (x)\nu(dx).$\vspace*{0.1cm}\\
 Thus\vspace*{0.1cm}\\
 $\left((\mu_{0}^{S,N},\varphi),(\mu_{0}^{I,N},\psi),(\mu_{0}^{N},\phi)\right)\xrightarrow{a.s}\Big((1-p)\int_{A}\varphi(x)\nu(dx)+\int_{A^{c}}\varphi(x)\nu(dx),\\\hspace*{10cm}p\int_{A}\psi(x)\nu(dx),\hspace*{0.2cm}\int_{\mathbb{T}^{2}}\phi(x)\nu(dx)\Big). $
  \end{proof}
 \section{Central Limit Theorem of Initial Measures \label{se10}}
 We define $ U_{0}^{N}=\sqrt{N}(\mu_{0}^{S,N}-\mu_{0}^{S})$ , $V_{0}^{N}=\sqrt{N}(\mu_{0}^{I,N}-\mu_{0}^{I}))$  and $Z_{0}^{N}=\sqrt{N}(\mu_{0}^{N}-\mu_{0})$. \\ In this section we  study the  convergence of the sequence  $ (U_{0}^{N},V_{0}^{N},Z_{0}^{N}) $ in $( H^{-s}(\mathbb{T}^{2}))^{3},\\ $  as $ N\rightarrow \infty $,   with s>1.
 \begin{prp}
 For any s>1, there exists   $C_{1},C_{2},C_{3}>0 $ such that \vspace*{0.2cm}\\\hspace*{1cm} $\underset{N\geq 1}{\sup}\mathbb{E}(\lVert Z_{0}^{N} \lVert_{H^{-s}}^{2})<C_{1}; \quad  \underset{N\geq 1}{\sup}\mathbb{E}(\lVert U_{0}^{N} \lVert_{H^{-s}}^{2})<C_{2} $ and $\underset{N\geq 1}{\sup}\mathbb{E}(\lVert V_{0}^{N} \lVert_{H^{-s}}^{2})<C_{3}. $\label{uo}
 \end{prp}
 \begin{proof}
We only prove that  $\underset{N\geq 1}{\sup}\mathbb{E}(\lVert V_{0}^{N} \lVert_{H^{-s}}^{2})<C_{3}$. The other estimates follow by a similar argument. Since $1_{A}(X_{j})\xi_{j}\delta_{X_{j}}$ are i.i.d with law $\mu_{0}^{I}$, from assumption (H0) and Lemma \ref{ap} in the Appendix below,
  we have  \vspace*{0.05cm}\\\hspace*{1.5cm} $\mathbb{E}(\lVert V_{0}^{N} \lVert_{H^{-s}}^{2} )=\mathbb{E}(\sum\limits_{i,n_{1},n_{2}}(V_{0}^{N},\rho^{i,s}_{n_{1},n_{2}})^{2})\\\hspace*{3.8cm}=N\sum\limits_{i,n_{1},n_{2}}\mathbb{E}\left(\left((\mu_{0}^{I,N},\rho^{i,s}_{n_{1},n_{2}})-(\mu_{0}^{I},\rho^{i,s}_{n_{1},n_{2}})\right)^{2}\right)\\\hspace*{3.8cm}=\frac{1}{N}\sum\limits_{i,n_{1},n_{2}}\mathbb{E}\left(\left[\sum\limits_{j=1}^{N}[1_{A}(X_{j})\xi_{j}\rho^{i,s}_{n_{1},n_{2}}(X_{j})-(\mu_{0}^{I},\rho^{i,s}_{n_{1},n_{2}})]\right]^{2}\right)\\\hspace*{3.8cm}=\frac{1}{N}\sum\limits_{i,n_{1},n_{2}}\sum\limits_{j=1}^{N}\mathbb{E}\left(\left[1_{A}(X_{j})\xi_{j}\rho^{i,s}_{n_{1},n_{2}}(X_{j})-(\mu_{0}^{I},\rho^{i,s}_{n_{1},n_{2}})\right]^{2}\right)\\\hspace*{3.8cm}\leq\frac{1}{N}\sum\limits_{i,n_{1},n_{2}}\sum\limits_{j=1}^{N}\mathbb{E}\left([1_{A}(X_{j})\xi_{j}\rho^{i,s}_{n_{1},n_{2}}(X_{j})]^{2}\right)\\\hspace*{3.8cm}\leq p\displaystyle\int_{A}\sum\limits_{i,n_{1},n_{2}}(\rho^{i,s}_{n_{1},n_{2}})^{2}(x)g(x)dx\leq p\delta_{2} C $ \quad  if s>1.  
  \end{proof}Let us give now the main result of this section.
 \begin{thm}  For any s>1, the sequence 
  $(U_{0}^{N},V_{0}^{N},Z_{0}^{N})_{N\geq 1} $ converges in law  in $(H^{-s}(\mathbb{T}^{2}))^{3}$ towards    $(U_{0},V_{0},Z_{0})$ where $\forall \varphi,\psi,\phi \in H^{s}(\mathbb{T}^{2})$, $\left((U_{0},\varphi),(V_{0},\psi),(Z_{0},\phi)\right)$ is a Gaussian vector which satisfies:\vspace*{0.03cm}\\
 $\displaystyle(U_{0},\varphi)=W_{1}[\varphi\sqrt{g}\{(1-p)1_{A}+1_{A^{c}}\}]-(1-p)W_{1}(\sqrt{g})\int_{A}\varphi(x)g(x)dx-W_{1}(\sqrt{g})\int_{A^{c}}\varphi(x)g(x)dx$\vspace*{-0.4cm}\[\hspace*{-9.5cm}\tag{4.1}+W_{2}(1_{A}\varphi\sqrt{(p-p^{2})g}),\label{u1}\]\vspace*{-0.5cm}
   \[\hspace*{-3.8cm}\tag{4.2}\displaystyle(V_{0},\psi)=pW_{1}(1_{A}\psi\sqrt{g})-pW_{1}(\sqrt{g})\int_{A}\psi(x)g(x)dx-W_{2}(1_{A}\psi\sqrt{(p-p^{2})g}),\label{u2}\]\vspace*{-0.34cm}
    \[\hspace*{-8cm}\tag{4.3}\displaystyle(Z_{0},\phi)=W_{1}(\phi\sqrt{g})-W_{1}(\sqrt{g})\left(\int_{\mathbb{T}^{2}}\phi(x)g(x)dx\right),\label{u3}\] 
   where   $W_{1},W_{2}$ are mutually independent  two dimentional white noises.\label{lo} 
 \end{thm}
 \subsection{\textit{Proof of Theorem \ref{lo}} }
 We first prove the tightness of the sequence $(U_{0}^{N},V_{0}^{N},Z_{0}^{N})_{N\geq 1}$,  then  identify the limit.
 \subsubsection{Tightness of $(U_{0}^{N},V_{0}^{N},Z_{0}^{N})_{N\geq 1}$} 
 \begin{prp}
 For any s>1, the sequences $(U_{0}^{N})_{N\geq 1}$; $(V_{0}^{N})_{N\geq 1}$ and  $(Z_{0}^{N})_{N\geq 1}$ are tight in $ H^{-s}.$ \label{uoo}
 \end{prp}
 \begin{proof}
 The tightness of $(U_{0}^{N})_{N\geq 1}$ in $H^{-s}$ follows readily from the fact that \vspace*{0.12cm}\\\hspace*{6cm}$\underset{N\geq 1}{\sup}\mathbb{E}(\lVert U_{0}^{N}\lVert_{H^{-s}}^{2}) \leq C_{1}$. \\ Indeed, since $\forall \hspace*{0.1cm} 1<s'<s$, the embedding $H^{-s'}(\mathbb{T}^{2})\hookrightarrow H^{-s}(\mathbb{T}^{2})$ is compact (see Theorem 1.69 page 47 of [\ref{ec}]) then  $B_{H^{-s'}}=\{ \mu \in H^{-s'} ; \lVert \mu \lVert_{H^{-s'}}\leq R\}$ is a compact subset of $H^{-s}$. \vspace*{0.1cm}\\Thus  $\mathbb{P}( U_{0}^{N}\notin B_{H^{-s'}})=\mathbb{P}(\lVert U_{0}^{N}\lVert_{H^{-s'}}> R)\vspace*{0.1cm}\\\hspace*{3.8cm}\leq \frac{1}{R^{2}}\mathbb{E}(\lVert U_{0}^{N}\lVert_{H^{-s'}}^{2})\leq \frac{C_{1}}{R^{2}}$.\vspace*{0.13cm}\\So by  choosing R large enough we get the result.\vspace*{0.12cm}\\
 The tightness of $ (V_{0}^{N})_{N\geq 1}$ and $ (Z_{0}^{N})_{N\geq 1}$ are obtained  by similar arguments.
     \end{proof}
       From Proposition \ref{uoo} we deduce that the   sequence $(U_{0}^{N},V_{0}^{N},Z_{0}^{N})_{N \geq 1} $  is  tight in $(H^{-s})^{3}$, thus  by Prokhorov's theorem  there exists  a subsequence still denoted   $(U_{0}^{N},V_{0}^{N},Z_{0}^{N})_{N\geq 1} $  which converges in law   towards $(U_{0},V_{0},Z_{0})$ in $(H^{-s})^{3}$.
  \subsubsection{Gaussian caracter and expressions of ($(U_{0},\varphi)$,$ (V_{0},\psi),(Z_{0},\phi)$)  using white noises}
  Let $\varphi, \psi,\phi  \in H^{s}$, let us first compute the values of Var($(U_{0},\varphi)$);  Var($(V_{0},\psi)$);  Var($(Z_{0},\phi)$);  Cov($(U_{0},\varphi),(V_{0},\psi)$);  Cov($(U_{0},\varphi),(Z_{0},\phi)$); and  Cov($(V_{0},\psi),(Z_{0},\phi)$).
\vspace*{0.1cm}\\$-$ Computation of Var($(U_{0},\varphi)$). We have\vspace*{0.2cm}\\   $(U_{0}^{N},\varphi)=\sqrt{N}[(\mu_{0}^{S,N},\varphi)-(\mu_{0}^{S},\varphi)]\\
 \hspace*{1.4cm}=\sqrt{N}[\frac{1}{N}\sum\limits_{i=1}^{N}\{1_{A}(X^{i})(1-\xi_{i})+1_{A^{c}}(X^{1})\}\varphi({X^{i}})-(1-p)\int_{A}\varphi(x)dx-\int_{A^{c}}\varphi(x)dx],$\vspace*{0.15cm}\\
 and \\Var$[\{1_{A}(X^{1})(1-\xi_{1})+1_{A^{c}}(X^{1})\}\varphi({X_{1}})]$=\vspace*{0.15cm}\\
   \hspace*{0.14cm}=Var$[1_{A}(X^{1})(1-\xi_{1})\varphi({X^{1}})]+$Var$(1_{A^{c}}(X^{1})\varphi({X^{1}}))+$2Cov$[1_{A}(X^{1})(1-\xi_{1})\varphi({X^{1}}),1_{A^{c}}(X^{1})\varphi({X^{1}})]$\vspace*{0.16cm}\\
   \hspace*{0.1cm}=Var$[1_{A}(X^{1})(1-\xi_{1})\varphi({X^{1}})]+$Var$(1_{A^{c}}(X^{1})\varphi({X^{1}}))-2\mathbb{E}[1_{A}(X^{1})(1-\xi_{1})\varphi({X^{1}})]\mathbb{E}(1_{A^{c}}(X^{1})\varphi({X^{1}}))$\vspace*{0.15cm}\\\\
   \hspace*{0.1cm}=$\displaystyle(1-p)\int_{A}\varphi^{2}(x)g(x)dx-(1-p)^{2}\Big(\int_{A}\varphi(x)g(x)dx\Big)^{2}+\int_{A^{c}}\varphi^{2}(x)g(x)dx-\Big(\int_{A^{c}}\varphi(x)g(x)dx\Big)^{2}\\\hspace*{0.2cm}-\displaystyle2(1-p)\int_{A}\varphi(x)g(x)dx\int_{A^{c}}\varphi(x)g(x)dx$\vspace*{0.03cm}\\
   \hspace*{0.1cm}= $\alpha_{p}.$\\ So according to the central limit theorem \vspace*{-0.3cm}
    \[\tag{4.4}( U_{0}^{N},\varphi)\xrightarrow{L}(U_{0},\varphi) \hspace*{0.13cm}\textrm{where}\hspace*{0.13cm} (U_{0},\varphi)\rightsquigarrow \mathcal{N}(0,\alpha_{p}).\label{a1}\]
    $-$ Computation of Var($(V_{0},\psi)$). One has   \vspace*{0.12cm}\\ $\hspace*{1.5cm}(V_{0}^{N},\psi)=\sqrt{N}((\mu_{0}^{I,N},\psi)-(\mu_{0}^{I},\psi))$  
    and  \\\hspace*{1.5cm}Var$[1_{A}(X^{1})\xi_{1}\psi({X^{1}})]=\displaystyle p\int_{A}\psi^{2}(x)g(x)dx-p^{2}\Big(\int_{A}\psi(x)g(x)dx\Big)^{2}=\beta_{p}.$\\
     So according to the central limit theorem\vspace*{-0.2cm}
      \[\tag{4.5}(V_{0}^{N},\psi)\xrightarrow{L}(V_{0},\psi) \hspace*{0.13cm}\textrm{where}\hspace*{0.13cm}  (V_{0},\psi)\rightsquigarrow \mathcal{N}(0,\beta_{p}). \label{a2}\] 
      $-$ Computation of Var($(Z_{0},\phi)$). 
      According to central limit theorem \vspace*{-0.2cm}\[\tag{4.6}(Z_{0}^{N},\phi)\xrightarrow{L}(Z_{0},\phi)\hspace*{0.13cm}\textrm{where}\hspace*{0.13cm} (Z_{0},\phi)\rightsquigarrow N\Big(0,\sigma^{2}=\displaystyle\int_{\mathbb{T}_{2}}\phi^{2}(x)g(x)dx-\Big(\int_{\mathbb{T}^{2}}\phi(x)g(x)dx\Big)^{2}\Big).\label{a3}\]
      $-$ Computation of Cov($(U_{0},\varphi),(V_{0},\psi)$), Cov($(U_{0},\varphi),(Z_{0},\phi)$) and  Cov($(V_{0},\psi),(Z_{0},\phi)$).  \vspace*{0.25cm}\\ As 
       Cov($(U_{0},\varphi),(V_{0},\psi ))=\frac{1}{2}$[Var($(U_{0},\varphi)+(V_{0},\psi )$)-Var($(U_{0},\varphi))$-Var($(V_{0},\psi )$)], \vspace*{0.12cm}\\so we  need to compute Var($(U_{0},\varphi)+(V_{0},\psi )$), Var($(U_{0},\varphi)+(Z_{0},\phi )$) and Var($(V_{0},\psi)+(Z_{0},\phi )$).\\ We have\\
 $(U_{0}^{N},\varphi)+(V_{0}^{N},\psi )+(Z_{0}^{N},\phi )=\sqrt{N}\Big\{ \frac{1}{N}\sum\limits_{i=1}^{N}\Big[\{1_{A}(X^{i})(1-\xi_{i})+1_{A^{c}}(X^{i})\}\varphi({X^{i}})+1_{A}(X^{i})\xi_{i}\psi({X^{i}})\\\hspace*{1.5cm}+\phi(X^{i})\Big]-(1-p)\int_{A}\varphi(x)g(x)dx-\int_{A^{c}}\varphi(x)g(x)dx-p\int_{A}\psi(x)g(x)dx-\int_{\mathbb{T}^{2}}\phi(x)g(x)dx\Big\},$  \\ and \\
  Var[$\{1_{A}(X^{1})(1-\xi_{1})+1_{A^{c}}(X^{1})\}\varphi({X^{1}})+1_{A}(X^{1})\xi_{1}\psi({X^{1}})+\phi(X^{1})$]\vspace*{0.1cm} \\\hspace*{2cm}= $\alpha_{p}+ \beta_{p}+\sigma^{2}$+2Cov[$\{1_{A}(X^{1})(1-\xi_{1})+1_{A^{c}}(X^{1})\}\varphi({X^{1}}), 1_{A}(X^{1})\xi_{1}\psi({X^{1}})$]\vspace*{0.1cm}\\\hspace*{2.3cm}+ 2Cov[$\{1_{A}(X^{1})(1-\xi_{1})+1_{A^{c}}(X^{1})\}\varphi({X^{1}}), \phi(X^{1})$]+ 2Cov[$ 1_{A}(X^{1})\xi_{1}\psi({X^{1}}),\phi(X^{1})$].\vspace*{0.2cm}\\
  Furthemore since $ \{1_{A}(X^{1})(1-\xi_{1})+1_{A^{c}}(X^{1})\}\varphi({X^{1}}) 1_{A}(X^{1})\xi_{1}\psi({X^{1}})$)=0 a.s, \vspace*{0.1cm}\\Cov[$\{1_{A}(X^{1})(1-\xi_{1})+1_{A^{c}}(X^{1})\}\varphi({X^{1}}), 1_{A}(X^{1})\xi_{1}\psi({X^{1}})$]\vspace*{0.05cm}\\\hspace*{3cm}= $-\mathbb{E}[\{1_{A}(X^{1})(1-\xi_{1})+1_{A^{c}}(X^{1})\}\varphi({X^{1}})]\mathbb{E}(1_{A}(X^{1})\xi_{1}\psi({X^{1}}))$\vspace*{0.05cm}\\\hspace*{2.98cm}= $\displaystyle- p(1-p)\int_{A}\varphi(x)g(x)dx\int_{A}\psi(x)g(x)dx-p\int_{A^{c}}\varphi(x)g(x)dx\int_{A}\psi(x)g(x)dx$\\\hspace*{3cm}= $ \gamma_{p}.$ \hspace*{0.2cm} \\On the other hand:\vspace*{0.06cm}\\
 Cov[$\{1_{A}(X^{1})(1-\xi_{1})+1_{A^{c}}(X^{1})\}\varphi({X^{1}}), \phi(X^{1})$] $ \\\hspace*{2cm}=\displaystyle(1-p)\int_{A}\varphi(x)\phi(x)g(x)dx+(p-1)\int_{A}\varphi(x)g(x)dx\int_{\mathbb{T}^{2}}\phi(x)g(x)dx\\\hspace*{2.3cm}+\int_{A^{c}}\varphi(x)\phi(x)g(x)dx-\int_{A^{c}}\varphi(x)g(x)dx\int_{\mathbb{T}^{2}}\phi(x)g(x)dx$ 
   \\\hspace*{2cm}= $\eta_{p},$ \vspace*{0.07cm} \\ and  \\
    Cov[$ 1_{A}(X^{1})\xi_{1}\psi({X^{1}}),\phi(X^{1})$]$ = \displaystyle p\int_{A}\psi(x)\phi(x)g(x)dx -p\int_{A}\psi(x)g(x)dx\int_{\mathbb{T}^{2}}\phi(x)g(x)dx
    =\lambda_{p}.$ \vspace*{0.2cm}\\
  Thus according to the central limit  theorem \vspace*{-0.3cm}
\[\tag{4.7}(U_{0}^{N},\varphi)+(V_{0}^{N},\psi )+(Z_{0}^{N},\phi )\rightsquigarrow \mathcal{N}(0,\alpha_{p}+ \beta_{p}+\sigma^{2}+2(\gamma_{p}+\lambda_{p}+\eta_{p})).\label{a4}\]
 Taking $\phi\equiv 0$; $\psi\equiv 0$ and  $\varphi\equiv 0$  respectively in (\ref{a4}), we obtain\vspace*{0.08cm}\\
 \hspace*{4.1cm}Var($(U_{0},\varphi)+(V_{0},\psi )$) = $\alpha_{p}+ \beta_{p} +2\gamma_{p},$\vspace*{0.08cm} \\
 \hspace*{4cm} Var($(U_{0},\varphi)+(Z_{0},\phi )$) =$\alpha_{p}+\sigma^{2}+2\eta_{p},$\vspace*{0.08cm}\\
   \hspace*{4cm} Var($(V_{0},\psi)+(Z_{0},\phi )$) =$\beta_{p}+\sigma^{2}+2\lambda_{p}.$ \\ So we deduce that\vspace*{0.2cm}\\\hspace*{0.8cm}
 Cov($(U_{0},\varphi),(V_{0},\psi)$)=$\gamma_{p}$; Cov($(U_{0},\varphi),(Z_{0},\phi))=\eta_{p}$ and  Cov($(V_{0},\psi),(Z_{0},\phi))=\lambda_{p}.$\hspace*{0.4cm}(4.8)\vspace*{0.18cm}\\
   Hence from (\ref{a1}), (\ref{a2}), (\ref{a3}), (\ref{a4}) and (4.8), we  conclude that for any  $ \varphi$,$\psi$, $\phi$ $\in H^{s}$,  \\ $((U_{0},\varphi),(V_{0},\psi),(Z_{0},\phi))$ is a Gaussian vector with the same law as the vector given by (\ref{u1}), (\ref{u2}) and (\ref{u3}).
 \subsubsection{Conclusion}   
The sequence $(U_{0}^{N},V_{0}^{N},Z_{0}^{N})_{N\geq 1} $  being  tight  so by Prokhorov's theorem  there exists a subsequence still  denoted   $(U_{0}^{N},V_{0}^{N},Z_{0}^{N})_{N\geq 1} $ which converge in law in $(H^{-s}(\mathbb{T}^{2}))^{3}$  towards $(U_{0},V_{0},Z_{0})$. On the other hand  $\forall \varphi,\psi, \phi \in H^{s}$ $((U_{0},\varphi),(V_{0},\psi),(Z_{0},\phi))\rightsquigarrow \mathcal{N} \left((0,0,0),\begin{pmatrix} 
\alpha_{p} & \gamma_{p}&  \eta_{p} \\
\gamma_{p} & \beta_{p}& \lambda_{p} \\
\eta_{p} & \lambda_{p}& \sigma^{2}
\end{pmatrix}\right),$ \\ hence  we conclude that the whole sequence  $(U_{0}^{N},V_{0}^{N},Z_{0}^{N})_{N\geq 1} $ converge in law in $(H^{-s})^{3}$  towards $(U_{0},V_{0},Z_{0}) \vspace*{-0.2cm}$.
\section{Law of Large Numbers}
The aim of this section is to study the convergence  of $
(\mu^{S,N},\mu^{I,N})$  under  Assumption (H1) below, and the convergence of $\mu^{N}$ as $ N\rightarrow \infty. $ \vspace*{0.2cm}\\
To this end   we are going to:\vspace*{0.1cm}\\
$\hspace*{0.5cm} \bullet$ Write the system of evolution equations of $( \mu^{S,N},\mu^{I,N}).$\vspace*{0.1cm}\\
$\hspace*{0.5cm} \bullet$ Study the convergence of  $\{\mu_{t}^{N},t\geq 0\}$ in $C(\mathbb{R_{+}},\mathcal{M}_{F}(\mathbb{T}^{2})).$  \vspace*{0.2cm} \\
$ \hspace*{0.5cm}\bullet$ Study the tightness of $  (\mu^{S,N},\mu^{I,N})_{ N \geq 1}$ in Skorokhod's space $(D(\mathbb{R_{+}},\mathcal{M}_{F}(\mathbb{T}^{2})))^{2}. $\vspace*{0.2cm}\\
$ \hspace*{0.5cm}\bullet$  Show that all limit points $\mu^{S}$ and $\mu^{I }$ of $ (\mu^{S,N})_{N\geq 1}$ and $(\mu^{I,N})_{N\geq 1}$ are absolutely continuous \hspace*{0.7cm} with  respect to the Lebesgue measure with density $f_{S}$ and $f_{I}$ bounded by $\delta_{2}$ ($\delta_{2}$ is defined \\\hspace*{0.8cm}in section \ref{se8}).\vspace*{0.2cm} \\
$ \hspace*{0.5cm} \bullet$ Show that the system of   PDEs verified by the pair  $(f_{S}, f_{I})$  admits a unique   solution in \\\hspace*{0.7cm} $\Lambda=\{(f_{1},f_{2})/0\leq f_{i}\leq \delta_{2}, i\in \{1,2\}\}.$\vspace*{0.2cm} \\The following is assumed to hold throughout section 5.\vspace*{0.15cm}\\
\textbf{Assumption (H1):
\hspace*{0.4cm} k is Lipschitz, with the Lipschitz constant $C_{k}.$} 
\subsection{ System of evolution equations  of  $ \{ (\mu_{t}^{S,N},\mu_{t}^{I,N}) $, $ t\geq 0  \}$ \label{se5}}
\subsubsection{Evolution equation  of $\{ \mu_{t}^{S,N},t\geq0\}$} 
Let   $ \{ M^{i}, 1\leq i \leq N \} $  be a family of mutualy  independant  standard (i.e with mean measure the Lebesgue measure ) Poisson Random  Measures (in short PRMs)  on $ \mathbb{R}^{2}_{+} $  which are  globally independent of $ \{ {X_{t}^{i}, 0\leq t, 1\leq i \leq N}\} $. We note by  $ \{\overline{M}^{i},  1\leq i \leq N \} $  the family of compensated  PRMs.  We recall that $\varPi$ is the canonical projection from $\mathbb{R}^{2}$ to $\mathbb{T}^{2}, \widetilde{X}_{t}^{i}=X^{i}+\sqrt{2\gamma}B^{i}_{t}$ and $X_{t}^{i}=\varPi(\widetilde{X}_{t}^{i})$. Now 
Let $\varphi \in C^{2}(\mathbb{T}^{2}),$ if we let $\widetilde{\varphi}=\varphi\circ\varPi,$
         according to the Itô formula, we have \\\hspace*{2cm} $\widetilde{\varphi}(\widetilde{X}_{t}^{i})=\widetilde{\varphi}(X^{i})+\displaystyle\sqrt{2\gamma}\int_{0}^{t}\bigtriangledown\widetilde{\varphi}(\widetilde{X}_{r}^{i})dB_{r}^{i}+\gamma\int_{0}^{t}\bigtriangleup\widetilde{\varphi}(\widetilde{X}_{r}^{i})dr, $\\hence \hspace*{1cm} $\varphi(X_{t}^{i})=\varphi(X^{i})+\displaystyle\sqrt{2\gamma}\int_{0}^{t}\bigtriangledown\varphi(X_{r}^{i})dB_{r}^{i}+\gamma\int_{0}^{t}\bigtriangleup\varphi(X_{r}^{i})dr.$ \\Thus $\{1_{\{E_{t}^{i}=S\}}\varphi(X_{t}^{i}),t\geq0 \} $ is a jump process satisfying \\
         $1_{\{E_{t}^{i}=S\}}\varphi({X_{t}^{i}})=\displaystyle 1_{\{E_{0}^{i}=S\}}\varphi({X^{i}})+\sqrt{2\gamma}\int_{0}^{t}1_{\{E_{r}^{i}=S\}}\bigtriangledown\varphi({X_{r}^{i}})dB_{r}^{i}+\gamma \int_{0}^{t}1_{\{E_{r}^{i}=S\}}\bigtriangleup\varphi(X_{r}^{i})dr\\\hspace*{2.3cm}-
         \int_{0}^{t}\int_{0}^{\infty}1_{\{u\leq\beta\sum_{j=1}^{N}\frac{K(X_{r}^{i},X_{r}^{j})}{\sum\limits_{l=1}^{N}K(X_{r}^{l},X_{r}^{j})}1_{\{E_{r}^{j}=I\}}\} } 1_{\{E_{r^{-}}^{i}=S\}}\varphi({X_{r}^{i}})M^{i}(du,dr).$\\
        Taking the sum over i and multiplying by $\frac{1}{N}$ we obtain \\
         $\displaystyle\frac{1}{N}\sum_{i=1}^{N}1_{\{E_{t}^{i}=S\}}\varphi({X_{t}^{i}})=\displaystyle\frac{1}{N}\sum_{i=1}^{N}1_{\{E_{0}^{i}=S\}}\varphi({X^{i}})+\frac{\sqrt{2\gamma}}{N}\sum_{i=1}^{N}\int_{0}^{t}1_{\{E_{r}^{i}=S\}}\bigtriangledown\varphi({X_{r}^{i}})dB_{r}^{i}\\\hspace*{4cm}+\frac{\gamma}{N}\sum_{i=1}^{N}\int_{0}^{t}1_{\{E_{r}^{i}=S\}}\bigtriangleup\varphi(X_{r}^{i})dr\\\hspace*{4cm}-\frac{1}{N}\sum_{i=1}^{N}\int_{0}^{t}\int_{0}^{\infty}1_{\{u\leq \beta\sum_{j=1}^{N}\frac{K(X_{r}^{i},X_{r}^{j})}{\sum\limits_{l=1}^{N}K(X_{r}^{l},X_{r}^{j})}1_{\{E_{r}^{j}=I\}}\}} 1_{\{E_{r^{-}}^{i}=S\}}\varphi({X_{r}^{i}})\overline{M}^{i}(du,dr)\\\hspace*{4cm}-\frac{1}{N}\sum_{i=1}^{N}\int_{0}^{t} \frac{\beta}{N}\sum_{j=1}^{N}\frac{K(X_{r}^{i},X_{r}^{j})}{\frac{1}{N}\sum\limits_{l=1}^{N}K(X_{r}^{l},X_{r}^{j})}1_{\{E_{r}^{j}=I\}} 1_{\{E_{r}^{i}=S\}}\varphi({X_{r}^{i}})dr.$\\
      So \vspace*{-0.3cm} \[\tag{5.1}
         \displaystyle(\mu_{t}^{S,N},\varphi)=(\mu_{0}^{S,N},\varphi)+\gamma\int_{0}^{t}(\mu_{r}^{S,N},\bigtriangleup \varphi) dr - \beta \int_{0}^{t}\left(\mu_{r}^{S,N}, \varphi (\mu_{r}^{I,N},\frac{K}{(\mu^{N}_{r},K)})\right) dr + M_{t}^{N,\varphi},\label{e1}\]
       
        where \\ \vspace*{0.3cm}
        $\left(\mu_{r}^{S,N}, \varphi ( \mu_{r}^{I,N},\frac{K}{(\mu^{N}_{r},K)})\right)=\displaystyle\int_{\mathbb{T}_{2}} \varphi(x) \int_{\mathbb{T}^{2}} \frac{K(x,y)}{\int_{\mathbb{T}_{2}}K(y,z)\mu_{r}^{N}(dz)} \mu_{r}^{I,N}(dy) \mu_{r}^{S,N}(dx) $\\and \\
        $ M_{t}^{N,\varphi}= - \displaystyle\frac{1}{N} \sum_{i=1}^{N} \int_{0}^{t} \int_{0}^{\infty}1_{\{E_{r^{-}}^{i}=S\}}\varphi(X_{r}^{i})1_{\{u\leq \beta \sum\limits_{j=1}^{N}\frac{K(X_{r}^{i},X_{r}^{j})}{\sum\limits_{l=1}^{N}K(X_{r}^{l},X_{r}^{j}))} 1_{\{E_{r}^{j}=I\}}\}} \overline{M}^{i}(dr,du)  \\\hspace*{1.5cm}+\frac{\sqrt{2\gamma}}{N} \sum\limits_{i=1}^{N} \int_{0}^{t}1_{\{E_{r}^{i}=S\}}\bigtriangledown\varphi(X_{r}^{i})dB_{r}^{i}.$ 
         \subsubsection{Evolution equation of  $ \{\mu_{t}^{I,N}, t\geq0 \}$}
         Let $ \{ Q^{i}, 1\leq i \leq N \} $ be a family of  mutually independent  standard Poisson Random  Measures (abreviated below as PRMs) on  $ \mathbb{R}^{2}_{+} $ which are  globally independent of $ \{ {X_{t}^{i}, 0\leq t , 1\leq i \leq N}\} $ and  $ \{ M^{i}, 1\leq i \leq N \} $. We note by   $ \{\overline{Q}^{i},  1\leq i \leq N \} $  the family of compensated PRMs.\\
         Let $\varphi \in C^{2}(\mathbb{T}^{2}),$ 
          $\{1_{\{E_{t}^{i}=I\}}\varphi(X_{t}^{i}),t\geq0 \} $ is a jump process satisfying \\
         $1_{\{E_{t}^{i}=I\}}\varphi({X_{t}^{i}})=\displaystyle 1_{\{E_{0}^{i}=I\}}\varphi({X^{i}})+\sqrt{2\gamma}\int_{0}^{t}1_{\{E_{r}^{i}=I\}}\bigtriangledown\varphi({X_{r}^{i}})dB_{r}^{i}+\gamma\int_{0}^{t}1_{\{E_{r}^{i}=I\}}\bigtriangleup\varphi(X_{r}^{i})dr\\\hspace*{2.3cm}+
         \int_{0}^{t}\int_{0}^{\infty}1_{\{u\leq \beta\sum_{j=1}^{N}\frac{K(X_{r}^{i},X_{r}^{j})}{\sum_{l=1}^{N}K(X_{r}^{l},X_{r}^{j})}1_{\{E_{r}^{j}=I\}}} 1_{\{E_{r^{-}}^{i}=S\}}\varphi({X_{r}^{i}})M^{i}(du,dr)\\\hspace*{2.3cm}-
                  \int_{0}^{t}\int_{0}^{\alpha} 1_{\{E_{r^{-}}^{i}=I\}}\varphi({X_{r}^{i}})Q^{i}(du,dr).$\vspace*{0.17cm}\\ 
          Summing over i  and multiplying by $\frac{1}{N}$ we obtain \\
         $\frac{1}{N}\sum_{i=1}^{N}1_{\{E_{t}^{i}=I\}}\varphi({X_{t}^{i}})=
         \displaystyle\frac{1}{N}\sum_{i=1}^{N}1_{\{E_{0}^{i}=I\}}\varphi({X^{i}})+\frac{\sqrt{2\gamma}}{N}\sum_{i=1}^{N}\int_{0}^{t}1_{\{E_{r}^{i}=I\}}\bigtriangledown\varphi({X_{r}^{i}})dB_{r}^{i}\\\hspace*{4.1cm}+\frac{\gamma}{N}\sum_{i=1}^{N}\int_{0}^{t}1_{\{E_{r}^{i}=I\}}\bigtriangleup\varphi(X_{r}^{i})dr\\\hspace*{4.1cm}+\frac{1}{N}\sum\limits_{i=1}^{N}\int_{0}^{t}\int_{0}^{\infty}1_{\{u\leq \beta\sum_{j=1}^{N}\frac{K(X_{r}^{i},X_{r}^{j})}{\sum_{l=1}^{N}K(X_{r}^{l},X_{r}^{j})}1_{\{E_{r}^{j}=I\}}} 1_{\{E_{r^{-}}^{i}=S\}}\varphi({X_{r}^{i}})\overline{M}^{i}(du,dr)\\\hspace*{4.1cm}+\frac{1}{N}\sum_{i=1}^{N}\int_{0}^{t}\int_{0}^{\infty}1_{\{u\leq \beta\sum_{j=1}^{N}\frac{K(X_{r}^{i},X_{r}^{j})}{\sum_{l=1}^{N}K(X_{r}^{l},X_{r}^{j})}1_{\{E_{r}^{j}=I\}}} 1_{\{E_{r}^{i}=S\}}\varphi({X_{r}^{i}})drdu\\\hspace*{2.1cm}- \frac{1}{N}\sum_{i=1}^{N}\int_{0}^{t}\int_{0}^{\alpha} 1_{\{E_{r^{-}}^{i}=I\}}\varphi({X_{r}^{i}})\overline{Q}^{i}(du,dr)-  \frac{\alpha}{N}\sum_{i=1}^{N}\int_{0}^{t} 1_{\{E_{r}^{i}=I\}}\varphi({X_{r}^{i}})dr.$\\
         So\\
         $
         (\mu_{t}^{I,N} \varphi)=\displaystyle (\mu_{0}^{I,N},\varphi) + \gamma\int_{0}^{t}(\mu_{r}^{I,N},\bigtriangleup \varphi) dr+ \beta\int_{0}^{t}(\mu_{r}^{S,N}, \varphi (\mu_{r}^{I,N},\frac{K}{(\mu^{N}_{r},K)})) dr-\alpha \int_{0}^{t}(\mu_{r}^{I,N},\varphi)dr$ \vspace*{-0.3cm}\[\hspace*{-12cm}\tag{5.2} +L_{t}^{N,\varphi},\label{e2}
         \vspace*{-0.3cm} \]
         where \\
          $ L_{t}^{N,\varphi} = \displaystyle\frac{1}{N} \sum_{i=1}^{N} \int_{0}^{t} \int_{0}^{\infty}1_{\{E_{r^{-}}^{i}=S\}}\varphi(X_{r}^{i})1_{\{u\leq \beta \sum_{j=1}^{N}\frac{K(X_{r}^{i},X_{r}^{j})}{\sum\limits_{l=1}^{N}K(X_{r}^{l},X_{r}^{j})} 1_{\{E_{r}^{j}=I\}}\}} \overline{M}^{i}(dr,du))  \\ \hspace*{1cm}+\frac{\sqrt{2\gamma}}{N} \sum_{i=1}^{N} \int_{0}^{t}1_{\{E_{r}^{i}=I\}}\bigtriangledown\varphi(X_{r}^{i})dB_{r}^{i} $ $   -\displaystyle \frac{1}{N} \sum_{i=1}^{N} \int_{0}^{t} \int_{0}^{\alpha}1_{\{E_{r^{-}}^{i}=I\}}\varphi(X_{r}^{i})\overline{Q}^{i}(dr,du). $
   \subsection{ Convergence of $ \{\mu_{t}^{N},t\geq 0 \}_{N\geq 1} \hspace*{0.1cm} $ in $ \hspace*{0.1cm} C(\mathbb{R}_{+},\mathcal{M}_{F}(\mathbb{T}^{2})) $ }
    Recall that  we equip $\mathcal{M}_{F}(\mathbb{T}^{2})$ with the topology of weak convergence and the space of continuous functions from $\mathbb{R}_{+}$ to $\mathcal{M}_{F}(\mathbb{T}^{2}),$ denoted  $C(\mathbb{R}_{+},\mathcal{M}_{F}(\mathbb{T}^{2}))$ with the uniform topology.\\
         It follows from  the Itô  formula that the  processes $ \{\mu_{t}^{N},t\geq 0\}$ satisfies \\
                  $ (\mu_{t}^{N},  \varphi)=\displaystyle (\mu_{0}^{N},\varphi) + \gamma\int_{0}^{t}(\mu_{r}^{N},\bigtriangleup \varphi) dr + \mathcal{H}_{t}^{N,\varphi}$, with $\mathcal{H}_{t}^{N,\varphi}=\displaystyle\frac{\sqrt{2\gamma}}{N} \sum_{i=1}^{N} \int_{0}^{t}\bigtriangledown\varphi(X_{r}^{i})dB_{r}^{i}. $
         \begin{prp}
               The sequence $\{\mu_{t}^{N},t\geq 0,N\geq 1\}$ converges in probability \\ in $C(\mathbb{R}_{+},\mathcal{M}_{F}(\mathbb{T}^{2}))$ towards $\{\mu_{t},t\geq 0\},$  where for each $t\geq 0$,  $\mu_{t}$ is the law of $X_{t}^{1}$ \\ and for any $ \varphi \in C^{2}(\mathbb{T}^{2})$, $t\geq0$, $ \displaystyle (\mu_{t},\varphi)=(\mu_{0},\varphi) + \gamma\int_{0}^{t}(\mu_{r},\bigtriangleup \varphi) dr. \label{ee} $
               \end{prp}
               \begin{proof}
              We refer to  Theorem 2.2 and Remark page 58 of  Roelly [\ref{uc}]. Let $\Pi$ be a dense subset of $C(\mathbb{T}^{2})$. In order to  prove that $(\mu^{N})_{N\geq 1}$ converges in probability in $C(\mathbb{R}_{+},\mathcal{M}_{F}(\mathbb{T}^{2}))$ towards $\{\mu_{t},t\geq 0\}$ it is enough to prove that:\vspace*{0.12cm}\\
               1-\hspace*{0.1cm}$\forall \varphi \in \Pi$, \hspace*{0.001cm} $\{(\mu_{t}^{N},\varphi),t\geq 0\}_{N\geq 1} $ is tight in $C(\mathbb{R}_{+},\mathbb{R}_{+})$\vspace*{0.2cm}\\
                2- \hspace*{0.1cm}For any $m\geq 1$,  any $(t_{1},t_{2}........,t_{m})\in \mathbb{R}_{+}^{m}$,  and  any $(\varphi_{1},\varphi_{2},.........,\varphi_{m})\in (\Pi)^{n} $ the sequence  $ \left((\mu_{t_{1}}^{N},\varphi_{1}),...,(\mu_{t_{m}}^{N},\varphi_{m})\right)$ converges in probability in  $ \mathbb{R}^{m}  $ towards  $ \left((\mu_{t_{1}},\varphi_{1}),....,(\mu_{t_{m}},\varphi_{m})\right)$\vspace*{0.2cm}\\
               Proof of 1. We choose $\Pi=C^{2}(\mathbb{T}^{2})$. Let $ \varphi \in C^{2}(\mathbb{T}^{2}) $, due to Proposition 37 of Pardoux [\ref{wc}], Corollary page 179 of Billingsleg [\ref{fc}]  a sufficient condition for the sequence  $(\mu^{N},\varphi)_{N\geq 1} $ to be  tight in $C(\mathbb{R}_{+},\mathbb{R}_{+})$ is that both \vspace*{0.1cm} \\
                 $\hspace*{1cm}\bullet \hspace*{0.2cm} (\mu_{0}^{N},\varphi) $ is tight in $\mathbb{R},$\vspace*{0.1cm}\\
                 $ \hspace*{1cm}\bullet \hspace*{0.2cm} \forall T>0 , \underset{0\leq t\leq T}{\sup} \left( \mid(\mu_{t}^{N},\bigtriangleup \varphi)\mid+ \frac{1}{N}(\mu_{t}^{N},(\bigtriangledown \varphi)^{2})\right) $ is tight in $\mathbb{R}.$ \vspace*{0.1cm}\\
                Since for all  $N\geq0$, $t\geq0$, $\mu_{t}^{N}$ is a probability measure, these two points follow readily from the fact that $\varphi $, $\bigtriangleup \varphi$ and $(\bigtriangledown \varphi)^{2}$ are bounded on $\mathbb{T}^{2}$.\vspace*{0.2cm}\\
               Proof of 2.  According to the law of large numbers, $\forall t\geq 0$, $ (\mu_{t}^{N},\varphi)  \xrightarrow{a.s}\mathbb{E}(\varphi(X_{t}^{1}))$ \\so $\left((\mu_{t_{1}}^{N},\varphi_{1}),..,(\mu_{t_{m}}^{N},\varphi_{m})\right)\xrightarrow{a.s}\left( \mathbb{E}(\varphi_{1}(X_{t}^{1})),... ,\mathbb{E}(\varphi_{m}(X_{t}^{1}))\right)=\left((\mu_{t_{1}},\varphi_{1}),..,(\mu_{t_{m}},\varphi_{m})\right).$ \vspace*{0.08cm}\\Finally the fact that $(\mu_{t}, t\geq 0)$ solves the PDE appearing in the statement follows readily from the Itô formula.
               \end{proof}
               \begin{lem}
               For any $t\geq 0$, the measure $\mu_{t}$ is absolutely continuous with respect to the Lebesgue measure and its density $f(t,.)$  verifies $\delta_{1}\leq f(t,x)\leq \delta_{2}$, $\forall x\in \mathbb{T}^{2},$ where $\delta_{1}$ and $\delta_{2}$ are defined in section \ref{se8}. \label{int1}\vspace*{-0.15cm}
               \end{lem}
               \begin{proof}
              Given that $\displaystyle \mu_{t}=\mu_{0} + \gamma\int_{0}^{t}\bigtriangleup\mu_{r}dr,$  $\mu_{t}=\Upsilon(t)\mu_{0}$.\\ Thus as from (\ref{pp}) in Proposition \ref{b5}, for any measurable subset of $\mathbb{T}^{2}$, with zero Lebesgue measure $\Upsilon(t)1_{A}\equiv 0$, the absolute continuity of   $\mu_{t}$ with respect to the Lebesgue measure follows from the fact that $\mu_{t}=\Upsilon(t)\mu_{0}$. Furthemore we notice that the law $\mu_{t}$ of $X_{t}^{1}$ is absolute continuous with respect to the Lebesgue measure, this being true whether the law $\mu_{0}$ of $X^{1}$ has or not  this property. \\$-$ Let us now  show  that $\forall x\in \mathbb{T}^{2},$ $\delta_{1}\leq f(t,x)\leq \delta_{2}.$\\We first recall that $g$ is the density of the law $\mu_{0}$ of $X^{1}$.\\ Let  $P_{t}$ be   the heat kernel  on the two dimentional torus. As the solution of the heat equation  with the initial condition $\phi$, is the function defined on $\mathbb{T}^{2}$ by $\Upsilon(t)\phi(x)=\int_{\mathbb{T}^{2}}P_{t}(x,y)\phi(y)dy,$  $\Upsilon(t)$ is non decreasing in the sence that $\forall \varphi,\psi \in L^{2}(\mathbb{T}^{2})$ such that $\varphi\leq\psi$, $\Upsilon(t)\varphi\leq\Upsilon (t)\psi$. So since for any $C\in \mathbb{R}$, $\Upsilon(t)C=C$ (which follows from (\ref{pp})) and  $f(t,.)=\Upsilon(t)g$,  the result follows from the facts that $\delta_{1}\leq g\leq \delta_{2}$ and $\Upsilon(t)$ is non decreasing.\vspace*{-0.15cm}
               \end{proof}
               \subsection{Tighness and Convergence of   $  (\mu^{S,N},\mu^{I,N})_{N \geq 1 } $  in \small{ $(D(\mathbb{R_{+}},\mathcal{M}_{F}(\mathbb{T}^{2})))^{2}$ }}
              Recall that  we equip $\mathcal{M}_{F}(\mathbb{T}^{2})$ with the topology of weak convergence and the Skorokhod space of $c\grave{a}dl\grave{a}g$ functions from $\mathbb{R}_{+}$ to $\mathcal{M}_{F}(\mathbb{T}^{2}),$ denoted  $D(\mathbb{R_{+}},\mathcal{M}_{F}(\mathbb{T}^{2}))$ with the Skorokhod topology.\\
              We first note that:
                               \begin{displaymath} \left\{ \begin{array}{lcl}
                                 	       \displaystyle  (\mu_{t}^{S,N},1_{\mathbb{T}^{2}})=\frac{1}{N}\sum_{i=1}^{N}1_{\{E_{t}^{i}=S\}}\leq1,
                                        \\
                                        \displaystyle  (\mu_{t}^{I,N},1_{\mathbb{T}^{2}})=\frac{1}{N}\sum_{i=1}^{N}1_{\{E_{t}^{i}=I\}}\leq1,
                               \end{array} \right.
                               \end{displaymath} 
                              and therefore,
                      $ \forall \varphi \in C(\mathbb{T}^{2})$
                      \begin{displaymath}  \left\{ \begin{array}{lcl}
                                 	       \displaystyle \lvert ( \mu_{t}^{S,N},\varphi)\lvert\leq\lVert \varphi \lVert_{\infty},
                                        \\
                                        \displaystyle \lvert ( \mu_{t}^{I,N},\varphi)\lvert\leq\lVert \varphi \lVert_{\infty}.
                               \end{array} \right.
                               \end{displaymath} 
                   
                      \begin{lem}
                      Let $\mathcal{H}^{1}=\{(\mu,\nu,\rho)\in (\mathcal{M}(\mathbb{T}^{2}))^{3}/(\nu,1_{\mathbb{T}^{2}})\leq 1;\hspace*{0.1cm} (\mu,\varphi)\leq (\rho,\varphi), \forall \varphi \in C(\mathbb{T}^{2};\mathbb{R}_{+}) \}$\\
                      For all $(\mu,\nu,\rho)\in\mathcal{H}^{1}$, $ \varphi \in C(\mathbb{T}^{2}),$     we have\\\\\hspace*{5cm}
                      $  \displaystyle\Big \lvert \left(\mu, \varphi \left(\nu,\frac{K}{(\rho,K)}\right)\right)\Big \lvert \leq \lVert \varphi\lVert_{\infty}. $ 
                      \label{e3}
                      \end{lem}
                      \begin{proof}
                  $\displaystyle\Big \lvert \left(\mu, \varphi \left(\nu,\frac{K}{(\rho,K)}\right)\right)\Big \lvert=\Big \lvert \int_{\mathbb{T}^{2}} \varphi(x) \int_{\mathbb{T}^{2}} \frac{K(x,y)}{\int_{\mathbb{T}^{2}}K(x',y)\rho(dx')} \nu(dy) \mu(dx)  \Big \lvert $ \\\hspace*{5.3cm} $\leq\displaystyle\lVert \varphi\lVert_{\infty} \Big \lvert \int_{\mathbb{T}^{2}}   \frac{\int_{\mathbb{T}_{2}}K(x,y)\mu(dx)}{\int_{\mathbb{T}^{2}}K(x',y)\rho(dx')} \nu(dy)\Big \lvert$\\\hspace*{5.3cm} $\leq \lVert \varphi \lVert_{\infty},$ \vspace*{0.1cm} \\ 
                  where we have exploited the symetry of K: $ K(x,y)=K(y,x)$ for the first inequality and the facts that $\displaystyle\frac{\int_{\mathbb{T}^{2}}K(x,y)\mu(dx)}{\int_{\mathbb{T}^{2}}K(x',y)\rho(dx')}\leq 1$ and $(\nu,1_{\mathbb{T}^{2}})\leq 1$ for the last inequality.
                       \end{proof}
                      We can now establish the wished tightness.
                      \begin{prp}
                                Both sequences  $ (\mu^{S,N})_{N\geq 1}$ and $(\mu^{I,N})_{N \geq 1}$ are   tight in $  D(\mathbb{R_{+}},\mathcal{M}_{F}(\mathbb{T}^{2})).\label{eee}$
                                \end{prp}
         
\begin{proof}
     - Let us  prove that  $ ( \mu^{S,N})_{N \geq 1} $ is tight in $D(\mathbb{R_{+}},\mathcal{M}_{F}(\mathbb{T}^{2})).$\\
      As already stated in the proof of Proposition \ref{ee}, it suffices to prove that\vspace*{0.16cm}\\ $\hspace*{3cm} \forall \varphi \in C^{2}(\mathbb{T}^{2})$, $ ((\mu^{S,N}_{t},\varphi), t\geq 0)_{N \geq 1} $ is tight in  $ D(\mathbb{R_{+}} ,\mathbb{R}) $.\\
       Let $\varphi \in C^{2}(\mathbb{T}^{2}) $,  we have \\ $ (\mu_{t}^{S,N},\varphi)=\displaystyle(\mu_{0}^{S,N},\varphi) + \gamma\int_{0}^{t}(\mu_{r}^{S,N},\bigtriangleup \varphi)dr - \beta \int_{0}^{t}\left(\mu_{r}^{S,N}, \varphi (\mu_{r}^{I,N},\frac{K}{(\mu^{N}_{r},K)})\right) dr + M_{t}^{N,\varphi}$ \\
       \hspace*{1.7cm}= $\displaystyle(\mu_{0}^{S,N},\varphi) + \int_{0}^{t} \gamma(\mu_{r}^{S,N},\bigtriangleup \varphi) -\beta \left(\mu_{r}^{S,N}, \varphi (\mu_{r}^{I,N},\frac{K}{(\mu^{N}_{r},K)})\right) dr + M_{t}^{N,\varphi}.$ \\ 
       We notice that $ \{ (\mu_{t}^{S,N},\varphi), t \geq 0 \} $ is a semi-martingale since $ M^{N,\varphi}$ is a square integrable  martingale. Indeed, $ M^{N,\varphi}$ is a local martingale as the  sum of local martingales, and  from Lemma \ref{e3} we deduce that\vspace*{0.12cm}  \\
       $ <M^{N,\varphi}>_{t} $ =$ \displaystyle\frac{\beta}{N} \int_{0}^{t}\left(\mu_{r}^{S,N}, \varphi^{2} (\mu_{r}^{I,N},\frac{K}{(\mu^{N}_{r},K)})\right)dr+ \frac{2\gamma}{N} \int_{0}^{t} (\mu_{r}^{S,N}, (\bigtriangledown \varphi )^{2})dr $\vspace*{0.12cm} \\ 
       \hspace*{2.2cm}$\leq\displaystyle  \frac{\beta \Arrowvert \varphi^{2}\Arrowvert_{\infty}t}{N}  + \frac{2\gamma t}{N}   \Arrowvert (\bigtriangledown \varphi )^{2} \Arrowvert_{\infty}$.  \\
        Hence  $\mathbb{E}(  \mid M^{N,\varphi}_{t} \mid^{2} ) = \mathbb{E}( <M^{N,\varphi}>_{t} ) <\infty.$ \vspace*{0.2cm}\\ 
  Consequently  \\\hspace*{1.5cm} $ (\mu_{t}^{S,N},\varphi)=\displaystyle(\mu_{0}^{S,N},\varphi) +\int_{0}^{t} \omega_{r}^{N,\varphi} dr+  M^{N,\varphi}_{t} $ with  $  < M^{N,\varphi}>_{t} =\displaystyle\int_{0}^{t} \varpi_{r}^{N,\varphi}dr,  $\\ 
      and\\\hspace*{1.5cm}
       $    \omega_{r}^{N,\varphi}=  \gamma (\mu_{r}^{S,N},\bigtriangleup \varphi) - \beta \left(\mu_{r}^{S,N}, \varphi (\mu_{r}^{I,N}\frac{K}{(\mu^{N}_{r},K)})\right), $\\\hspace*{1.5cm}  $ \varpi_{r}^{N,\varphi}=\frac{\beta}{N} \left(\mu_{r}^{S,N},\varphi^{2}(\mu_{r}^{I,N},\frac{K}{(\mu^{N}_{r},K)}) \right) + \frac{1}{N}  (\mu_{r}^{S,N} ,(\bigtriangledown \varphi)^{2} ). $ \vspace*{0.12cm}\\ Furthemore  $\omega^{N,\varphi}$ and  $ \varpi^{N,\varphi}$ are progressively measurable since they are adapted and right continuous, so according to Proposition 37 of [\ref{wc}] a sufficient condition for   $((\mu_{t}^{S,N}, \varphi ))_{N\geq 1}$  to be tight in $ D(\mathbb{R}_{+},\mathbb{R})$  is that both:
       \begin{itemize}
       \item $ \{ (\mu_{0}^{S,N}, \varphi ) , N \geq 1  \}    $ is tight in  $ \mathbb{R}, $
       \item $ \forall T\geq 0, \underset{0\leq t\leq T}{\sup}(\mid \omega_{t}^{N,\varphi} \mid + \varpi_{t}^{N,\varphi} ) $ is tight in $ \mathbb{R}. $
       \end{itemize} 
       These follow readily from the facts that:\vspace*{0.1cm}\\\hspace*{0.5cm} $-$ 
       $ \lvert (\mu_{0}^{S,N}, \varphi )\lvert \leq \Arrowvert \varphi \Arrowvert_{\infty}$.\\\hspace*{0.5cm} $-$ From Lemma \ref{e3},
       $ \mid \omega_{t}^{N,\varphi} \mid \leq \gamma\Arrowvert \bigtriangleup \varphi \Arrowvert_{\infty}+\beta\Arrowvert \varphi\Arrowvert_{\infty}$  and
      $ \varpi_{t}^{N,\varphi}\leq  \frac{\beta\Arrowvert \varphi^{2}\Arrowvert_{\infty}}{N}  + \frac{1}{N}   \Arrowvert (\bigtriangledown \varphi )^{2} \Arrowvert_{\infty}. $ \vspace*{0.2cm}\\
      The same arguments yields the tightness of   $ \{ \mu_{t}^{I,N}, t\geq 0 , N \geq 1 \} $ in $ D(\mathbb{R_{+}},\mathcal{M}_{F}(\mathbb{T}^{2})). $
      \end{proof}
                               
            \begin{prp}
            All limit points $ (\mu^{S},\mu^{I})$ of the sequence  $ (\mu^{S,N},\mu^{I,N})_{N \geq 1 }$ are  elements of $(C(\mathbb{R}_{+},\mathcal{M}_{F}(\mathbb{T}^{2})))^{2}.$
            \end{prp}
                \begin{proof}  
             Let us prove that  $\{\mu_{t}^{S}, t\geq 0\}$  is continuous. It is enough to prove that  $\forall \varphi\in C(\mathbb{T}^{2})$, the processes  $\{(\mu_{t}^{S},\varphi), t\geq 0\}$ is continuous. However  according  to Proposition 3.26 page 315 in [\ref{lc}], a sufficient condition for  $\{(\mu_{t}^{S},\varphi), t\geq 0\}$ to be continuous is that: \vspace*{0.15cm}\\\hspace*{3cm}
             $\forall T>0,\forall \varepsilon >0$ $\lim\limits_{N\rightarrow \infty}\mathbb{P}(\underset{0\leq t\leq T}{\sup} \lvert(\mu_{t}^{S,N},\varphi)-(\mu_{t^{-}}^{S,N},\varphi)\lvert>\varepsilon)=0$\vspace*{0.1cm} \\ Let T>0, $\varepsilon>0,$ since the infection of two individuals can not occur at the same time, we have: \hspace*{0.9cm} $\lvert(\mu_{t}^{S,N},\varphi)-(\mu_{t^{-}}^{S,N},\varphi)\lvert\leq\frac{1}{N}\sum\limits_{i=1}^{N}\lvert\varphi(X^{i}_{t})\lvert\lvert 1_{\{E_{t}^{i}=S\}}-1_{\{E_{t^{-}}^{i}=S\}}\lvert\\\hspace*{4.8cm}\leq \frac{\lVert \varphi \lVert_{\infty}}{N}\sum\limits_{i=1}^{N}\lvert1_{\{E_{t}^{i}=S\}}-1_{\{E_{t^{-}}^{i}=S\}}\lvert\leq\frac{\lVert \varphi \lVert_{\infty}}{N} .$ \\ So for any $\varepsilon>0$, $\lim\limits_{N\rightarrow \infty}\mathbb{P}(\underset{0\leq t\leq T}{\sup}\lvert(\mu_{t}^{S,N},\varphi)-(\mu_{t^{-}}^{S,N},\varphi)\lvert>\varepsilon)=0.$\vspace*{0.2cm}\\By a similar argument  we obtain the continuity of  $\{\mu_{t}^{I}, t\geq 0\}.$
              \end{proof}
              Let us  now state the main result of this section. 
     
      \begin{thm}
               The sequence $ (\mu^{S,N},\mu^{I,N})_{ N \geq 1 }$ converges in probability in $ (D(\mathbb{R}_{+},\mathcal{M}_{F}(\mathbb{T}^{2})))^{2} $ to\vspace*{0.1cm}\\ $ (\mu^{S},\mu^{I})$ $\in (C(\mathbb{R}_{+},\mathcal{M}_{F}(\mathbb{T}^{2})))^{2}$ where  $ \forall \varphi \in C^{2}(\mathbb{T}^{2})$,  $\{((\mu_{t}^{S},\varphi),(\mu_{t}^{I},\varphi)), t\geq 0 \}$ satisfies\vspace*{-0.13cm}
                         \[\tag{5.3}\hspace*{-3.64cm}\displaystyle   (\mu_{t}^{S},\varphi)=(\mu_{0}^{S},\varphi) +\displaystyle  \gamma\int_{0}^{t}(\mu_{r}^{S},\bigtriangleup \varphi) dr - \beta\int_{0}^{t}\left(\mu_{r}^{S}, \varphi (\mu_{r}^{I},\frac{K}{(\mu_{r},K)})\right) dr\label{e4}\]\vspace*{-0.25cm}
                          \[\tag{5.4}\displaystyle  \displaystyle (\mu_{t}^{I},\varphi)=(\mu_{0}^{I},\varphi) + \gamma\int_{0}^{t}(\mu_{r}^{I},\bigtriangleup \varphi) dr +\beta \int_{0}^{t}\left(\mu_{r}^{S}, \varphi (\mu_{r}^{I},\frac{K}{(\mu_{r},K)})\right) dr -\alpha\int_{0}^{t}(\mu_{r}^{I},\varphi)dr\label{eec}\vspace*{-0.3cm}\]
                                                       \label{eeee}\vspace*{-0.15cm}
                              \end{thm}
                           
                             \subsubsection{\textit{Proof of Theorem \ref{eeee}} \label{int2}}
  By   Proposition \ref{eee}, both  sequence $(\mu^{S,N})_{N\geq 1}$ and $(\mu^{I,N})_{N\geq 1}$ are tight in  $D(\mathbb{R_{+}},\mathcal{M}_{F}(\mathbb{T}^{2}))$,  so the sequence $(\mu^{S,N},\mu^{I,N})_{N\geq 1}$  is tight in $(D(\mathbb{R_{+}},\mathcal{M}_{F}(\mathbb{T}^{2})))^{2}$. Thus  according to Prokhorov's theorem there exists a  subsequence of $(\mu^{S,N},\mu^{I,N})_{N\geq 1}$ still denoted $(\mu^{S,N},\mu^{I,N})_{N\geq 1}$ which converges in law towards $(\mu^{S},\mu^{I})$. Hence to complete the proof of Theorm \ref{eeee} it remains to:\\
   \hspace*{1cm} $-$ Find the  system of PDEs satisfied by $\{(\mu_{t}^{S},\mu_{t}^{I}), t\geq 0\}.$ \\
   \hspace*{1cm} $-$ Show that $\forall t\geq 0$, the measure $\mu_{t}^{S}$ and $\mu_{t}^{I}$ are absolutely continuous with respect to \hspace*{1.5cm} the Lebesgue measure with densities  $f_{S}(t)$ and $f_{I}(t)$ bounded by $\delta_{2}.$\\
   \hspace*{1cm} $-$ Show that the system verifies by $(f_{S}(t),f_{I}(t))$ admits a unique solution on the set \\\hspace*{1.5cm}  $\Lambda=\{(f_{1},f_{2})/ 0\leq f_{i}\leq \delta_{2}, i\in \{1,2\}  \}$.\vspace*{0.1cm}\\
         We first prove  the  following Lemmas, which will be useful to establish the system of PDEs satisfied by $\{(\mu_{t}^{S},\mu_{t}^{I}), t\geq 0\}.$
                      \begin{lem}
          Under the assumption (H1), the function  K is Lipschitz on $\mathbb{T}^{2}\times \mathbb{T}^{2}$,  with the Lipschitz constant $2\sqrt{2}C_{k}.$
          \label{e5}
                          \end{lem}
                                                              \begin{proof}
                                                               Let $x,x',y,y' \in \mathbb{T}^{2},$ one has\\
                                                              $ \mid K(x,y)-K(x',y')\mid =\mid k(d_{\mathbb{T}^{2}}^{2}(x,y))-k(d_{\mathbb{T}^{2}}^{2}(x',y'))\mid \\\hspace*{3.6cm}\leq  C_{k}\lvert d_{\mathbb{T}^{2}}(x,y)-d_{\mathbb{T}^{2}}(x',y')\lvert(d_{\mathbb{T}^{2}}(x,y)+d_{\mathbb{T}^{2}}(x',y')).$
                                                          \vspace*{0.1cm}\\
       Furthemore $\lvert d_{\mathbb{T}^{2}}(x,y)-d_{\mathbb{T}^{2}}(x',y')\lvert\leq d_{\mathbb{T}^{2}}(x,x')+d_{\mathbb{T}^{2}}(y,y'). $\vspace*{0.12cm}\\Indeed, $\lvert d_{\mathbb{T}^{2}}(x,y)-d_{\mathbb{T}^{2}}(x',y') \lvert= d_{\mathbb{T}^{2}}(x,y)-d_{\mathbb{T}^{2}}(x',y')$ or $ d_{\mathbb{T}^{2}}(x',y')-d_{\mathbb{T}^{2}}(x,y)$  and \\$ d_{\mathbb{T}^{2}}(x,y)-d_{\mathbb{T}^{2}}(x',y')\leq  d_{\mathbb{T}^{2}}(x,x')+d_{\mathbb{T}^{2}}(x',y)-d_{\mathbb{T}^{2}}(x',y')\leq d_{\mathbb{T}^{2}}(x,x')+d_{\mathbb{T}^{2}}(y,y')$.\vspace*{0.1cm} \\Thus since $\sqrt{2}$ is the maximal distance between two points in $\mathbb{T}^{2}$, we conclude from the above results that\vspace*{-0.16cm}\[\tag{5.5}\textrm{For} \hspace*{0.1cm}\textrm{any} \hspace*{0.1cm} x,x',y,y'\in \mathbb{T}^{2},\hspace*{0.1cm}\mid K(x,y)-K(x',y')\mid \leq  2\sqrt{2}C_{k}(d_{\mathbb{T}^{2}}(x,x')+d_{\mathbb{T}^{2}}(y,y')).\vspace*{-0.35cm}\label{rrr}\]
                                                              \end{proof} 
                                                              \begin{lem} 
                                                                                                                                                                                         For all $\mu$, $\nu$ $\in \mathcal{M}(\mathbb{T}^{2})$, we have \\\hspace*{4cm}
                                                                                                                                                                                         $\underset{y}{\sup}\left\lvert \displaystyle \int_{\mathbb{T}^{2}}K(x',y)(\mu(dx')-\nu(dx'))\right\lvert\leq (\lVert k \lVert_{\infty}+2\sqrt{2}C_{k})d_{F}(\mu,\nu).$
            \label{e6}  
               \end{lem}
                \begin{proof}
                 Since from  (\ref{rrr}), for any $y\in \mathbb{T}^{2},$ the function $K(.,y)$ is Lipschitz uniformly in $y$ with the Lipschitz constant $2\sqrt{2}C_{k}$ ,   we have \\$\left\lvert \displaystyle \int_{\mathbb{T}^{2}}K(x',y)(\mu-\nu)(dz)\right\lvert=(\lVert k\lVert_{\infty} +2\sqrt{2}C_{k})\left\lvert \displaystyle \int_{\mathbb{T}^{2}}\frac{K(x',y)}{(\lVert K(.,y) \lVert_{\infty} +2\sqrt{2}C_{k})}(\mu-\nu)(dx')\right\lvert\\\hspace*{4.4cm}\leq (\lVert k\lVert_{\infty} +2\sqrt{2}C_{k})d_{F}(\mu,\nu)$.\\Hence the result.                                                                                                                           \end{proof}
                                                              \begin{lem}
                                                              The following map is continuous.\vspace*{0.2cm}\\
                                                              $ \hspace*{3cm}
                                                                 G:(\mathcal{M}_{F}(\mathbb{T}^{2}),d_{F} )\times (\mathcal{M}_{F}(\mathbb{T}^{2}),d_{F})
                                                                  \rightarrow (\mathcal{M}_{F}(\mathbb{T}^{2}\times \mathbb{T}^{2}),d_{F})\\\hspace*{8cm}
                                                                    (\mu,\nu) \longmapsto \mu\otimes \nu$ \\
                                                                  
                                                                 where  $\forall \phi\in C(\mathbb{T}^{2}\times\mathbb{T}^{2})$,  $(\mu \otimes \nu,\phi)=\displaystyle\int_{\mathbb{T}^{2}\times\mathbb{T}^{2}}\phi(x,y)\nu(dy)\mu(dx)$
                                                                 \label{e7}
                                                              \end{lem}
                                                              
                                                              \begin{proof}
        Let $(\mu,\nu) ,(\mu^{1},\nu^{1})\in (\mathcal{M}_{F}(\mathbb{T}^{2}))^{2}$; $\phi $  a Lipschitz function on $\mathbb{T}^{2}\times\mathbb{T}^{2}$, such that \\ $ \lVert \phi\lVert_{\infty}\leq 1 \hspace*{0.2cm}and \hspace*{0.2cm} \lVert \phi\lVert_{L}\leq 1.$ We have\\
        $\Big \lvert \displaystyle\int_{\mathbb{T}^{2}\times \mathbb{T}^{2}} \phi (x,y)(\mu \otimes \nu-\mu^{1} \otimes \nu^{1})(dx,dy)\Big \lvert\leq \\\hspace*{3cm}\leq\int_{\mathbb{T}^{2}}\Big \lvert \int_{\mathbb{T}^{2}}\phi (x,y) (\mu-\mu^{1})(dx)\Big \lvert\nu(dy)+ \int_{\mathbb{T}^{2}}\Big \lvert \int_{\mathbb{T}^{2}}\phi (x,y) (\nu-\nu^{1})(dy)\Big \lvert\mu^{1}(dx)\\\hspace*{3cm}\leq \nu(\mathbb{T}^{2})\underset{y}{\sup}\Big \lvert \int_{\mathbb{T}^{2}}\phi (x,y) (\mu-\mu^{1})(dx)\Big \lvert+ \mu^{1}(\mathbb{T}^{2})\underset{x}{\sup}\Big \lvert \int_{\mathbb{T}^{2}}\phi (x,y) (\nu-\nu^{1})(dy)\Big \lvert\\\hspace*{3cm} \leq C$  $(d_{F}(\mu, \mu^{1})+ d_{F}(\nu, \nu^{1}))$, \vspace*{0.1cm}\\since $\underset{y}{\sup}\lVert \phi(.,y)\lVert_{L}\vee \underset{x}{\sup}\lVert \phi(x,.)\lVert_{L}\leq \lVert \phi\lVert_{L}\leq 1.$
                                                              \end{proof}
                                                             We can now establish the system of equations satisfied by $(\mu^{S},\mu^{I}).$
      \begin{prp}
      The processes   $ (\mu^{S}, \mu^{I}) $  satisfies the equations (\ref{e4}) and (\ref{eec})
      \end{prp}    
      \begin{proof}
 We  prove this Proposition by taking the limit in the equations (\ref{e1}) and (\ref{e2}).\vspace*{0.1cm} \\                        
                                 1- Let us  prove that\\
                                 $
                                            \hspace*{2.5cm}\displaystyle \int_{0}^{t}\left(\mu_{r}^{S,N}, \varphi (\mu_{r}^{I,N},\frac{K}{(\mu_{r}^{N},K)})\right)dr  \xrightarrow{L} \int_{0}^{t}\left(\mu_{r}^{S}, \varphi ( \mu_{r}^{I},\frac{K}{(\mu_{r},K)})\right)dr. \hspace*{2cm}
                                             $\vspace*{0.1cm}\\
  One has\vspace*{0.12cm}\\  $\displaystyle\left(\mu_{r}^{S,N}, \varphi (\mu_{r}^{I,N},\frac{K}{(\mu_{r}^{N},K)})\right)=\left( \mu_{r}^{I,N},\frac{(\mu_{r}^{S,N},\varphi K)}{(\mu_{r}^{N},K)}\right)\\\hspace*{4.68cm}=-\displaystyle\left( \mu_{r}^{I,N},\frac{(\mu_{r}^{S,N},\varphi K)}{(\mu_{r}^{N},K)(\mu_{r},K)}(\mu_{r}^{N}-\mu_{r},K)\right)+\left( \mu_{r}^{I,N},\frac{(\mu_{r}^{S,N},\varphi K)}{(\mu_{r},K)}\right).$ \\ Moreover: \vspace*{0.16cm}\\1-1. Since from Lemma \ref{int1}, $f(t,.)$ is lower bounded by a positive constant and $\forall y\in \mathbb{T}^{2}$, $\displaystyle \int_{\mathbb{T}^{2}} K(x,y)dx$ is a positive constant   independent of y, \\\hspace*{4cm} $\exists \hspace*{0.1cm}
  C>0$ such that $\forall y\in \mathbb{T}^{2}$, $\displaystyle \int_{\mathbb{T}^{2}} K(x,y)f(t,x)dx\geq C.$\\   On the other hand, since from Proposition \ref{ee}  for any $\varphi\in C(\mathbb{T}^{2}),$   $(\mu_{r}^{N},\varphi)\xrightarrow{P}(\mu_{r},\varphi),$\\  $d_{F}(\mu_{r}^{N},\mu_{r})\xrightarrow{P}0$  (see Lemma \ref{b7}). Thus as  $d_{F}(\mu_{r}^{N},\mu_{r})\leq 2$,  so    from Lemmas \ref{e3} and \ref{e6}  from the Lebesgue dominated convergence theorem, we have \vspace*{0.1cm} \\
   $\mathbb{E} \left( \left\lvert \displaystyle \int_{0}^{t} \left( \mu_{r}^{I,N},\frac{(\mu_{r}^{S,N},\varphi K)}{(\mu_{r}^{N},K)(\mu_{r},K)}(\mu_{r}^{N}-\mu_{r},K)\right)dr\right\lvert \right)  $ \\\hspace*{2cm}$\leq \displaystyle\frac{1}{C}\mathbb{E} \left(  \displaystyle  \int_{0}^{t} \int_{\mathbb{T}^{2}}\left\lvert\frac{\int_{\mathbb{T}^{2}}\varphi(x)K(x,y)\mu_{r}^{S,N}(dx)}{\int_{\mathbb{T}^{2}} K(x',y)\mu_{r}^{N}(dx')}\right\lvert \left\lvert \int_{\mathbb{T}^{2}}K(x',y)(\mu_{r}^{N}-\mu_{r})(dx')\right\lvert \mu_{r}^{I,N}(dy) dr \right) \\\hspace*{1.9cm}\leq C\lVert \varphi \lVert_{\infty}(\lVert k \lVert_{\infty}+2\sqrt{2}C_{k}) \displaystyle  \int_{0}^{t} \mathbb{E} (d_{F}(\mu_{r}^{N},\mu_{r}))dr\xrightarrow{N\rightarrow \infty}0 $ \vspace*{0.2cm}\\1-2. We have $\displaystyle\left( \mu_{r}^{I,N},\frac{(\mu_{r}^{S,N},\varphi K)}{(\mu_{r},K)}\right) =\displaystyle\int_{\mathbb{T}^{2}\times\mathbb{T}^{2}}\frac{\varphi(x)K(x,y)}{\int_{\mathbb{T}^{2}}K(x',y)\mu_{r}(dx')}\mu_{r}^{I,N}(dy)\mu_{r}^{S,N}(dx).$\\ Moreover since  $\displaystyle \int_{\mathbb{T}^{2}}K(x',y)\mu_{r}(dx')$ is lower bounded by a positive constant independent of $ y\in \mathbb{T}^{2}$ and  from Lemma \ref{e5}, the map $y\in \mathbb{T}^{2}\mapsto \displaystyle \int_{\mathbb{T}^{2}}K(x',y)\mu_{r}(dx')$ is continuous,  from Lemma \ref{e5} again,  the map $(x,y)\in \mathbb{T}^{2}\times \mathbb{T}^{2}\mapsto \displaystyle \frac{\varphi(x)K(x,y)}{\int_{\mathbb{T}^{2}}K(x',y)\mu_{r}(dx')}$ is continuous and bounded on $\mathbb{T}^{2}\times \mathbb{T}^{2}$. Thus from Lemma \ref{e7}, we deduce that \vspace*{0.23cm}\\$\hspace*{1cm}\displaystyle\int_{\mathbb{T}^{2}\times\mathbb{T}^{2}}\frac{\varphi(x)K(x,y)}{\int_{\mathbb{T}^{2}}K(x',y)\mu_{r}(dx')}\mu_{r}^{I,N}(dy)\mu_{r}^{S,N}(dx)\xrightarrow{L}\displaystyle\int_{\mathbb{T}^{2}\times\mathbb{T}^{2}}\frac{\varphi(x)K(x,y)}{\int_{\mathbb{T}^{2}}K(x',y)\mu_{r}(dx')}\mu_{r}^{I}(dy)\mu_{r}^{S}(dx)$.\vspace*{0.2cm}\\
  2- Since $ \bigtriangleup \varphi$ and $ \varphi $ are continuous and bounded, \vspace*{0.1cm}\\\hspace*{2cm}
  $
 \displaystyle\int_{0}^{t}(\mu_{r}^{S,N},\bigtriangleup \varphi) dr \xrightarrow{L} \int_{0}^{t}(\mu_{r}^{S},\bigtriangleup \varphi)dr \hspace*{0.1cm} and \hspace*{0.1cm} \int_{0}^{t}(\mu_{r}^{I,N}, \varphi) dr \xrightarrow{L} \int_{0}^{t}(\mu_{r}^{I}, \varphi)dr.$\vspace*{0.13cm}\\
                                                      3 - Let us prove that,
$
                                                           M_{t}^{N,\varphi}\xrightarrow{P} 0 $ and  $ L_{t}^{N,\varphi}\xrightarrow{P}0.
                                                          $\vspace*{0.11cm}\\
                                                           From Lemma \ref{e3}, we have \vspace*{0.1cm}\\
                                                          $\hspace*{1.03cm}\mathbb{E}( \mid M^{N,\varphi}_{t} \mid^{2} ) $= $\displaystyle\mathbb{E}( < M^{N,\varphi}>_{t} ) \\\hspace*{3.1cm}  $ $= \frac{\beta}{N} \displaystyle\int_{0}^{t}  \mathbb{E}\left( \left(\mu_{r}^{S,N}, \varphi^{2} (\mu_{r}^{I,N},\frac{K}{(\mu^{N}_{r},K)})\right) \right) dr+ \frac{1}{N} \displaystyle\int_{0}^{t} $E( $ (\mu_{r}^{S,N}, (\bigtriangledown \varphi )^{2}) )dr\\\hspace*{3.1cm}\leq t \frac{\beta}{N} \lVert \varphi^{2}\lVert_{\infty}+\frac{t}{N}\lVert (\bigtriangledown \varphi )^{2} \lVert_{\infty}\xrightarrow{N\rightarrow\infty}0.$
                                                           \vspace*{0.15cm}\\
                                                             $  M^{N,\varphi}_{t} $ converges to 0 in  $ L^ {2} $, so also in probability. A similar argument yields the fact that $  L^{N,\varphi}_{t} $ converges in probability to 0.\\
                                                             Thus from the results 1-, 2-, 3-, and from the convergence of the initial measures obtained in Theorem \ref{kl}   we conclude  that  $ (\mu^{S},\mu^{I})  $ satisfies the equations (\ref{e4})  and (\ref{eec}).
                                                          \end{proof}  

                                            \begin{prp}
                                            $\forall t\geq 0 $, the measures  $\mu_{t}^{S}$ and $\mu_{t}^{I}$ are absolutely continuous with respect to the Lebesgue measure and their densities $ f_{S}(t,.)$ and $f_{I}(t,.)$ are bounded by $\delta_{2}$.\label{e8}
                                            \end{prp}
                                            \begin{proof}
                                           For any $t\geq 0$,  we have 
       $\mu_{t}^{R,N}=\frac{1}{N} \sum_{i=1}^{N}1_{\{E_{t}^{i}=R\}}\delta_{X_{t}^{i}} =\mu_{t}^{N}-\mu_{t}^{S,N}-\mu_{t}^{I,N}. $  \vspace*{0.1cm}\\ Furthemore: \\\hspace*{0.5cm}- $\mu_{t}^{S,N}+\mu_{t}^{I,N}$ weakly converges towards $\mu_{t}^{S}+\mu_{t}^{I}$ since $(\mu_{t}^{S,N},\mu_{t}^{I,N})$ weakly converges towards \hspace*{0.5cm} $(\mu_{t}^{S},\mu_{t}^{I})$ and  the map $(\mu,\nu)\in(\mathcal{M}_{F}(\mathbb{T}_{2}),d_{F})^{2}\mapsto\mu+\nu $ is continuous,\\
    \hspace*{0.5cm}- $\mu_{t}^{N}$ converges in probability towards the measure  $\mu_{t}$ which is deterministic ($\mu_{t}=\Upsilon(t)\nu$),\vspace*{0.1cm}\\
                                                                                      thus   
       $\mu_{t}^{R,N}=\mu_{t}^{N}-\mu_{t}^{S,N}-\mu_{t}^{I,N}$ weakly converges towards $\mu_{t}-\mu_{t}^{S}-\mu_{t}^{I}.$\vspace*{0.12cm}\\Hence as the measures   $\mu_{t}^{S}$, $\mu_{t}^{I}$ and $\mu_{t}-\mu_{t}^{S}-\mu_{t}^{I}$  are  non negative (since  $\mu_{t}^{S,N}$, $\mu_{t}^{I,N}$ and $\mu_{t}^{R,N}$ are non negative),    we can  conclude that\\\hspace*{0.5cm}- $ \mu_{t}^{S} $ and  $ \mu_{t}^{I} $ are absolutely continuous with respect to  the Lebesgue measure since $ \mu_{t} $ has \\\hspace*{0.5cm} this property;\\\hspace*{0.5cm}- their densities satisfy  $ f_{S}(t,.)+ f_{I}(t,.) \leq f(t,.) \leq \delta_{2}$. 
                                            \end{proof}
 
     \begin{prp}
     The  pair of densities 
     $ (f_{S}(t,.),f_{I}(t,.) )$  of the pair of measures $(\mu_{t}^{S},\mu_{t}^{I})$ satisfies 
                
         \vspace*{-0.37cm}\[\tag{5.6} \hspace*{-3cm}\displaystyle  f_{S}(t) =\Upsilon(t)f_{S}(0)-\beta \int_{0}^{t}\Upsilon(t-r)\Big[ f_{S}(r) \int_{\mathbb{T}^{2}}\frac{K(.,y) }{\int_{\mathbb{T}^{2}}K(x',y)f(r,x')dx'}f_{I}(r,y)dy\Big]dr,\label{e9}\] $\displaystyle   f_{I}(t) =\Upsilon(t)f_{I}(0)+ \beta\int_{0}^{t}\Upsilon(t-r)\Big[ f_{S}(r) \int_{\mathbb{T}^{2}}\frac{K(.,y) }{\int_{\mathbb{T}^{2}}K(x',y)f(r,x')dx'}f_{I}(r,y)dy\Big]dr$ \vspace*{-0.28cm}\[\tag{5.7}\hspace*{-10cm}-\alpha \int_{0}^{t}\Upsilon(t-r)f_{I}(r)dr.\label{e10}\]
        Moreover the system formed by the equations (\ref{e9}) and (\ref{e10}) admits a unique solution on  the set $ \Lambda = \{(f_{1},f_{2}) /  0\leq f_{i} \leq \delta_{2} , i \in \{1,2\} \}. $ 
     \end{prp}
     \begin{proof}
    Recall that $\forall y\in \mathbb{T}^{2}$, $(\mu_{r},K(.,y))=\int_{\mathbb{T}^{2}}K(x',y)f(r,x')dx'\geq C $. From the equations (\ref{e4}) and (\ref{eec}), it is easy to deduce that $ (f_{S}(t,.),f_{I}(t,.) )$ satisfies the equations (\ref{e9}) and (\ref{e10}).\\ 
   Let $ ( f_{S}^{1}(t),f_{I}^{1}(t) ) ,( f_{S}^{2}(t),f_{I}^{2}(t) ) \in \Lambda $ be two solutions of the  system formed by equations (5.6) and (5.7) with the same initial value.  Noticing that for any  $ \varphi \in L^{2}( \mathbb{T}^{2}),$ $\lVert\Upsilon(t)\varphi\lVert_{\infty}\leq \lVert \varphi \lVert_{\infty}$ (see Lemma \ref{b6})
    and $\Big \lVert \displaystyle \int_{\mathbb{T}^{2}} K(\cdot,y)\varphi(y)dy\Big\lVert_{\infty}\leq C \lVert \varphi \lVert_{\infty}$, we have \\
  $ \Arrowvert f_{S}^{2}(t)-f_{S}^{1}(t) \Arrowvert _{\infty} \leq  \displaystyle\beta  \int_{0}^{t} \Big\Arrowvert \Upsilon(t-r)\Big[(f_{S}^{2}(r)-f_{S}^{1}(r))\int_{\mathbb{T}^{2}} \frac{K(.,y) }{\int_{\mathbb{T}^{2}}K(x',y)f(r,x')dx'}f_{I}^{1}(r,y)dy\Big]\Big\Arrowvert_{\infty}dr \\\hspace*{2.5cm}+\beta \displaystyle \int_{0}^{t}\Big\Arrowvert \Upsilon(t-r)\Big[f_{S}^{2}(t)\int_{\mathbb{T}^{2}}\frac{K(.,y) }{\int_{\mathbb{T}^{2}}K(x',y)f(r,x')dx'}\left (f_{I}^{2}(r,y)-f_{I}^{1}(r,y)\right)dy\Big]\Big\Arrowvert_{\infty}dr  \\
      \hspace*{2.18cm}\leq \displaystyle \frac{\beta}{C}   \int_{0}^{t} \Arrowvert f_{S}^{2}(r)-f_{S}^{1}(r)\Arrowvert_{\infty} \Big\Arrowvert\int_{\mathbb{T}^{2}}K(.,y)f_{I}^{1}(r,y)dy\Big\Arrowvert_{\infty}dr \\\hspace*{2.4cm}+\displaystyle\frac{\beta}{C}  \displaystyle \int_{0}^{t}\Arrowvert f_{S}^{2}(r)\Arrowvert_{\infty}\Big\Arrowvert\int_{\mathbb{T}^{2}}K(.,y)\left (f_{I}^{2}(r,y)-f_{I}^{1}(r,y)\right)dy\Big\Arrowvert_{\infty}dr  \\\hspace*{2.18cm}\leq  \displaystyle\beta C \delta_{2} \displaystyle  \displaystyle \int_{0}^{t} \Big\{\Arrowvert f_{S}^{2}(r)-f_{S}^{1}(r)\Arrowvert_{\infty}  +\Arrowvert f_{I}^{2}(r)-f_{I}^{1}(r)\Arrowvert_{\infty}\Big\}dr. \hspace*{4.2cm} (5.8)$
                                                        \\ Furthemore as 
                                                         $\displaystyle f_{I}^{2}(t)-f_{I}^{1}(t)=-\left(f_{S}^{2}(t)-f_{S}^{1}(t)\right)-\alpha \int_{0}^{t}\Upsilon(t-r)(f_{I}^{2}(r)-f_{I}^{1}(r))dr,$  \\ then
                                                               $ \Arrowvert f_{I}^{2}(t)-f_{I}^{1}(t) \Arrowvert _{\infty} \leq \Arrowvert f_{S}^{2}(t)-f_{S}^{1}(t) \Arrowvert _{\infty}  +\displaystyle\alpha \displaystyle \int_{0}^{t}\Arrowvert f_{I}^{2}(r)-f_{I}^{1}(r)\Arrowvert_{\infty}dr. \hspace*{3.95cm} (5.9)$
                                                    \\\\  
                                                    Hence summing $ (5.8)$ and $ (5.9) $ and  applying   Gronwall's lemma,  we obtain
                                                    $  f_{s}^{1}(t)= f_{s}^{2}(t)  $ and  $ f_{I}^{1}(t)= f_{I}^{2}(t). $
                                                    \end{proof}
                                                     We can now finish the proof of Theorem \ref{eeee}.\vspace*{0.1cm}\\ Since $(\mu^{S,N},\mu^{I,N})_{N\geq 1}$ is tight in  $(D(\mathbb{R}_{+},\mathcal{M}_{F}(\mathbb{T}^{2})))^{2}$, and all   converging subsequences of the sequence $(\mu^{S,N},\mu^{I,N})_{N\geq 1}$ weakly converge  to the same limit $(\mu^{S},\mu^{I})$, the sequence $(\mu^{S,N},\mu^{I,N})_{N\geq 1}$  weakly converge in $(D(\mathbb{R}_{+},\mathcal{M}_{F}(\mathbb{T}^{2})))^{2}$ towards  $(\mu^{S},\mu^{I})$; furthemore  $(\mu^{S},\mu^{I})$ is deterministic, so we have  convergence in probability.\vspace*{-0.2cm}
                                                    
           \section{Central Limit Theorem}
         In this section we will study the convergence  of 
                      $(U^{N}=\sqrt{N}(\mu^{S,N}-\mu^{S}),V^{N}=\sqrt{N}(\mu^{I,N}-\mu^{I}))$ under the assuption (H2) below  and the convergence of $ Z^{N}=\sqrt{N}(\mu^{N}-\mu)  $ as $N\rightarrow \infty $.\\
               Note that the trajectories of these processes belong to  $ (D(\mathbb{R}_{+},\mathcal{E}(\mathbb{T}^{2})))^{2} $ and $C(\mathbb{R}_{+},\mathcal{E}(\mathbb{T}^{2})) $  respectively,   
         where $ \mathcal{E}(\mathbb{T}^{2}) $ is the space of signed measures on the torus, which can  be seen as the dual of $ C(\mathbb{T}^{2})  $. However, since the limit processes may be less regular than their approximations  we will first:\vspace*{0.2cm}\\
                   \hspace*{0.2cm} $\bullet$ Establish the equations verified by the process $ Z^{N} $  and by the pair $ (U^{N},V^{N}) .$\vspace*{0.1cm} \\
                   \hspace*{0.2cm} $ \bullet $ Fix the space in which the convergence results will be established.\vspace*{0.2cm}\\
   Then we will study the convergence of the above sequences.\vspace*{0.2cm} \\The following is assumed to  hold throughout section 6.
  \vspace*{0.1cm}\\ \textbf{Assumption (H2):   $k\in C^{3}(\mathbb{R}_{+}).$} 
  \begin{rmq}
  Let $x\in \mathbb{T}^{2},$ if we let $  A(x)=support\{K(x,.)\}$, from $(\ref{p1})$ and under $(H2)$, we have\\
  $-$ $\forall \eta \in \mathbb{N}^{2}$, $\lvert \eta \lvert \leq 2$ the map $y\in A(x)\mapsto D^{\eta}K(x,.)$ is Lipschitz and bounded with the Lipschitz constant independent of $x.$ \\$-$ $\forall \eta \in \mathbb{N}^{2}$, $\lvert \eta \lvert \leq 3$ the map $y\in A(x)\mapsto D^{\eta}K(x,.)$ is continuous and bounded by $C\underset{0\leq \lvert\eta\lvert\leq 3}{\max}\lVert k^{(\lvert \eta\lvert)}\lVert_{\infty}.$ \\ Indeep, this  two points follow from the facts that:  \vspace*{0.1cm}\\ $-$ If \hspace*{0.3cm}$\mathcal{C}_{1}=\{2k'(\lVert x-y\lVert^{2}),4k''(\lVert x-y\lVert^{2}), 8k^{(3)}(\lVert x-y\lVert^{2})\}$ and \\\hspace*{0.5cm}$\mathcal{C}_{2}=\{(x_{1}-y_{1})^{n}(x_{2}-y_{2})^{m},(n,m)\in\{0,1,2,3\}^{2}\},$ 
                         $\forall \lvert \eta\lvert\leq 3$, $D^{\eta}K(x,y)=D^{\eta}k(\lVert x-y\lVert^{2})$  is\\\hspace*{0.4cm} written as the sum of the products of elements of  $\mathcal{C}_{1}$ and  $\mathcal{C}_{2}$;\\ $-$
                       $\forall \lvert \eta \lvert\leq 2$, $\lvert k^{(\lvert \eta \lvert)}\lvert$ is locally Lipschitz in $\mathbb{R}_{+}$ and $\forall \lvert \eta \lvert\leq 3$, $\lvert k^{(\lvert \eta \lvert)}\lvert$ is bounded in $\mathbb{R}_{+};$ \\ $-$ $\sqrt{2}$ is the maximal distance between two points of the torus.  \label{f2f}
  \end{rmq}
 
                 \begin{lem} Under the assumption (H2),  we have 
                    $ \underset{x}{\sup}\Arrowvert K(x,.) \Arrowvert_{H^{3}}<\infty.$\label{ff3}
                 \end{lem}          
                 \begin{proof}
 From (\ref{p1}), if we let $\forall x\in \mathbb{T}^{2},$  $A(x)$=support$\{K(x,.)\}$, we will have\vspace*{0.2cm}\\
                 $\hspace*{1.5cm} \lVert K(x,\cdot)\lVert_{H^{3}}
                 = \sum\limits_{\lvert \eta\lvert\leq 3}\displaystyle\int_{A(x)}\lvert D^{\eta}k(\lVert y-x\lVert^{2})\lvert^{2}dy$\\$\hspace*{3.6cm}\leq C,$\vspace*{0.2cm}\\where the first inequality follows from Remark \ref{f2f}.\vspace*{-0.17cm}
                    \end{proof}
    \subsection{Evolution equations of  $Z^{N} $ and of the  pair $(U^{N},V^{N})$ \label{se6}}
                                                                     \subsubsection{Evolution equation of $ Z^{N}=\sqrt{N}(\mu^{N}-\mu) $ }
                                                                       Let $ \varphi \in C^{2}(\mathbb{T}^{2}),$ we have \\
      $(\mu_{t}^{N},\varphi)=\displaystyle(\mu_{0}^{N},\varphi)+\gamma\int_{0}^{t}(\mu_{r}^{N},\triangle\varphi)dr + \mathcal{H}_{t}^{N,\varphi},$ where $\mathcal{H}_{t}^{N,\varphi}=\displaystyle\frac{\sqrt{2\gamma}}{N}\sum_{i=1}^{N}\int_{0}^{t} \nabla \varphi(X_{r}^{i})dX_{r}^{i},$ \\   $\displaystyle (\mu_{t},\varphi)=(\mu_{0},\varphi)+\gamma\int_{0}^{t}(\mu_{r},\triangle\varphi)dr, $\vspace*{-0.4cm} \[\tag{6.1}\hspace*{-2.8cm}\textrm{hence} \quad
      (Z_{t}^{N},\varphi)=\displaystyle(Z_{0}^{N},\varphi)+\gamma\int_{0}^{t}(Z_{r}^{N},\triangle\varphi)dr+ \widetilde{\mathcal{H}_{t}}^{N,\varphi},\hspace*{0.1cm}  \textrm{where} \hspace*{0.1cm}  \widetilde{\mathcal{H}_{t}}^{N,\varphi}=\sqrt{N}\mathcal{H}_{t}^{N,\varphi}.\label{d1}\vspace*{-0.35cm}\] 
                        \subsubsection{System of evolution equations of the pair  $ (U^{N},V^{N})$ } 
                      Let $ \varphi \in C^{2}(\mathbb{T}^{2})$,  we have\\$\hspace*{1cm}(\mu_{t}^{S,N},\varphi)=\displaystyle(\mu_{0}^{S,N},\varphi) + \gamma\int_{0}^{t}(\mu_{r}^{S,N},\bigtriangleup \varphi) dr - \beta \int_{0}^{t}\left(\mu_{r}^{S,N}, \varphi( \mu_{r}^{I,N},\frac{K}{(\mu^{N}_{r},K)})\right) dr+ M_{t}^{N,\varphi},$ \vspace*{0.1cm}\\
     $\hspace*{1cm}(\mu_{t}^{S},\varphi)=\displaystyle(\mu_{0}^{S},\varphi) + \gamma\int_{0}^{t}(\mu_{r}^{S},\bigtriangleup \varphi) dr - \beta \int_{0}^{t}\left(\mu_{r}^{S}, \varphi (\mu_{r}^{I},\frac{K}{(\mu_{r},K)}) \right) dr. $\vspace*{0.1cm}\\
     Note first that $\left(\mu_{r}^{S,N}, \varphi (\mu_{r}^{I,N},\frac{K}{(\mu^{N}_{r},K)})\right) = \displaystyle
      \int_{\mathbb{T}^{2}} \varphi(x)\int_{\mathbb{T}^{2}} \frac{K(x,y)}{\int_{\mathbb{T}^{2}} K(x',y)\mu^{N}(dx')}\mu_{r}^{I,N}(dy)\mu_{r}^{S,N}(dx) \vspace*{0.1cm}\\\hspace*{6.82cm}=\displaystyle
     \int_{\mathbb{T}^{2}}  \frac{\int_{\mathbb{T}^{2}}\varphi(x)K(x,y)\mu_{r}^{S,N}(dx)}{\int_{\mathbb{T}^{2}} K(x',y)\mu^{N}(dx')}\mu_{r}^{I,N}(dy)\vspace*{0.1cm}\\
     \hspace*{6.82cm}=\left( \mu_{r}^{I,N}, \frac{(\mu_{r}^{S,N},\varphi K)}{(\mu^{N}_{r},K)} \right).$ \\
                                                          Thus\\
      $ (U_{t}^{N},\varphi)=\displaystyle(U_{0}^{N},\varphi) + \gamma\int_{0}^{t}(U_{r}^{N},\bigtriangleup \varphi) dr - \beta \int_{0}^{t}\left( \sqrt{N}\mu_{r}^{I,N}, \frac{(\mu_{r}^{S,N},\varphi K)}{(\mu^{N}_{r},K)} \right) dr  \\ \hspace*{2cm}+ \beta \int_{0}^{t}\left( \sqrt{N}\mu_{r}^{I}, \frac{(\mu_{r}^{S},\varphi K)}{(\mu_{r},K)}\right) dr + \sqrt{N}M_{t}^{N,\varphi}$ \\
      \hspace*{1.5cm}$= \displaystyle (U_{0}^{N},\varphi) + \gamma\int_{0}^{t}(U_{r}^{N},\bigtriangleup \varphi) dr + \beta \int_{0}^{t}\left( \mu_{r}^{I,N}, \frac{(\mu_{r}^{S,N},\varphi K)}{(\mu^{N}_{r},K)(\mu_{r},K)}(Z_{r}^{N},K) \right)dr \vspace*{0.1cm}\\\hspace*{1.6cm}-\beta \int_{0}^{t}\left( \mu_{r}^{I,N}, \frac{(U_{r}^{N},\varphi K)}{(\mu_{r},K)} \right)dr - \beta \int_{0}^{t}\left( V_{r}^{N}, \frac{(\mu_{r}^{S},\varphi K)}{(\mu_{r},K)} \right)dr + \sqrt{N}M_{t}^{N,\varphi}. $ \\
      Hence if we let $\widetilde{M}_{t}^{N,\varphi}=\sqrt{N}M_{t}^{N,\varphi},$ one has \vspace*{0.08cm}\\        
  $ \displaystyle (U_{t}^{N},\varphi)=\displaystyle  (U_{0}^{N},\varphi) + \gamma\int_{0}^{t}(U_{r}^{N},\bigtriangleup \varphi) dr +\beta \int_{0}^{t}\left(Z_{r}^{N} , G_{r}^{S,I,N}\varphi \right)dr-\beta\int_{0}^{t}\left(U_{r}^{N} ,G_{r}^{I,N}\varphi \right)dr$\vspace*{-0.3cm}\[\tag{6.2}\hspace*{-8cm}- \beta\int_{0}^{t}\left( V_{r}^{N}, G_{r}^{S}\varphi \right)dr + \widetilde{M}_{t}^{N,\varphi},\label{d2}\]
                                                    and also \\
      $ \displaystyle (V_{t}^{N},\varphi)=(V_{0}^{N},\varphi) + \gamma\int_{0}^{t}(V_{r}^{N},\bigtriangleup \varphi) dr - \beta \int_{0}^{t}\left(Z_{r}^{N} ,  G_{r}^{S,I,N}\varphi \right)dr +\beta\int_{0}^{t}\left(U_{r}^{N} ,G_{r}^{I,N}\varphi \right)dr$\vspace*{-0.3cm}\[\tag{6.3}\hspace*{-5cm} +\beta\int_{0}^{t}\left( V_{r}^{N}, G_{r}^{S}\varphi \right)dr -\alpha \int_{0}^{t}(V_{r}^{N},\varphi)dr +\widetilde{L}_{t}^{N,\varphi},\label{d3}\] where $\forall x,y,x' \in \mathbb{T}^{2}$,\\\hspace*{1cm} $G_{r}^{S,I,N}\varphi(x')=\displaystyle\left(\mu_{r}^{I,N}, K(x',.)\frac{(\mu_{r}^{S,N},\varphi K)}{(\mu^{N}_{r},K)(\mu_{r},K)}\right)\vspace*{0.1cm}\\\hspace*{3.18cm}=\displaystyle\int_{\mathbb{T}^{2}}K(x',y)\frac{\int_{\mathbb{T}^{2}}\varphi(x)K(x,y)\mu_{r}^{S,N}(dx)}{\int_{\mathbb{T}^{2}} K(y',y)\mu_{r}^{N}(dy')\int_{\mathbb{T}^{2}} K(y',y)\mu_{r}(dy')}\mu_{r}^{I,N}(dy),$\vspace*{0.17cm}\\ \hspace*{1cm} $G_{r}^{I,N}\varphi(x)= \displaystyle\varphi(x)\Big(\mu_{r}^{I,N}, \frac{K(x,.)}{(\mu_{r},K)}\Big)=\displaystyle\varphi(x) \int_{\mathbb{T}^{2}}\frac{K(x,y)}{\int_{\mathbb{T}^{2}} K(y',y)\mu_{r}(dy')}\mu_{r}^{I,N}(dy),$  \vspace*{0.15cm} \\\hspace*{1cm} $G_{r}^{S}\varphi(y)=\displaystyle \frac{(\mu_{r}^{S}, \varphi K(.,y))}{(\mu_{r},K(.,y))}=\displaystyle\frac{\int_{\mathbb{T}^{2}}\varphi(x) K(x,y)\mu_{r}^{S}(dx)}{\int_{\mathbb{T}^{2}} K(y',y)\mu_{r}(dy')}.$
        \subsection{The  space of convergence of the sequences $Z^{N}$ and $(U^{N},V^{N})$ \label{ss1} }
        We first recall that for any $s>0$, the family  $(\rho^{i,s}_{n_{1},n_{2}})_{i,n_{1},n_{2}}$ (as defined in Proposition 2.2) is an orthonormal basis of $H^{s}(\mathbb{T}^{2})$.
        \begin{prp}
     Every limit point $W^{1}$ of the seguence $(\widetilde{M}^{N})_{N\geq 1}$ satisfies\vspace*{0.1cm}\\
      \hspace*{3.3cm} $
            \forall t\geq 0,\hspace*{0.2cm} \mathbb{E}(\lVert W^{1}_{t}\lVert_{H^{-s}}^{2})<\infty\quad$ iff   $ s>2$.\label{add}
        \end{prp}
        
        \begin{proof}

  We have \\ $\widetilde{M}_{t}^{N,\varphi}= \displaystyle- \frac{1}{\sqrt{N}} \sum_{i=1}^{N} \int_{0}^{t} \int_{0}^{\infty}1_{\{E_{r^{-}}^{i}=S\}}\varphi(X_{r}^{i})1_{\{u\leq \frac{\beta}{N} \sum_{j=1}^{N}\frac{K(X_{r}^{i},X_{r}^{j})}{(\mu_{r}^{N},K(.,X_{r}^{i}))} 1_{\{E_{r}^{j}=I\}}\}} \overline{M}^{i}(dr,du)  \\\hspace*{2cm}+\sqrt{\frac{2\gamma_{S}}{N}} \displaystyle\sum_{i=1}^{N} \int_{0}^{t}1_{\{E_{r}^{i}=S\}}\bigtriangledown\varphi(X_{r}^{i})dB_{r}^{i}.$\\
 $ <\widetilde{M}^{N,\varphi}>_{t}=\displaystyle\beta \int_{0}^{t}\left(\mu_{r}^{S,N}, \varphi^{2} (\mu_{r}^{I,N},\frac{K}{(\mu^{N}_{r},K)})\right)dr+ 2\gamma \int_{0}^{t} (\mu_{r}^{S,N}, (\bigtriangledown \varphi )^{2})dr, $\\\\
                                                          and   it follows from  Theorem \ref{eeee} that \vspace*{0.1cm}\\
  \hspace*{2cm}$<\widetilde{M}^{N,\varphi}>_{t}\xrightarrow{P}\displaystyle\int_{0}^{t}\left\{\beta\left(\mu_{r}^{S},\varphi^{2}(\mu_{r}^{I},\frac{K}{(\mu_{r},K)})\right)+2\gamma(\mu_{r}^{S},(\bigtriangledown\varphi)^{2})\right\}dr,$\\
                                                              furthemore  
   $\displaystyle\int_{0}^{t}\left\{\beta\left(\mu_{r}^{S},\varphi^{2}(\mu_{r}^{I},\frac{K}{(\mu_{r},K)})\right)+2\gamma(\mu_{r}^{S},(\bigtriangledown\varphi)^{2})\right\}dr$ being the quadratic variation of a Gaussian martingale of the form $(W^{1},\varphi)$,
   our aim  is to find the smallest value of s for which $\mathbb{E}(\lVert W_{t}^{1}\lVert_{H^{-s}}^{2})<\infty$. We have\\ 
    $\mathbb{E}(\lVert W_{t}^{1}\lVert_{H^{-s}}^{2})=\mathbb{E}(\sum\limits_{i,n_{1},n_{2}}\lvert (W_{t}^{1},\rho_{n_{1},n_{2}}^{i,s})\lvert^{2})$
    =$\sum\limits_{i,n_{1},n_{2}}\mathbb{E}( <(W^{1},\rho_{n_{1},n_{2}}^{i,s})>_{t}). $\\However   as $\displaystyle\int_{\mathbb{T}^{2}}\frac{\int_{\mathbb{T}^{2}}K(x,y)\mu_{r}^{S}(dx)}{\int_{\mathbb{T}^{2}}K(x',y)\mu_{r}(dx')}\mu_{r}^{I}(dy)\leq1$,  then from Lemma \ref{ap} in the Appendix below, iff s>2,  we will have \\
    $\sum\limits_{i,n_{1},n_{2}}<W^{1},\rho_{n_{1},n_{2}}^{i}>_{t}=\displaystyle\sum\limits_{i,n_{1},n_{2}}\int_{0}^{t}\beta\left(\mu_{r}^{S},(\rho_{n_{1},n_{2}}^{i,s})^{2}
    (\mu_{r}^{I},\frac{K}{(\mu_{r},K)})\right)+2\gamma(\mu_{r}^{S},(\bigtriangledown\rho_{n_{1},n_{2}}^{i,s})^{2})dr$\\
    $\hspace*{3.8cm}\leq\displaystyle\int_{0}^{t}\Big\{\beta\int_{\mathbb{T}^{2}}\int_{\mathbb{T}^{2}}\sum\limits_{i,n_{1},n_{2}}(\rho_{n_{1},n_{2}}^{i,s})^{2}(x)\frac{K(x,y)}{\int_{\mathbb{T}^{2}}K(y,x')\mu_{r}(dx')}\mu_{r}^{I}(dy)\mu_{r}^{S}(dx)\\\hspace*{9cm}+2\gamma\int_{\mathbb{T}^{2}}\sum\limits_{i,n_{1},n_{2}}(\bigtriangledown\rho_{n_{1},n_{2}}^{i,s}(x))^{2}\mu_{r}^{S}(dx)\Big\}dr$ \\$\hspace*{3.8cm}\leq\displaystyle\beta C\int_{0}^{t}\int_{\mathbb{T}^{2}}\frac{\int_{\mathbb{T}^{2}}K(x,y)\mu_{r}^{S}(dx)}{\int_{\mathbb{T}^{2}}K(y,x')\mu_{r}(dx')}\mu_{r}^{I}(dy)dr+2\gamma C\int_{0}^{t}\int_{\mathbb{T}^{2}}\mu_{r}^{S}(dx)dr$\\$\hspace*{3.8cm}\leq Ct(\beta +2\gamma).$\vspace*{0.2cm}\\
        The result follows.        
    \end{proof} 
     By  Doob's inequality and by calculations similar to those done above we obtain the following result.
                                                                       \begin{cor}
               $\forall T>0$, $s>2$, $ \exists\hspace*{0.1cm} C_{1}(T)>0,C_{2}(T)>0,C_{3}(T)>0$   such that: $\vspace*{0.12cm}\\\hspace*{5.18cm}\underset{N\geq 1}{\sup}\mathbb{E}(\underset{0\leq t \leq T}{\sup}\lVert \widetilde{\mathcal{H}}_{t}^{N} \lVert_{H^{-s}}^{2} )\leq C_{1}(T)$,\\\hspace*{5cm}  $\underset{N\geq 1}{\sup}\mathbb{E}(\underset{0\leq t \leq T}{\sup}\lVert \widetilde{M}_{t}^{N} \lVert_{H^{-s}}^{2} )\leq C_{2}(T), $ \\\hspace*{5cm} $\underset{N\geq 1}{\sup}\mathbb{E}(\underset{0\leq t \leq T}{\sup}\lVert \widetilde{L}_{t}^{N} \lVert_{H^{-s}}^{2})\leq C_{3}(T).$\label{d4}
                 \end{cor}      
                 $\textbf{In the rest of this section we arbitrarily choose 2< s <3, }$ and we prove that 
  the sequences  $(Z^{N})_{N\geq 1}$ and $(U^{N},V^{N})_{N\geq 1}$    converge in law in $ C(\mathbb{R}_{+},H^{-s}) $ and  in   $(D(\mathbb{R}_{+},H^{-s}))^{2}$ respectively, where we have equipped $C(\mathbb{R}_{+},H^{-s}) $ with the uniform topology and $D(\mathbb{R}_{+},H^{-s})$  with the Skorokhod topology.\vspace*{-0.15cm}
           \subsection{Tighness and Convergence of $(Z^{N})_{N\geq 1}$}
             Recall that the  sequence $(Z^{N})_{N\geq 1}$ satisfies (6.1). We first give an estimate for the norm of the fluctuations process $Z^{N}$ which is not uniform in $N$.
              \begin{lem}
                                            For all $N\geq 1,$  $ Z^{N}\in C(\mathbb{R}_{+},H^{-s}).$ \label{d5}
                                            \end{lem}
                                            \begin{proof} Since s>2,  $H^{s}(\mathbb{T}_{2})\hookrightarrow C(\mathbb{T}_{2})$ (see Proposition \ref{b4}). Thus\\
                                           \hspace*{2.55cm} $\lvert (Z^{N}_{t},\varphi)\lvert=\sqrt{N}\lvert\frac{1}{N}\sum\limits_{i=1}^{N}\varphi(X^{i}_{t})-(\mu_{t},\varphi)\lvert\\\hspace*{4.2cm}\leq\sqrt{N}(\frac{1}{N}(\sum\limits_{i=1}^{N}\lvert\varphi(X^{i}_{t})\lvert+\lVert \varphi \lVert_{\infty})\\\hspace*{4.2cm}\leq 2\sqrt{N}\lVert \varphi \lVert_{\infty}\vspace*{0.1cm}\\\hspace*{4.2cm}\leq 2C\sqrt{N}\lVert \varphi \lVert_{H^{s}}.$ \\
                                           This inequality combined with  $\lVert Z^{N}_{t}\lVert_{H^{-s}}=\underset{\varphi\neq 0, \varphi\in H^{s}}{\sup}\frac{\lvert (Z^{N}_{t},\varphi)\lvert}{\lVert \varphi \lVert_{H^{s}}}$, yields  $ \displaystyle  \mathbb{E}(\underset{ t\geq 0}{\sup}\Arrowvert Z_{t}^{N}\Arrowvert_{H^{-s}}^{2})\leq 4CN.$   
                                            \end{proof}
                                          The main result of this subsection is the next Theorem.
                                                                       \begin{thm}
       The sequence  $\{Z^{N},N\geq 1\}$  converges in law  in $ C(\mathbb{R}_{+},H^{-s})$
              towards \\$\{Z_{t},t\geq 0 \} \in C(\mathbb{R}_{+},H^{-s})$,  where $ \forall  t \geq 0,$ \\
                  $ Z_{t}=\displaystyle Z_{0}+\gamma\int_{0}^{t} \bigtriangleup Z_{r}dr+\widetilde{\mathcal{H}}_{t}$ and  $\forall \varphi \in H^{s},  (\widetilde{\mathcal{H}},\varphi)$ is   a centered Gaussian martingale whose predictable quadratic variation is given by $\displaystyle <(\widetilde{\mathcal{H}},\varphi)>_{t}=2\gamma \int_{0}^{t}\left( \mu_{r},(\bigtriangledown\varphi)^{2}\right) dr.$\label{d6}
                                                                                         \end{thm}
             Before we prove this Theorem we first state a condition of Aldous-Rebolledo type for the tightness of a sequence of H-valued c$\grave{a}$dl$\grave{a}$g processes,  where H is a Hilbert space (see  definition 2.2.1, Corollary page 16, and the particular case 2.1.5 of [\ref{mc}]).
                                                                                          \begin{prp}
            Let H be a separable Hilbert space, $ (\vartheta^{n})_{n} $ a sequence of H-valued c$\grave{a}$dl$\grave{a}$g processes, their laws ($\widetilde{P}^{n}$)  form a tight sequence in $D(\mathbb{R}_{+},H)$  if \\
         $(T_{1}) \quad  $ for each t in a   dense subset $\mathbb{T}$ of $\mathbb{R}_{+}$, the sequence $(\vartheta_{t}^{n})_{n}$ is  tight in $H$\\ $
              (T_{2}) \quad \forall T>0, \forall \varepsilon_{1},\varepsilon_{2}>0 ,\exists\delta>0, n_{0}\geq 1$ such that for any stopping times $\tau^{n} \leq T \\\hspace*{4.5cm}\underset{\underset{\theta\leq\delta}{n\geq n_{0}}}{\sup} $ $\mathbb{P}(\lVert \vartheta^{n}_{(\tau^{n}+\theta) }-\vartheta^{n}_{\tau^{n}}\lVert_{H}>\varepsilon_{1})\leq\varepsilon_{2}. $   \label{d7}
                                                                                           \end{prp}
     Note that if $ (\vartheta^{n})_{n} $ is a sequence of H-valued continuous processes, then a way  to show that $ (\vartheta^{n})_{n} $ is tight is  to prove that   (T1) and  (T2) are satisfied.\\
             Let us now prove  the following  results  which are useful for the proof of Theorem \ref{d6}.
                                 \begin{prp}
          The sequence  $ \widetilde{\mathcal{H}}^{N} $ converges in law in $C(\mathbb{R}_{+},H^{-s})$ towards $\widetilde{\mathcal{H}}$ where $\forall \varphi\in H^{s}$, $(\widetilde{\mathcal{H}},\varphi)$ is a centered, continuous, Gaussian martingale having the same law as \\$(\widetilde{\mathcal{H}}_{t},\varphi)=\displaystyle\int_{0}^{t}\int_{\mathbb{T}^{2}}\frac{\partial \varphi}{\partial x_{1}}(x)\sqrt{2\gamma 
                                  f_{S}(r,x)}\mathcal{W}_{2}(dr,dx)+\int_{0}^{t}\int_{\mathbb{T}^{2}}\frac{\partial \varphi}{\partial x_{2}}(x)\sqrt{2\gamma f_{S}(r,x)}\mathcal{W}_{3}(dr,dx)\\\hspace*{1.5cm}+\int_{0}^{t}\int_{\mathbb{T}^{2}}\frac{\partial \varphi}{\partial x_{1}}(x)\sqrt{2\gamma f_{I}(r,x)}\mathcal{W}_{4}(dr,dx)+\int_{0}^{t}\int_{\mathbb{T}^{2}}\frac{\partial \varphi}{\partial x_{2}}(x)\sqrt{2\gamma f_{I}(r,x)}\mathcal{W}_{5}(dr,dx)\\\hspace*{1.5cm}+\int_{0}^{t}\int_{\mathbb{T}^{2}}\frac{\partial \varphi}{\partial x_{1}}(x)\sqrt{2\gamma (f(r,x)-f_{S}(r,x)-f_{I}(r,x))}\mathcal{W}_{7}(dr,dx)$\vspace*{-0.2cm}\[\tag{6.4}\hspace*{-4cm}+\int_{0}^{t}\int_{\mathbb{T}_{2}}\frac{\partial \varphi}{\partial x_{2}}(x)\sqrt{2\gamma (f(r,x)-f_{S}(r,x)-f_{I}(r,x))}\mathcal{W}_{8}(dr,dx),\label{d8}\vspace*{-0.15cm}\] where $\mathcal{W}_{1},\mathcal{W}_{2},\mathcal{W}_{3},\mathcal{W}_{4},\mathcal{W}_{5},\mathcal{W}_{6},\mathcal{W}_{7},\mathcal{W}_{8}$ are independent spatio-temporal white noises. \label{d9}
                                   \end{prp}

                                                                  \begin{proof}
                               We first establish the tightness of the sequence $ \widetilde{\mathcal{H}}^{N} $, then show that all converging subsequences have the same limit  which we shall  identify. However,  Theorem 2.15 of [\ref{sc}] restricted the  tightness criterion    of right-continuous martingales  (Theorem 2.3.2 of [\ref{mc}])  to that of continuous martingales, thus to avoid repetition    we obtain the tightness of $ \widetilde{\mathcal{H}}^{N} $  by adapting  the proof of  Proposition \ref{f1} below. So by Prokhorov's  theorem there exists a subsequence of $ \widetilde{\mathcal{H}}^{N} $ still denoted $ \widetilde{\mathcal{H}}^{N} $ which converge in law toward $\widetilde{\mathcal{H}} $.  By adapting  the proof of  Lemma \ref{f3} below,  we show that $\forall \varphi\in H^{s}$, the processes $(\widetilde{\mathcal{H}},\varphi) $ is a centered, continuous  martingale. On the other hand $\forall \varphi\in H^{s}$, \\\hspace*{2cm}
                           $<\widetilde{\mathcal{H}}^{N,\varphi}>_{t}=\displaystyle2\gamma\int_{0}^{t}(\mu_{r}^{N},(\bigtriangledown\varphi)^{2})dr\xrightarrow{P}\displaystyle2\gamma\int_{0}^{t}(\mu_{r},(\bigtriangledown\varphi)^{2})dr=<(\widetilde{\mathcal{H}},\varphi)>_{t},$\\thus the quadractic variation  $<(\widetilde{\mathcal{H}},\varphi)>_{t}$ being deterministic, $(\widetilde{\mathcal{H}},\varphi)$ is a centered, continuous, Gaussian martingale.\vspace*{0.2cm}\\$-$  Expression of $(\widetilde{\mathcal{H}},\varphi)$ using the white noises.\vspace*{0.2cm}\\We have $\widetilde{\mathcal{H}_{t}}^{N,\varphi}=\displaystyle\sqrt{\frac{2\gamma}{N}} \sum_{i=1}^{N} \int_{0}^{t}\bigtriangledown\varphi(X_{r}^{i})dX_{r}^{i}\\\hspace*{2.72cm}=\displaystyle\sqrt{\frac{2\gamma}{N}} \sum_{i=1}^{N} \int_{0}^{t}1_{\{E_{r}^{i}=S\}}\bigtriangledown\varphi(X_{r}^{i})dX_{r}^{i}+\displaystyle\sqrt{\frac{2\gamma}{N}} \sum_{i=1}^{N} \int_{0}^{t}1_{\{E_{r}^{i}=I\}}\bigtriangledown\varphi(X_{r}^{i})dX_{r}^{i}\\\hspace*{3cm}+\displaystyle\sqrt{\frac{2\gamma}{N}} \sum_{i=1}^{N} \int_{0}^{t}1_{\{E_{r}^{i}=R\}}\bigtriangledown\varphi(X_{r}^{i})dX_{r}^{i}.$\\Furthemore \\ $\hspace*{1cm}<\widetilde{M}^{N,\phi},\widetilde{\mathcal{H}}^{N,\varphi}>_{t}=\displaystyle2\gamma\int_{0}^{t}(\mu_{r}^{S,N},(\bigtriangledown \varphi)(\bigtriangledown \phi))dr\xrightarrow{P}\displaystyle2\gamma\int_{0}^{t}(\mu_{r}^{S},(\bigtriangledown  \varphi)(\bigtriangledown\phi))dr,$ \\\hspace*{1cm}$<\widetilde{L}^{N,\psi},\widetilde{\mathcal{H}}^{N,\varphi}>_{t}=\displaystyle2\gamma\int_{0}^{t}(\mu_{r}^{I,N},(\bigtriangledown\varphi)(\bigtriangledown \phi))dr\xrightarrow{P}\displaystyle2\gamma\int_{0}^{t}(\mu_{r}^{I},(\bigtriangledown \psi)(\bigtriangledown \phi))dr,$\\\hspace*{1cm}  $<\widetilde{\mathcal{H}}^{N,\varphi}>_{t}=\displaystyle2\gamma\int_{0}^{t}(\mu_{r}^{N},(\bigtriangledown\phi)^{2})dr\xrightarrow{P}2\gamma\int_{0}^{t}(\mu_{r},(\bigtriangledown\phi)^{2})dr.$\\ Thus since $f(r,.)-f_{S}(r,.)-f_{I}(r,.)\geq 0$, for any $t\geq 0$, $(\widetilde{\mathcal{H}}_{t},\varphi)$ satisfies (\ref{d8}).
                              \end{proof} 
                              \begin{prp}
                               There exists $ C>0$, such that for any stopping times $\overline{\tau}<\infty$ a.s and $\theta>0$, \vspace*{0.2cm}\\$\hspace*{4.6cm} \mathbb{E}\left(\Big\lVert \displaystyle\int_{\overline{\tau}}^{\overline{\tau}+\theta}\Upsilon(\overline{\tau}+\theta-r)d\widetilde{\mathcal{H}}_{r}^{N} \Big\lVert_{H^{-s}}^{2}\right)\leq C\theta.$\label{d10}
                              \end{prp}
                              \begin{proof}
                              We first recall that $(f^{i}_{n_{1},n_{2}})_{i,n_{1},n_{2}}$ (as defined is Proposition \ref{b2}) is a family of eigenfonctions of the operator $\gamma\triangle$ associated to the family of eigenvalues $(-\lambda_{n_{1},n_{2}})_{n_{1},n_{2}}$.\\From (\ref{pp}) in Proposition \ref{b5} we see that $\bigtriangledown\Upsilon(t)f^{i}_{n_{1},n_{2}}=e^{-t\lambda_{n_{1},n_{2}}}\bigtriangledown f^{i}_{n_{1},n_{2}}$,
so by noticing that for any $\varphi \in H^{s+1} \subset C^{2}(\mathbb{T}^{2}),$\\\hspace*{1.7cm} $\displaystyle\int_{\overline{\tau}}^{\overline{\tau}+\theta}(\Upsilon(\overline{\tau}+\theta-r)\varphi, d\widetilde{\mathcal{H}}_{r}^{N})=\displaystyle\sqrt{\frac{2\gamma}{N}}\sum_{i=1}^{N}\int_{\overline{\tau}}^{\overline{\tau}+\theta}\bigtriangledown\Upsilon(\overline{\tau}+\theta-r) \varphi(X_{r}^{i})dB_{r}^{i},$\\ one has\\ $\mathbb{E}\left(\Big\lVert \displaystyle\int_{\overline{\tau}}^{\overline{\tau}+\theta}\Upsilon(\overline{\tau}+\theta-r)d\widetilde{\mathcal{H}}_{r}^{N} \Big\lVert_{H^{-s}}^{2}\right) =
                                \\\hspace*{2.5cm}=\sum\limits_{i,n_{1},n_{2}}\mathbb{E}\left( \left(\displaystyle\int_{\overline{\tau}}^{\overline{\tau}+\theta}\Upsilon(\overline{\tau}+\theta-r)\rho^{i,s}_{n_{1},n_{2}},d\widetilde{\mathcal{H}}_{r}^{N} \right)^{2}\right)\\\hspace*{2.5cm}=\sum\limits_{i,n_{1},n_{2}}\displaystyle\frac{2\gamma}{N}\sum_{j=1}^{N}\mathbb{E}\left( \left(\displaystyle\int_{0}^{\theta}\bigtriangledown\Upsilon(\theta-r)\rho^{i,s}_{n_{1},n_{2}}(X_{r+\tau}^{j})dB_{r+\overline{\tau}}^{j} \right)^{2}\right)
                              \\\hspace*{2.5cm}=2\gamma\sum\limits_{i,n_{1},n_{2}}\mathbb{E}\left( \displaystyle\int_{0}^{\theta}(\mu_{r+\overline{\tau}}^{N},(1+\gamma\pi^{2}(n_{1}^{2}+n_{2}^{2}))^{-s}(\bigtriangledown\Upsilon(\theta-r)f^{i}_{n_{1},n_{2}})^{2})dr \right)\\\hspace*{2.5cm}=2\gamma\sum\limits_{i,n_{1},n_{2}}\mathbb{E}\left( \displaystyle\int_{0}^{\theta}(\mu_{r+\overline{\tau}}^{N},(1+\gamma\pi^{2}(n_{1}^{2}+n_{2}^{2}))^{-s}e^{-2(\theta-r)\lambda_{n_{1},n_{2}}}(\bigtriangledown f^{i}_{n_{1},n_{2}})^{2})dr \right)\\\hspace*{2.5cm}\leq2\gamma\mathbb{E}\left( \displaystyle\int_{0}^{\theta}\int_{\mathbb{T}^{2}}\sum\limits_{i,n_{1},n_{2}}(\bigtriangledown \rho^{i,s}_{n_{1},n_{2}}(x))^{2}\mu_{r+\overline{\tau}}^{N}(dx)dr \right)\\\hspace*{2.5cm}\leq 2\gamma C\mathbb{E}\Big(\displaystyle\int_{0}^{\theta}\int_{\mathbb{T}^{2}}\mu_{r+\overline{\tau}}^{N}(dx)dr\Big)\\\hspace*{2.5cm}\leq 2\gamma C\theta,$
                              \vspace*{0.2cm}\\ 
                              where the second inequality follows from Lemma \ref{ap} below  and the last one follows from the fact that $\displaystyle \int_{\mathbb{T}^{2}}\mu_{r+\overline{\tau}}^{N}(dx)\leq 1$.
                              \end{proof}
                                \begin{prp}
                                For all T>0,  \\\hspace*{4.5cm}  $\underset{N\geq 1}{\sup}\hspace*{0.07cm}\underset{0 \leq t\leq T}{\sup}\mathbb{E}(\Arrowvert Z_{t}^{N}\Arrowvert_{H^{-s}}^{2})<\infty.\label{d11}  $ 
                                \end{prp}
                                \begin{proof}  Recall that the semigroup $ \Upsilon(t)$ generated by $ \gamma \bigtriangleup  $ satisfies $ \lvert \Upsilon(t) \lvert_{
                                                  \mathcal{L}(H^{-s})}\leq 1 $,   where $ \lvert . \lvert_{
                                                                                       \mathcal{L}(H^{-s})} $ denotes the  operator norm on $H^{-s}$.\\From equation (\ref{d1}), we have 
                                $Z_{t}^{N}=\displaystyle\Upsilon(t)Z_{0}^{N}+\int_{0}^{t}\Upsilon(t-r)d\widetilde{\mathcal{H}}_{r}^{N}.$ \\ Thus $\underset{0 \leq t\leq T}{\sup}\mathbb{E}(\Arrowvert Z_{t}^{N}\Arrowvert_{H^{-s}}^{2})\leq 2\mathbb{E}(\Arrowvert Z_{0}^{N}\Arrowvert_{H^{-s}}^{2})+2\underset{0 \leq t\leq T}{\sup}\mathbb{E}\Big(\Big\Arrowvert \displaystyle\int_{0}^{t}\Upsilon(t-r)d\widetilde{\mathcal{H}}_{r}^{N}\Big\Arrowvert_{H^{-s}}^{2}\Big).$ 
     \\Furthemore from Proposition \ref{d10} we deduce that $\underset{0 \leq t\leq T}{\sup}\mathbb{E}\Big(\Big\Arrowvert \displaystyle\int_{0}^{t}\Upsilon(t-r)d\widetilde{\mathcal{H}}_{r}^{N}\Big\Arrowvert_{H^{-s}}^{2}\Big)\leq CT$.\\
 Combined with Proposition \ref{uo} in section 4, this show that  $\underset{N\geq1}{\sup}\hspace*{0.07cm}\underset{0 \leq t\leq T}{\sup}\mathbb{E}(\Arrowvert Z_{t}^{N}\Arrowvert_{H^{-s}}^{2})<\infty.$
    \end{proof}
          \subsubsection{\textit{Proof of Theorem \ref{d6}}}
         We  first  prove that  $ Z^{N}$ is tight in $C(\mathbb{R}_{+},H^{-s})$ then we show that all converging subsequences have the same limit.
                                                       \begin{prp}

                          The sequence  $ Z^{N}$  is tight in $ C(\mathbb{R}_{+},H^{-s}).$ \label{ff4}
                                                          \end{prp}
                                                          \begin{proof}
                                                         
                                                       We  prove that $ Z^{N}$  satisfies the conditions of  Proposition \ref{d7} with H=$H^{-s}$.\vspace*{0.2cm}\\  
  $-$ Proof of (T1).  It suffices to show that\vspace*{0.15cm}\\\hspace*{1cm}$\forall t\geq 0 $, $\forall \varepsilon >0$ there exists a compact subset $\mathcal{K}$ of $H^{-s}$ such that $\mathbb{P}(Z_{t}^{N}\notin \mathcal{K})<\varepsilon$.\vspace*{0.12cm}\\It follows from Proposition \ref{d11} that for each  2<s'<s, there exists $C$ such that, \\\hspace*{5cm} $\underset{N\geq 1}{\sup}\hspace*{0.07cm}\underset{0\leq t \leq T}{\sup}\mathbb{E}(\lVert Z_{t}^{N} \lVert_{H^{-s'}}^{2} )\leq C. $ \\  Thus since $\forall \hspace*{0.1cm}2<s'<s$ the embedding $H^{-s'}(\mathbb{T}^{2})\hookrightarrow H^{-s}(\mathbb{T}^{2})$ is compact (see theorem 1.69 page 47 of [\ref{ec}]),  $B_{H^{-s'}}=\{ \mu \in H^{-s'} ; \lVert \mu \lVert_{H^{-s'}}\leq R\}$ is a compact subset of $H^{-s}$. But \vspace*{0.2cm}\\\hspace*{3cm} $\mathbb{P}(Z_{t}^{N}\notin B_{H^{-s'}})=\mathbb{P}(\lVert Z_{t}^{N}\lVert_{H^{-s'}}> R)\leq \frac{1}{R^{2}}\mathbb{E}(\lVert Z_{t}^{N} \lVert_{H^{-s'}}^{2})\leq \frac{C}{R^{2}}.$ \\ So by  choosing R large enough we get the result. 
   \vspace*{0.2cm}\\
    - Proof of (T2). Let T>0, $\varepsilon_{1}, \varepsilon_{2}$ >0, $(\tau^{N})_{N}$ a family of stopping times such that $\tau^{N}\leq T$. \\ By noticing that  $ \forall\hspace*{0.1cm} 0\leq u\leq t $, 
     $Z^{N}_{t}=\displaystyle\Upsilon(t-u)Z^{N}_{u}+\int_{u}^{t}\Upsilon(t-r)d\widetilde{\mathcal{H}}_{r}^{N} $, we have\\
$Z^{N}_{\tau^{N}+\theta}-Z^{N}_{\tau^{N}}=(\Upsilon(\theta)-I_{d})Z^{N}_{\tau^{N}}+
\displaystyle\int_{\tau^{N}}^{\tau^{N}+\theta}\Upsilon(\tau^{N}+\theta-r)d\widetilde{\mathcal{H}}_{r}^{N}.$
\\We want to find $\delta>0$ and $N_{0}\geq 1$ such  that \vspace*{-0.23cm} \[\tag{6.5}\hspace*{-2.5cm}\underset{N\geq N_{0}}{\sup}\underset{\delta\geq \theta}{\sup}\mathbb{P}(\lVert (\Upsilon(\theta)-I_{d})Z^{N}_{\tau^{N}} \lVert_{H^{-s}}\geq\varepsilon_{1})\leq\varepsilon_{2},\label{d12}\]\vspace*{-0.23cm}\[\tag{6.6}\underset{N\geq N_{0}}{\sup}\underset{\delta\geq \theta}{\sup}\mathbb{P}\left(\left\lVert \displaystyle\int_{\tau^{N}}^{\tau^{N}+\theta}\Upsilon(\tau^{N}+\theta-r)d\widetilde{\mathcal{H}}_{r}^{N} \right\lVert_{H^{-s}}\geq\varepsilon_{1}\right)\leq\varepsilon_{2}.\label{d13}\]Proof of $(6.5)$.\hspace*{0.1cm} Recall that $(\lambda_{n_{1},n_{2}})_{n_{1},n_{2}}$ denotes the family of eigenvalues of the operator $-\gamma\bigtriangleup$. Let $m_{1},m_{2}\in\mathbb{N}^{*}$, such that\\\hspace*{3cm}  $\displaystyle\left( \frac{32\underset{N\geq 1}{\sup}\hspace*{0.07cm}\underset{0\leq t\leq T}{\sup}\mathbb{E}(\lVert Z^{N}_{t} \lVert_{H^{-s+\sigma}}^{2})}{\varepsilon_{1}^{2}\varepsilon_{2}}\right)^{\frac{1}{\sigma}}<\lambda_{m_{1},m_{2}}$, for some $0<\sigma<s-2.\hspace*{1.5cm}(6.7)$\\Note that we can choose $m_{1}$ and $m_{2}$ such that $(6.7)$ is satisfied since  \\ $\underset{N\geq1}{\sup}\hspace*{0.06cm}\underset{0\leq t\leq T}{\sup}\mathbb{E}(\lVert Z^{N}_{t} \lVert_{H^{-s+\sigma}}^{2})<\infty$  for $0<\sigma<s-2 $ (see Proposition \ref{d11}) and  $(\lambda_{m_{1},m_{2}})_{m_{1},m_{2}}$ is a non-decreasing sequence which converges to $+\infty$ as $m_{1}\longrightarrow \infty$ or $m_{2}\longrightarrow \infty$ (see Proposition \ref{b2}).\vspace*{0.2cm}\\ Let $F_{m_{1},m_{2}}$ denotes the  sub-space of $H^{s}$ generated by  \\ $\{\rho^{0},(\rho^{i,s}_{\kappa_{1},0},i\in\{5,6\}),(\rho^{i,s}_{0,\kappa_{2}},i\in\{7,8\}),(\rho^{i,s}_{\kappa_{1},\kappa_{2}}, i\in [\mid 1,4\mid]),\kappa_{1},\kappa_{2}\hspace*{0.1cm} even \hspace{0.1cm}and\hspace*{0.1cm} \kappa_{1}\leq m_{1},\kappa_{2}\leq m_{2}\}.$\vspace*{0.12cm}\\Let 
   $Z_{t\mid F_{m_{1},m_{2}}}^{N}$ be  the orthogonal projection of $Z_{t}^{N}$ on the dual space of $F_{m_{1},m_{2}}.$\vspace*{0.2cm}\\ We have  $\mathbb{P}(\lVert (\Upsilon(\theta)-I_{d})Z^{N}_{\tau^{N}} \lVert_{H^{-s}}\geq\varepsilon_{1})\leq \mathbb{P}(\lVert (\Upsilon(\theta)-I_{d})Z_{\tau^{N}\mid F_{m_{1},m_{2}}}^{N} \lVert_{H^{-s}}\geq\frac{\varepsilon_{1}}{2})\\\hspace*{7cm}+\mathbb{P}(\lVert (\Upsilon(\theta)-I_{d})(Z^{N}_{\tau^{N}}-Z_{\tau^{N}\mid F_{m_{1},m_{2}}}^{N}) \lVert_{H^{-s}}\geq\frac{\varepsilon_{1}}{2}). $\vspace*{0.2cm} \\Let us bound each of the two terms of the above right hand side.\vspace*{0.25cm}\\
 $-\quad \mathbb{P}(\lVert (\Upsilon(\theta)-I_{d})Z_{\tau^{N}\mid F_{m_{1},m_{2}}}^{N} \lVert_{H^{-s}}\geq\frac{\varepsilon_{1}}{2})\leq \frac{4}{\varepsilon_{1}^{2}}\underset{0\leq t\leq T}{\sup}\mathbb{E}(\lVert(\Upsilon(\theta)-I_{d})Z_{t\mid F_{m_{1},m_{2}}}^{N} \lVert_{H^{-s}}^{2}).$\vspace*{0.12cm} \\ Furthemore from (\ref{pp}) in Proposition \ref{b5} $\Upsilon(t)f^{i}_{\kappa_{1},\kappa_{2}}=e^{-\lambda_{\kappa_{1},\kappa_{2}}t}f^{i}_{\kappa_{1},\kappa_{2}},$ thus\\\hspace*{0.8cm} $\lVert(\Upsilon(\theta)-I_{d})Z_{t\mid F_{m_{1},m_{2}}}^{N} \lVert_{H^{-s}}^{2}=\sum\limits_{i,\kappa{1},\kappa_{2}}^{m_{1},m{2}}(1+\gamma\pi^{2}(\kappa_{1}+\kappa_{2}))^{-s}((\Upsilon(\theta)-I_{d})Z_{t}^{N},f^{i}_{\kappa_{1},\kappa_{2}})^{2}\\\hspace*{5.43cm}=\sum\limits_{i,\kappa{1},\kappa_{2}}^{m_{1},m{2}}(1+\gamma\pi^{2}(\kappa_{1}+\kappa_{2}))^{-s}(Z_{t}^{N},(\Upsilon(\theta)-I_{d})f^{i}_{\kappa_{1},\kappa_{2}})^{2} \\\hspace*{5.43cm}=\sum\limits_{i,\kappa{1},\kappa_{2}}^{m_{1},m{2}}(e^{-\theta\lambda_{\kappa_{1},\kappa_{2}}}-1)^{2}(1+\gamma\pi^{2}(\kappa_{1}+\kappa_{2}))^{-s}(Z_{t}^{N},f^{i}_{\kappa_{1},\kappa_{2}})^{2}\\\hspace*{5.43cm}\leq(e^{-\delta\lambda_{m_{1},m_{2}}}-1)^{2}\lVert Z_{t}^{N}\lVert_{H^{-s}}^{2},$ \vspace*{-0.25cm} \[\tag{6.8}\textrm{hence}\quad\mathbb{P}(\lVert (\Upsilon(\theta)-I_{d})Z_{\tau^{N}\mid F_{m_{1},m_{2}}}^{N} \lVert_{H^{-s}}\geq\frac{\varepsilon_{1}}{2})\leq\frac{4(e^{-\delta\lambda_{m_{1},m_{2}}}-1)^{2}}{\varepsilon_{1}^{2}}\underset{N\geq 1}{\sup}\hspace*{0.06cm}\underset{0\leq t \leq T}{\sup}\mathbb{E}(\lVert Z_{t}^{N}\lVert_{H^{-s}}^{2}).\label{d14}\] $-$ Since $\lVert\Upsilon(t)Z^{N}_{t}\lVert_{H^{-s}}\leq\lVert Z^{N}_{t}\lVert_{H^{-s}}$,\vspace*{0.2cm}\\ $\hspace*{1cm}\mathbb{P}(\lVert (\Upsilon(\theta)-I_{d})(Z^{N}_{\tau^{N}}-Z_{\tau^{N}\mid F_{m_{1},m_{2}}}^{N}) \lVert_{H^{-s}}\geq\frac{\varepsilon_{1}}{2})\leq\mathbb{P}(\lVert (Z^{N}_{\tau^{N}}-Z_{\tau^{N}\mid F_{m_{1},m_{2}}}^{N}) \lVert_{H^{-s}}\geq\frac{\varepsilon_{1}}{4})\\\hspace*{8.84cm}\leq\frac{16}{\varepsilon_{1}^{2}}\underset{0\leq t \leq T}{\sup}\mathbb{E}(\lVert (Z^{N}_{t}-Z_{t\mid F_{m_{1},m_{2}}}^{N}) \lVert_{H^{-s}}^{2}).$ \\ On the other hand since $( \lambda_{\kappa_{1},\kappa_{2}})_{\kappa_{1},\kappa_{2}}$ is a non-decreasing sequence,   for any $0<\sigma<s-2$, one  has\vspace*{0.2cm}\\$\lVert Z_{t}^{N}-Z_{t\mid F_{m_{1},m_{2}}}^{N}\lVert_{H^{-s}}^{2}=\lVert (I_{d}-\gamma\bigtriangleup)^{\frac{-\sigma}{2}}(I_{d}-\gamma\bigtriangleup)^{\frac{\sigma}{2}}(Z_{t}^{N}-Z_{t\mid F_{m_{1},m_{2}}}^{N})\lVert_{H^{-s}}^{2}\\\hspace*{3.6cm}=\sum\limits_{i,\kappa_{1},\kappa_{2}}\lambda_{\kappa_{1},\kappa_{2}}^{-s}((I_{d}-\gamma\bigtriangleup)^{\frac{-\sigma}{2}} (I_{d}-\gamma\bigtriangleup)^{\frac{\sigma}{2}}(Z^{N}_{t}-Z_{t\mid F_{m_{1},m_{2}}}^{N}),f^{i}_{\kappa_{1},\kappa_{2}})^{2}\\\hspace*{3.6cm}=\sum\limits_{i,\kappa_{1},\kappa_{2}}\lambda_{\kappa_{1},\kappa_{2}}^{-s}(Z^{N}_{t}-Z_{t\mid F_{m_{1},m_{2}}}^{N},(I_{d}-\gamma\bigtriangleup)^{\frac{\sigma}{2}} (I_{d}-\gamma\bigtriangleup)^{\frac{-\sigma}{2}}f^{i}_{\kappa_{1},\kappa_{2}})^{2}\\\hspace*{3.6cm}=\sum\limits_{i,\kappa_{1},\kappa_{2}}\lambda_{\kappa_{1},\kappa_{2}}^{-s+\sigma}\lambda_{\kappa_{1},\kappa_{2}}^{-\sigma}(Z^{N}_{t}-Z_{t\mid F_{m_{1},m_{2}}}^{N},f^{i}_{\kappa_{1},\kappa_{2}})^{2}\\\hspace*{3.6cm}\leq\lambda_{m_{1},m_{2}}^{-\sigma}\sum\limits_{i,\kappa_{1}=m_{1}+1,\kappa_{2}=m_{2}+1}^{\infty}\lambda_{\kappa_{1},\kappa_{2}}^{-s+\sigma}(Z^{N}_{t},f^{i}_{\kappa_{1},\kappa_{2}})^{2}\\\hspace*{3.6cm}\leq\lambda_{m_{1},m_{2}}^{-\sigma}\lVert Z^{N}_{t} \lVert_{H^{-s+\sigma}}^{2} .$\\Thus\vspace*{-0.28cm}\[\tag{6.9}\mathbb{P}(\lVert (\Upsilon(\theta)-I_{d})(Z^{N}_{\tau^{N}}-Z_{\tau^{N}\mid F_{m_{1},m_{2}}}^{N}) \lVert_{H^{-s}}\geq\frac{\varepsilon_{1}}{2})\leq \frac{16\lambda_{m_{1},m_{2}}^{-\sigma}}{\varepsilon_{1}^{2}}\underset{N\geq 1}{\sup}\underset{0\leq t\leq T}{\sup}\mathbb{E}(\lVert Z^{N}_{t} \lVert_{H^{-s+\sigma}}^{2}).
 \label{d15}\vspace*{-0.1cm}\]So  from $(6.7)$, $(\ref{d14})$ and $(\ref{d15})$ we  deduce $(\ref{d12})$.\vspace*{0.3cm}\\Proof of $(\ref{d13}).$ From Proposition \ref{d10}, we have $\\\hspace*{1cm}\mathbb{P}\left(\Big\lVert \displaystyle\int_{\tau^{N}}^{\tau^{N}+\theta}\Upsilon(\tau^{N}+\theta-r)d\widetilde{\mathcal{H}}_{r}^{N} \Big\lVert_{H^{-s}}\geq\varepsilon_{1}\right)\leq \frac{1}{\varepsilon_{1}^{2}}\mathbb{E}\left(\Big\lVert \displaystyle\int_{\tau^{N}}^{\tau^{N}+\theta}\Upsilon(\tau^{N}+\theta-r)d\widetilde{\mathcal{H}}_{r}^{N} \Big\lVert_{H^{-s}}^{2}\right)$\\   
$\hspace*{8.9cm}\leq\frac{ C}{\varepsilon_{1}^{2}} \delta. $\\
   $(\ref{d13})$ follows.
(T1), (T2) are proved, hence   $(Z^{N})_{N}$ is tight in $ C(\mathbb{R}_{+},H^{-s}). $ 
                                                          \end{proof}
  We end the proof of Theorem \ref{d6} by showing the next Proposition.
  \begin{prp}
  Every limit point Z of $Z^{N}$ is solution of\vspace*{0.3cm}\\\hspace*{5.2cm} $\displaystyle Z_{t}=\Upsilon(t)Z_{0}+\int_{0}^{t}\Upsilon(t-r)d\widetilde{\mathcal{H}}_{r}. $
  \end{prp}
  \begin{proof}
  We have   $Z_{t}^{N}=\Upsilon(t)Z_{0}^{N}+\displaystyle\int_{0}^{t}\Upsilon(t-r)d\widetilde{\mathcal{H}}_{r}^{N}.$  Furthemore \\\hspace*{1cm}- according to Proposition \ref{d9}  \hspace*{0.1cm}  $\displaystyle\int_{0}^{t}\Upsilon(t-r)d\widetilde{\mathcal{H}}_{r}^{N}\xrightarrow{L} $  $\displaystyle\int_{0}^{t}\Upsilon(t-r)d\widetilde{\mathcal{H}}_{r}; $ \\\hspace*{1cm}- according to Theorem \ref{lo} \hspace*{0.1cm} $\Upsilon(t)Z_{0}^{N}\xrightarrow{L}\Upsilon(t)Z_{0}.$  \vspace*{0.12cm}\\Hence the result. \vspace*{-0.2cm} 
  \end{proof}
               \subsection{Convergence of $ (U^{N},V^{N})_{N\geq 1} $ } 
               We recall that $(U^{N},V^{N})$ satisfies the equations  (\ref{d2}) and (\ref{d3}).\\ 
               The following Lemma is proved by the same argument as used in the proof of Lemma \ref{d5}.
                \begin{lem}
        For all $N\geq 1$, the processes $U^{N}$ and $V^{N}$ belong to $D(\mathbb{R}_{+},H^{-s})$.
                                          \end{lem}
                                         
             We now state the main result of this section.                       
                                   \begin{thm}
       Under  (H2), the sequence of processes $(U^{N},V^{N})_{N\geq1}$  converges in law   in $ (D(\mathbb{R}_{+},H^{-s}))^{2}$ to the pair of processes  $(U,V)$ which belongs  to $(C(\mathbb{R}_{+},H^{-s}))^{2}$ 
            and satisfies\vspace*{-0.2cm} 
         \[\tag{6.10} U_{t}=\displaystyle U_{0} + \gamma\int_{0}^{t} \bigtriangleup U_{r}dr+\beta \int_{0}^{t}(G_{r}^{S,I})^{*}Z_{r}dr -\beta \int_{0}^{t} (G_{r}^{I})^{*}U_{r} dr  -\beta \int_{0}^{t} (G_{r}^{S})^{*}V_{r}dr  + W_{t}^{1}, \vspace*{-0.14cm}\label{df1}\]
           \[\tag{6.11}V_{t}=\displaystyle V_{0} + \gamma\int_{0}^{t}\bigtriangleup V_{r} dr-\beta \int_{0}^{t}(G_{r}^{S,I})^{*}Z_{r}dr + \beta \int_{0}^{t} (G_{r}^{I})^{*}U_{r} dr +  \int_{0}^{t} (\beta(G_{r}^{S})^{*}-\alpha I_{d}) V_{r}dr  + W_{t}^{2},\vspace*{-0.14cm}\label{df2}\]
                    where $\forall \varphi,\psi \in H^{s},( (W^{1},\varphi) ,(W^{2},\psi)) $ is a centered continuous Gaussian martingale   verifying\vspace*{0.12cm} \\
                                $\hspace*{2cm}< (W^{1},\varphi)>_{t}=\displaystyle\beta\int_{0}^{t}\Big(\mu_{r}^{S},\varphi^{2}(\mu_{r}^{I},\frac{K}{(\mu_{r},K)})\Big)dr+2\gamma\int_{0}^{t}(\mu_{r}^{S},(\bigtriangledown\varphi)^{2})dr,$\vspace*{0.15cm}\\
                         $\hspace*{2cm}< (W^{2},\psi)>_{t}=\displaystyle\beta\int_{0}^{t}\Big(\mu_{r}^{S},\psi^{2}(\mu_{r}^{I},\frac{K}{(\mu_{r},K)})\Big)dr+2\gamma\int_{0}^{t}(\mu_{r}^{I},(\bigtriangledown\psi)^{2})dr+\alpha \int_{0}^{t}(\mu_{r}^{I},\psi^{2})dr,$\vspace*{0.17cm}\\
                                                                $\hspace*{2cm}< (W^{1},\varphi),(W^{2},\psi)>_{t}=\displaystyle-\beta\int_{0}^{t}\left(\mu_{r}^{S},\varphi\psi(\mu_{r}^{I},\frac{K}{(\mu_{r},K)})\right)dr.$ \label{ff1}
                                                      \end{thm}
     The proof of this Theorem is the aim of this subsection, however  let us first prove the following preliminary results.
                 
                 \begin{prp}
                Both sequences $(\widetilde{M}^{N})_{N\geq 1}$ and $(\widetilde{L}^{N})_{N\geq 1} $ are  tight in  $ D(\mathbb{R}_{+},H^{-s}).$ \label{f1}
                 \end{prp}
                 \begin{proof}
                 $-$ Tightness of  $ (\widetilde{M}^{N})_{N\geq1}.$\vspace*{0.1cm}\\
                                We prove that $(\widetilde{M}^{N} )_{N\geq1}$ satisfies the conditions  (T1) and  (T2) of the Proposition \ref{d7}.\vspace*{0.1cm}\\$-$ Based on Corollary \ref{d4}, we deduce $(T1)$  by the same argument as used in  the proof of $(T1)$ in Theorem \ref{d6}.\vspace*{0.1cm} \\         
         $-$ Proof of (T2). First note  that  $ <\widetilde{M}^{N,\varphi}>_{t}=\displaystyle\int_{0}^{t}A_{r}^{N}(\varphi)dr $, where \\\hspace*{3cm} $A_{r}^{N}(\varphi)=$ $\beta\left(\mu_{r}^{S,N},\varphi^{2}(\mu_{r}^{I,N},\frac{K}{(\mu_{r}^{N},K)})\right)+2\gamma(\mu_{r}^{S,N},(\bigtriangledown\varphi)^{2}).$\vspace*{0.16cm}\\According to Theorem 2.3.2 in [\ref{mc}] it is enough to prove that\vspace*{0.13cm} \\
  $\forall T>0 \quad \forall \varepsilon_{1},\varepsilon_{2}>0\quad \exists\delta>0 , N_{0}\geq 1$ such that for any stopping times $\tau^{N} \leq T,\vspace*{0.2cm}$\\\hspace*{3.7cm}$\underset{N\geq N_{0}}{\sup}\underset{\theta\leq\delta}{\sup}\mathbb{P}(\lvert<\widetilde{M}^{N}>_{(\tau^{N}+\theta) }-<\widetilde{M}^{N}>_{\tau^{N}}\lvert>\varepsilon_{1})<\varepsilon_{2}, $ \vspace*{0.2cm}\\where  $<\widetilde{M}>$ is the increasing, continuous process such that, $\lVert \widetilde{M}_{t}\lVert_{H^{-s}}^{2}-< \widetilde{M}>_{t}$ is a martingale. It is  the trace  of  $\ll \widetilde{M}\gg$ which is  the continuous increasing   operator valued  process such that $\{\widetilde{M}_{t}\otimes \widetilde{M}_{t}-\ll \widetilde{M}\gg_{t}, t\geq 0\}$ is a  martingale.
   \vspace*{0.2cm}\\Let $ T>0 , \varepsilon_{1}, \varepsilon_{2}>0 $, from Lemma \ref{ap} below, we have\vspace*{0.12cm}\\
          $\lvert<\widetilde{M}^{N}>_{(\tau^{N}+\theta)}-<\widetilde{M}^{N}>_{\tau^{N}}\lvert=\lvert\sum_{i,n_{1},n_{2}}\{<\widetilde{M}^{N,\rho_{n_{1},n_{2}}^{i,s}}>_{(\tau^{N}+\theta)}-<\widetilde{M}^{N,\rho_{n_{1},n_{2}}^{i,s}}>_{\tau^{N}}\}\lvert $ \\\\
                    \hspace*{5.4cm}=$\Big\lvert\sum_{i,n_{1},n_{2}}\displaystyle\int_{\tau^{N}}^{(\tau^{N}+\theta)}A_{r}^{N}(\rho_{n_{1},n_{2}}^{i,s})dr\Big\lvert$\\\\
        \hspace*{4.9cm}=$\Big\lvert\sum_{i,n_{1},n_{2}}\displaystyle\int_{0}^{\theta}A_{(\tau^{N}+r)}^{N}(\rho_{n_{1},n_{2}}^{i,s})dr\Big \lvert $\vspace*{0.12cm}\\\hspace*{4.9cm}$\leq C\displaystyle\int_{0}^{\theta}\Big\{\int_{\mathbb{T}^{2}}\frac{\int_{\mathbb{T}^{2}}K(x,y)\mu_{\tau^{N}+r}^{S,N}(dx)}{\int_{\mathbb{T}^{2}}K(x',y)\mu_{\tau^{N}+r}^{N}(dx')}\mu_{\tau^{N}+r}^{I,N}(dy)+\int_{\mathbb{T}^{2}}\mu_{\tau^{N}+r}^{S,N}(dx)\Big\}dr $ 
        $\hspace*{4.9cm} \leq\delta C.$\vspace*{0.1cm}\\
                                                              Hence it follows from the  Markov inequality that\vspace*{0.1cm} \\
                       $\mathbb{P}(\lvert<\widetilde{M}^{N}>_{(\tau^{N}+\theta)}-<\widetilde{M}^{N}>_{\tau^{N}}\lvert>\varepsilon_{1})\leq\displaystyle\frac{\mathbb{E}(\lvert<\widetilde{M}^{N}>_{(\tau^{N}+\theta)}-<\widetilde{M}^{N}>_{\tau^{N}}\lvert)}{\varepsilon_{1}}$\\
             \hspace*{7.10cm}$\leq\displaystyle\frac{C\delta}{\varepsilon_{1}}.$ \\
                                                (T2) follows.\vspace*{0.2cm}\\
          (T1) and (T2) proved, we conclude that $ (\widetilde{M}^{N})_{N\geq 1}$ is tight in $D(\mathbb{R}_{+},H^{-s})$. The same argument yields the tightness of  $ (\widetilde{L}^{N})_{N\geq 1}$ in $D(\mathbb{R}_{+},H^{-s})$.
             \end{proof}
             
             \begin{prp}
            Every limit point $(W^{1},W^{2})$ of the sequence $(\widetilde{M}^{N}, \widetilde{L}^{N})_{N\geq 1} $ is an  element of $ C(\mathbb{R}_{+},H^{-s})\times C(\mathbb{R}_{+},H^{-s}). $\label{f2}
             \end{prp}
             \begin{proof}
            We prove that the processes  $ W^{1}$ and $ W^{2}$ are continuous. \\
            Since   Proposition 3.26 page 315 in [\ref{lc}] concerning the continuity of the limit in law of $\mathbb{R}^{n}$-valued c$\grave{a}$dl$\grave{a}$g processes  can be adapded   to   c$\grave{a}$dl$\grave{a}$g processes that take values in an arbitrary Hilbert space, to prove that $ W^{1}$ is continuous, it is enough to show that \vspace*{0.2cm}\\
                                                  $ \hspace*{3cm}\forall T> 0 \quad 
         \forall \varepsilon > 0 \quad \lim\limits_{N\longrightarrow \infty}\mathbb{P}(\underset{0\leq t \leq T}{\sup} \lVert\widetilde{M}_{t}^{N}-\widetilde{M}_{t^{-}}^{N}\lVert_{H^{-s}}>\varepsilon)=0.$\vspace*{0.1cm}\\
         Let $ T>0$,  $\varphi \in H^{s+1}$, since  the jump of $M^{N,\varphi} $ and $(\mu^{S,N},\varphi)$ happen at the same time,  we have \vspace*{0.2cm}\\\hspace*{5cm}$\mid\widetilde{M}_{t}^{N,\varphi}-\widetilde{M}_{t^{-}}^{N,\varphi}\mid \leq\frac{C \Arrowvert \varphi \Arrowvert_{H^{s}}}{\sqrt{N}}.$ \\
                                                    Thus 
           $\mathbb{P}(\underset{0\leq t \leq T}{\sup} \lVert\widetilde{M}_{t}^{N}-\widetilde{M}_{t^{-}}^{N}\lVert_{H^{-s}}>\varepsilon)\leq \frac{1}{\varepsilon}\mathbb{E}(\underset{0\leq t\leq T}{\sup}\lVert\widetilde{M}_{t}^{N}-\widetilde{M}_{t^{-}}^{N}\lVert_{H^{-s}})\leq\frac{C  }{\varepsilon\sqrt{N}},$ \vspace*{0.1cm}\\
            \hspace*{4cm} so $\lim\limits_{N\longrightarrow \infty}\mathbb{P}(\underset{0\leq t \leq T}{\sup} \lVert\widetilde{M}_{t}^{N}-\widetilde{M}_{t^{-}}^{N}\lVert_{H^{-s}}>\varepsilon)=0.$   \\ The same argument yields the fact  that  $W^{2}$ is continuous. 
              \end{proof}
              \begin{lem}
              Every limit point $(W^{1},W^{2})$ of the sequence $(\widetilde{M}^{N}, \widetilde{L}^{N})_{N\geq 1} $ is such that for any $\varphi,\psi\in H^{s+1}$, $((W^{1},\varphi),(W^{2},\psi))$ is a  martingale.\label{f3}
              \end{lem}
              \begin{proof}
             $-$  Martingale property of $ ( W^{1},\varphi)$ \vspace*{0.2cm}\\  
                         A sufficient condition for  $ ( W^{1},\varphi)$ to be  a martingale, is that, for all $k\in \mathbb{N}^{*}$, $ \Phi_{k}\in C_{b}(\mathbb{R}^{k})$, $\varphi_{1},\varphi_{2},....\varphi_{k}\in H^{s}$, and $0\leq s_{0}<s_{1}<s_{2}<s_{3}<.....<s_{k}\leq s<t$,\vspace*{0.12cm} \\\hspace*{6.5cm} $\mathbb{ E}(\phi((W^{1},\varphi)))=0,$\\
                                   where $\phi((W^{1},\varphi))=\Phi_{k}((W_{s_{1}}^{1},\varphi_{1}),(W_{s_{2}}^{1},\varphi_{2}),............,(W_{s_{k}}^{1},\varphi_{k}))((W_{t}^{1},\varphi)-(W_{s}^{1},\varphi)).$\vspace*{0.2cm}\\
                             However given that $ \widetilde{M}^{N,\varphi}$ is a martingale, $ \mathbb{E}(\phi(\widetilde{M}^{N,\varphi}))=0,$ moreover $\phi $ is continuous  thus,\\ 
                                \hspace*{4cm}$ \phi(\widetilde{M}^{N,\varphi})\xrightarrow{L} \phi((W^{1},\varphi)), $ as $ N\rightarrow \infty. $ \\
                               On the other hand  $\phi(\widetilde{M}^{N,\varphi})$ is uniformly integrable since   $\mathbb{E}[(\phi(\widetilde{M}^{N,\varphi})^{2}]\leq C, $ \\hence 
                              $\mathbb{E}(\phi((W^{1},\varphi)))$= $\lim\limits_{N\rightarrow \infty}\mathbb{E}[\phi(\widetilde{M}^{N,\varphi})]$=0.\\
                            So we conclude that  $(W^{1},\varphi)$ is a martingale. A similar argument shows that  $(W^{2},\psi)$ is a martingale.
              \end{proof}
              \newpage
             \begin{prp}
                 The  sequence $ (\widetilde{M}^{N},\widetilde{L}^{N})_{N\geq1} $ converges in law  in  $ (D(\mathbb{R}_{+},H^{-s}))^{2}$ towards  the processes $(W^{1},W^{2})\in C(\mathbb{R}_{+},H^{-s})\times C(\mathbb{R}_{+},H^{-s}) $ where  $ \forall \varphi,\psi \in H^{s}, ((W^{1},\varphi),(W^{2},\psi))$ is a centered Gaussian martingale having the same law as:\vspace*{0.12cm}\\
                 $(W^{1}_{t},\varphi)=\displaystyle-\displaystyle\int_{0}^{t}\int_{\mathbb{T}^{2}}\sqrt{\beta f_{S}(r,x)\int_{\mathbb{T}^{2}}\frac{K(x,y)}{\int_{\mathbb{T}^{2}}K(x',y)f(r,x')dx'}f_{I}(r,y)dy}\varphi(x)\mathcal{W}_{1}(dr,dx)\\\hspace*{2.3cm}+\int_{0}^{t}\int_{\mathbb{T}^{2}}\frac{\partial \varphi}{\partial x_{1}}(x)\sqrt{2\gamma f_{S}(r,x)}\mathcal{W}_{2}(dr,dx)+\int_{0}^{t}\int_{\mathbb{T}^{2}}\frac{\partial \varphi}{\partial x_{2}}(x)\sqrt{2\gamma f_{S}(r,x)}\mathcal{W}_{3}(dr,dx),\\\hspace*{15.97cm}(6.12)
                 $\\$(W^{2}_{t},\psi)=\displaystyle\displaystyle\int_{0}^{t}\int_{\mathbb{T}^{2}}\sqrt{\beta f_{S}(r,x)\int_{\mathbb{T}^{2}}\frac{K(x,y)}{\int_{\mathbb{T}^{2}}K(x',y)f(r,x')dx'}f_{I}(r,y)dy}\psi(x)\mathcal{W}_{1}(dr,dx)\\\hspace*{2.3cm}
                 +\displaystyle\int_{0}^{t}\int_{\mathbb{T}^{2}}\frac{\partial \psi}{\partial x_{1}}(x)\sqrt{2\gamma f_{I}(r,x)}\mathcal{W}_{4}(dr,dx)+\int_{0}^{t}\int_{\mathbb{T}^{2}}\frac{\partial \psi}{\partial x_{2}}(x)\sqrt{2\gamma f_{I}(r,x)}\mathcal{W}_{5}(dr,dx)\\\hspace*{2.3cm}+\int_{0}^{t}\int_{\mathbb{T}^{2}}\psi(x)\sqrt{\alpha f_{I}(r,x)}\mathcal{W}_{6}(dr,dx), \hspace*{7.17cm}(6.13)$\vspace*{0.1cm}\\ where
                    $\mathcal{W}_{1},\mathcal{W}_{2},\mathcal{W}_{3},\mathcal{W}_{4},\mathcal{W}_{5},\mathcal{W}_{6}$ are independent spatio-temporal white noises defined in Proposition \ref{d9}\label{f4}.
                                                                       \end{prp} 
 \begin{proof}
          From Proposition \ref{f1} $(\widetilde{M}^{N}, \widetilde{L}^{N})_{N\geq1 }$ is  tight in $D(\mathbb{R}_{+},H^{-s})\times D(\mathbb{R}_{+},H^{-s})$, \\hence according to   Prokhorov's theorem there exists a subsequence still denoted $(\widetilde{M}^{N}, \widetilde{L}^{N})_{N\geq1 }$ which converges in law in $ (D(\mathbb{R}_{+},H^{-s}))^{2}$ towards $(W^{1},W^{2})$. By  Proposition \ref{f2} and  Lemma \ref{f3}, $ \forall \varphi,\psi \in H^{s}$,  $((W^{1},\varphi),(W^{2},\psi))$ is a  continuous martingale, thus we end the proof of proposition \ref{f4}   by showing that 
         the centered, continuous martingale $ ((W^{1},\varphi),(W^{2},\psi))$  is Gaussian  and satisfies  (6.12) and (6.13).\vspace*{0.1cm}\\
                                                        We have \\ 
       $\widetilde{M_{t}}^{N,\varphi}= - \displaystyle\frac{1}{\sqrt{N}} \sum_{i=1}^{N} \int_{0}^{t} \int_{0}^{\infty}1_{\{E_{r^{-}}^{i}=S\}}\varphi(X_{r}^{i})1_{\{u\leq \frac{\beta}{N} \sum_{j=1}^{N}\frac{K(X_{r}^{i},X_{r}^{j})}{\mu_{r}^{N}(K(.,X_{r}^{j}))} 1_{\{E_{r}^{j}=I\}}\}} \overline{M}^{i}(dr,du)  \\\hspace*{2cm}\displaystyle+\sqrt{\frac{2\gamma}{N}} \sum_{i=1}^{N} \int_{0}^{t}1_{\{E_{r}^{i}=S\}}\bigtriangledown\varphi(X_{r}^{i})dB_{r}^{i}$ \vspace*{0.1cm}\\
                                                   \hspace*{1.2cm}=$-M_{t}^{1,N,\varphi}+ M_{t}^{2,N,\varphi},$\\
       $ \widetilde{L_{t}}^{N,\varphi} =\displaystyle\frac{1}{\sqrt{N}} \sum_{i=1}^{N} \int_{0}^{t} \int_{0}^{\infty}1_{\{E_{r^{-}}^{i}=S\}}\varphi(X_{r}^{i})1_{\{u\leq \frac{\beta}{N} \sum_{j=1}^{N}\frac{K(X_{r}^{i},X_{r}^{j})}{\mu_{r}^{N}(K(.,X_{r}^{j}))} 1_{\{E_{r}^{j}=I\}}\}} \overline{M}^{i}(dr,du))  + \\ \hspace*{1cm}\displaystyle+\sqrt{\frac{2\gamma}{N}} \sum_{i=1}^{N} \int_{0}^{t}1_{\{E_{r}^{i}=I\}}\bigtriangledown\varphi(X_{r}^{i})dB_{r}^{i} $ $   -\displaystyle \frac{1}{\sqrt{N}} \sum_{i=1}^{N} \int_{0}^{t} \int_{0}^{\alpha}1_{\{E_{r^{-}}^{i}=I\}}\varphi(X_{r}^{i})\overline{Q}^{i}(dr,du) $\\
                                                        \hspace*{1.1cm}=$ M_{t}^{1,N,\varphi}+ M_{t}^{3,N,\varphi}+M_{t}^{4,N,\varphi}.$\vspace*{0.2cm}\\                   
                    Consider for   
                                   $\varphi,\psi\in H^{s+1}(\mathbb{T}^{2})$ the following sequence of martingales \vspace*{0.1cm}\\\hspace*{1.5cm} $\widetilde{M_{t}}^{N,\varphi}+\widetilde{L_{t}}^{N,\psi}=-M_{t}^{1,N,\varphi}+ M_{t}^{2,N,\varphi}+M_{t}^{1,N,\psi}+M_{t}^{3,N,\psi}+M_{t}^{4,N,\psi}$.
                                                  \vspace*{0.2cm}\\
                                                       The martingales  $M_{t}^{1,N,\varphi},M_{t}^{2,N,\varphi},M_{t}^{3,N,\psi},M_{t}^{4,N,\psi}$ being two by two orthogonal, \vspace*{0.2cm} \\
                                                        $<\widetilde{M}^{N,\varphi}+\widetilde{L}^{N,\psi}>_{t}=<M^{1,N,\varphi}>_{t}+<M^{2,N,\varphi}>_{t}+<M^{1,N,\psi}>_{t}+<M^{3,N,\psi}>_{t}+<M^{4,N,\psi}>_{t}\\\hspace*{3.5cm}-2<M^{1,N,\varphi},M^{1,N,\psi}>_{t}.$\\
                                                       In addition we have the following convergences in probability \vspace*{0.1cm}\\
            $\hspace*{1cm}<M^{1,N,\varphi}>_{t}\xrightarrow{P}\displaystyle\beta\int_{0}^{t} \left(\mu_{r}^{S},\varphi^{2}(\mu_{r}^{I},\frac{K}{(\mu_{r},K)})\right)dr$,
                            $\hspace*{0.1cm}<M^{2,N,\varphi}>_{t}\xrightarrow{P}\displaystyle 2\gamma\int_{0}^{t} (\mu_{r}^{S},(\bigtriangledown\varphi)^{2})dr$\vspace*{0.1cm}\\
                                           $\hspace*{1cm} <M^{3,N,\psi}>_{t}\xrightarrow{P}\displaystyle2\gamma\int_{0}^{t} (\mu_{r}^{I},(\bigtriangledown\psi)^{2})dr,$\hspace*{0.18cm}$<M^{4,N,\psi}>_{t}\xrightarrow{P}\displaystyle \alpha \int_{0}^{t} (\mu_{r}^{I},\psi^{2})dr.$\\
                                                      \\On the other hand:\vspace*{0.1cm}\\ -  $\widetilde{M}^{N,\varphi}+\widetilde{L}^{N,\psi}\xrightarrow{L}(W^{1},\varphi)+(W^{2},\psi)$ along a subsequence  since \\\hspace*{0.1cm}  $( \widetilde{M}^{N,\varphi},\widetilde{L}^{N,\psi})\xrightarrow{L}((W^{1},\varphi),(W^{2},\psi)).$\\ -   $(W^{1},\varphi)+(W^{2},\psi)$
                                                        is a continuous martingale since $(W^{1},\varphi)$ and $(W^{2},\psi)$ have this  property.\vspace*{0.12cm}\\
                                                       Thus  $ (W^{1},\varphi)+(W^{2},\psi)$ is a time changed Brownian motion.\\  The quadratic variation\vspace*{0.1cm} \\ $<(W^{1},\varphi)+(W^{2},\psi)>_{t}\vspace*{0.1cm}\\\hspace*{1cm}=\displaystyle\int_{0}^{t}\beta\Big[ \left(\mu_{r}^{S},\varphi^{2}(\mu_{r}^{I},\frac{K}{(\mu_{r},K)})\right)+\left(\mu_{r}^{S},\psi^{2}(\mu_{r}^{I},\frac{K}{(\mu_{r},K)})\right)-2\left(\mu_{r}^{S},\varphi\psi(\mu_{r}^{I},\frac{K}{(\mu_{r},K)})\right)\Big]dr\vspace*{0.1cm}\\\hspace*{1.3cm}+\displaystyle\int_{0}^{t}\{2\gamma(\mu_{r}^{S},(\bigtriangledown\varphi)^{2})+ 2\gamma(\mu_{r}^{I},(\bigtriangledown\psi)^{2})+\alpha(\mu_{r}^{I},\psi^{2})\}dr$\vspace*{0.3cm}\\ of $ (W^{1},\varphi)+(W^{2},\psi) $ being deterministic then 
                                                           we conclude that $ (W^{1},\varphi)+(W^{2},\psi) $ is a Gaussian martingale having the same law as \vspace*{0.2cm}\\
    $
 (W^{1}_{t},\varphi)+(W^{2}_{t},\psi)=\\\hspace*{2cm}\displaystyle-\int_{0}^{t}\int_{\mathbb{T}^{2}}\sqrt{\beta f_{S}(r,x)\int_{\mathbb{T}^{2}}\frac{K(x,y)}{\int_{\mathbb{T}^{2}}K(x',y)f(r,x')dx'}f_{I}(r,y)dy}(\varphi(x)-\psi(x))\mathcal{W}_{1}(dr,dx)\vspace*{0.11cm}\\\hspace*{2cm}+\int_{0}^{t}\int_{\mathbb{T}^{2}}\frac{\partial \varphi}{\partial x_{1}}(x)\sqrt{2\gamma f_{S}(r,x)}\mathcal{W}_{2}(dr,dx)+\int_{0}^{t}\int_{\mathbb{T}^{2}}\frac{\partial \varphi}{\partial x_{2}}(x)\sqrt{2\gamma f_{S}(r,x)}\mathcal{W}_{3}(dr,dx)\vspace*{0.11cm}\\\hspace*{2cm}
+\displaystyle\int_{0}^{t}\int_{\mathbb{T}^{2}}\frac{\partial \psi}{\partial x_{1}}(x)\sqrt{2\gamma f_{I}(r,x)}\mathcal{W}_{4}(dr,dx)+\int_{0}^{t}\int_{\mathbb{T}^{2}}\frac{\partial \psi}{\partial x_{2}}(x)\sqrt{2\gamma f_{I}(r,x)}\mathcal{W}_{5}(dr,dx)\vspace*{0.11cm}\\\hspace*{2cm}+\int_{0}^{t}\int_{\mathbb{T}^{2}}\psi(x)\sqrt{\alpha f_{I}(r,x)}\mathcal{W}_{6}(dr,dx),$\vspace*{0.1cm}\\ where
   $\mathcal{W}_{1},\mathcal{W}_{2},\mathcal{W}_{3},\mathcal{W}_{4},\mathcal{W}_{5},\mathcal{W}_{6}$ are independent spatio-temporal white noises defined in Proposition \ref{d9}. \\So taking $\psi\equiv0$, $\varphi\equiv0$ respectively in the above equation we see that $(W^{1},\varphi)$ and $(W^{2},\psi)$ satisfy (6.12) and (6.13).
   \end{proof}
 \subsubsection{\textit{Proof of Theorem \ref{ff1}}}
 We first prove  that   $U^{N}$ and $V^{N}$ are  tight   in $D(\mathbb{R}_{+},H^{-s}) $ then we  show that  all converging subsequences of $(U^{N},V^{N})_{N\geq 1} $ have the same limit which we shall identify. 
           Let us recall the following embeddings which follow from Proposition \ref{b4}, Lemma \ref{b6} above, and from Theorem 1.69 page 47 of [\ref{ec}].\\ $-$ If $s'>1 $ then $ H^{s'}(\mathbb{T}^{2}) \subset C(\mathbb{T}^{2}) $ and for all $ \varphi\in H^{s'}$, $\Upsilon(t)\varphi \in H^{s'}$ and  $\lVert \Upsilon(t)\varphi\lVert_{H^{s'}}\leq \lVert \varphi \lVert_{H^{s'}}$\\\hspace*{0.3cm} furthemore if $ \phi,\psi\in H^{s'}$ there exists $C>0$ such that $\lVert \phi\psi \lVert_{H^{s'}}\leq C\lVert \phi  \lVert_{H^{s'}}\lVert  \psi \lVert_{H^{s'}}.$ \\
           $-$ If $s'>2$ then ,  $H^{s'+1}(\mathbb{T}^{2})\subset C^{2}(\mathbb{T}^{2}).$ \\
           $-$ $\forall s_{1}, s_{2} \in \mathbb{R} $ such that $ s_{1}>s_{2} $ the embedding $ H^{s_{1}}(\mathbb{T}^{2})\hookrightarrow H^{s_{2}}(\mathbb{T}^{2}) $ is compact.
           \begin{prp}
             There exists a constant $C>0$ such that for any s>0, $ \varphi\in H^{s},$ we have \vspace*{-0.15cm} \[\tag{6.14}\hspace*{-2.2cm}\lVert G_{r}^{S,I,N}\varphi\lVert_{H^{s}}\leq C  \lVert \varphi \lVert_{H^{s}} \underset{y}{\sup}\lVert K(.,y)  \lVert_{H^{s}},\label{f51} \vspace*{-0.4cm}\]\[\hspace*{-2.4cm}\tag{6.15}\lVert G_{r}^{I,N}\varphi\lVert_{H^{s}}\leq C\lVert \varphi \lVert_{H^{s}} \underset{y}{\sup}\Arrowvert  K(.,y) \Arrowvert_{H^{s}},\label{f52}\] \vspace*{-0.38cm}\[\hspace*{-0.41cm}\tag{6.16}\lVert G_{r}^{S}\varphi\lVert_{H^{s}}\leq C \lVert \varphi \lVert_{H^{s}} \underset{x}{\sup}\Big\Arrowvert \frac{K(x,.)}{\int_{\mathbb{T}^{2}}K(x',.)\mu_{r}(dx')}\Big \Arrowvert^{2}_{H^{s}}.\vspace*{-0.4cm}\label{f53}\]\label{f5}
           \end{prp}
           \begin{proof}  Proof of $(\ref{f51}).$ We first recall that  $\displaystyle \int_{\mathbb{T}^{2}} K(x,y)f(t,x)dx$ is lower bounded by a positive constant $C$ independently of  $y\in \mathbb{T}^{2}$ and $\displaystyle\Big\lvert \frac{\int_{\mathbb{T}^{2}}K(x,y)\varphi(x)\mu_{r}^{S,N}(dx)}{\int_{\mathbb{T}^{2}}K(u,y)\mu_{r}^{N}(du)}\Big\lvert\leq \lVert\varphi\lVert_{\infty}.$ \\ Now we have\\
  $\lVert G_{r}^{S,I,N}\varphi\lVert_{H^{s}}^{2}=\Big\lVert \displaystyle\Big(\mu_{r}^{I,N}, K\frac{(\mu_{r}^{S,N},K \varphi )}{(\mu^{N}_{r},K)(\mu_{r},K)}\Big) \Big\lVert_{H^{s}}^{2}$\\
              $=\sum\limits_{i,n_{1},n_{2}} (1+\gamma\pi^{2}(n_{1}^{2}+n_{2}^{2}))^{s}\displaystyle\Big(\int_{\mathbb{T}^{2}} \int_{\mathbb{T}^{2}}K(z,y)\frac{\int_{\mathbb{T}^{2}}K(x,y)\varphi(x)\mu_{r}^{S,N}(dx)}{\int_{\mathbb{T}^{2}}K(u,y)\mu_{r}^{N}(du)\int_{\mathbb{T}^{2}}K(u,y)\mu_{r}(du)}\\\hspace*{12cm}\times f_{n_{1},n_{2}}^{i}(z)\mu_{r}^{I,N}(dy)dz\Big)^{2} $\\
              $ \leq \sum\limits_{i,n_{1},n_{2}} (1+\gamma\pi^{2}(n_{1}^{2}+n_{2}^{2}))^{s}\displaystyle\int_{\mathbb{T}^{2}}\left(\frac{\int_{\mathbb{T}^{2}}K(x,y)\varphi(x)\mu_{r}^{S,N}(dx)}{\int_{\mathbb{T}^{2}}K(u,y)\mu_{r}^{N}(du)\int_{\mathbb{T}^{2}}K(u,y)\mu_{r}(du)}\right)^{2}\mu_{r}^{I,N}(dy)\vspace*{0.07cm}\\\hspace*{9cm} \times \int_{\mathbb{T}^{2}}\Big(\int_{\mathbb{T}^{2}}K(z,y)f_{n_{1},n_{2}}^{i}(z)dz\Big)^{2}\mu_{r}^{I,N}(dy) $\vspace*{0.07cm}\\$\leq C \lVert\varphi \lVert_{\infty}^{2} \displaystyle\int_{\mathbb{T}^{2}}\sum\limits_{i,n_{1},n_{2}} (1+\gamma\pi^{2}(n_{1}^{2}+n_{2}^{2}))^{s}\Big(\int_{\mathbb{T}^{2}}K(z,y)f_{n_{1},n_{2}}^{i}(z)dz\Big)^{2}\mu_{r}^{I,N}(dy) \vspace*{0.07cm}$\\$\leq C\lVert \varphi \lVert_{H^{s}}^{2} \underset{y}{\sup}\lVert K(.,y)  \lVert_{H^{s}}^{2}.$ \vspace*{0.2cm}\\ Proof of $(\ref{f52}).$\\
            $\lVert G_{r}^{I,N}\varphi\lVert_{H^{s}}=\Big\lVert \displaystyle\varphi\Big(\mu_{r}^{I,N},\frac{K}{(\mu_{r},K)}\Big)\Big\lVert_{H^{s}}\leq C\lVert \varphi\lVert_{H^{s}}\Big\lVert\displaystyle\Big(\mu_{r}^{I,N},\frac{K}{(\mu_{r},K)}\Big)\Big\lVert_{H^{s}}$
            and \vspace*{0.16cm}\\
            $ \lVert (\mu_{r}^{I,N},K)\lVert_{H^{s}}^{2}=
                            \sum\limits_{i,n_{1},n_{2}} (1+\gamma\pi^{2}(n_{1}^{2}+n_{2}^{2}))^{s}\displaystyle((\mu_{r}^{I,N},\frac{K}{(\mu_{r},K)}),f_{n_{1},n_{2}}^{i})_{L^{2}}^{2} $\\
                             $\hspace*{2.48cm}  =\sum\limits_{i,n_{1},n_{2}} (1+\gamma\pi^{2}(n_{1}^{2}+n_{2}^{2}))^{s}\Big(\displaystyle\int_{\mathbb{T}_{2}} \int_{\mathbb{T}^{2}}\frac{K(x,y)}{\int_{\mathbb{T}^{2}}K(u,y)\mu_{r}(du)}f_{n_{1},n_{2}}^{i}(x)\mu_{r}^{I,N}(dy)dx\Big)^{2} $\\
                              $ \hspace*{2.48cm} \leq \sum\limits_{i,n_{1},n_{2}} (1+\gamma\pi^{2}(n_{1}^{2}+n_{2}^{2}))^{s}\displaystyle\int_{\mathbb{T}^{2}}\frac{1}{\Big(\int_{\mathbb{T}^{2}}K(u,y)\mu_{r}(du)\Big)^{2}}\mu_{r}^{I,N}(dy)\\\hspace*{7cm}\times\int_{\mathbb{T}^{2}}\Big( \int_{\mathbb{T}^{2}}K(x,y)f_{n_{1},n_{2}}^{i}(x)dx\Big)^{2}\mu_{r}^{I,N}(dy) $\\
                              $ \hspace*{2.48cm} \leq C\displaystyle\int_{\mathbb{T}^{2}}\sum\limits_{i,n_{1},n_{2}} (1+\gamma\pi^{2}(n_{1}^{2}+n_{2}^{2}))^{s}\Big( \int_{\mathbb{T}^{2}}K(x,y)f_{n_{1},n_{2}}^{i}(x)dx\Big)^{2}\mu_{r}^{I,N}(dy)  $\\
                           $\hspace*{2.48cm}\leq C\underset{y}{\sup}\Arrowvert  K(.,y) \Arrowvert_{H^{s}}^{2}.$\vspace*{0.3cm}\\ Proof of $(\ref{f53}).$ Since $\int_{\mathbb{T}^{2}}\varphi^{2}(x)\mu_{r}^{S}(dx)\leq\int_{\mathbb{T}^{2}}\varphi^{2}(x)dx,$ we have\vspace*{0.2cm}\\
                           $\lVert G_{r}^{S}\varphi\lVert_{H^{s}}=$
$\Big\lVert\displaystyle \frac{( \mu_{r}^{S},K \varphi )}{( \mu_{r},K)} \Big\lVert_{H^{s}}^{2}
               =\sum\limits_{i,n_{1},n_{2}} (1+\gamma\pi^{2}(n_{1}^{2}+n_{2}^{2}))^{s}\Big(\displaystyle \frac{( \mu_{r}^{S},K \varphi )}{( \mu_{r},K  )},f_{n_{1},n_{2}}^{i}\Big)_{L^{2}}^{2} $\\
               $\hspace*{1.70cm}=\sum\limits_{i,n_{1},n_{2}} (1+\gamma\pi^{2}(n_{1}^{2}+n_{2}^{2}))^{s}\Big(\displaystyle\int_{\mathbb{T}^{2}} \frac{\int_{\mathbb{T}^{2}}K(x,y)\varphi(x)\mu_{r}^{S}(dx)}{\int_{\mathbb{T}^{2}}K(x',y)\mu_{r}(dx')}f_{n_{1},n_{2}}^{i}(y)dy\Big)^{2} $\\
               $\hspace*{1.70cm}\leq  \sum\limits_{i,n_{1},n_{2}} (1+\gamma\pi^{2}(n_{1}^{2}+n_{2}^{2}))^{s}\displaystyle\int_{\mathbb{T}^{2}}\varphi^{2}(x)\mu_{r}^{S}(dx)\\\hspace*{5cm}\times \int_{\mathbb{T}^{2}} \Big(\int_{\mathbb{T}^{2}}\frac{K(x,y)}{\int_{\mathbb{T}^{2}}K(x',y)\mu_{r}(dx')}f_{n_{1},n_{2}}^{i}(y))dy\Big)^{2}\mu_{r}^{S}(dx) $\\
               $ \hspace*{1.70cm} \leq C\lVert \varphi \lVert_{L^{2}}^{2}
                 \displaystyle\int_{\mathbb{T}^{2}}\sum\limits_{i,n_{1},n_{2}} (1+\gamma\pi^{2}(n_{1}^{2}+n_{2}^{2}))^{s} \Big(\int_{\mathbb{T}^{2}}\frac{K(x,y)}{\int_{\mathbb{T}^{2}}K(x',y)\mu_{r}(dx')}f_{n_{1},n_{2}}^{i}(y)dy\Big)^{2}\mu_{r}^{S}(dx)\\ 
              \hspace*{1.60cm}\leq C\lVert \varphi \lVert_{H^{s}}^{2} \underset{x}{\sup}\Big\Arrowvert \frac{K(x,.)}{\int_{\mathbb{T}^{2}}K(x',.)\mu_{r}(dx')}\Big \Arrowvert^{2}_{H^{s}}.$
                \end{proof}
                We have the following immediate consequence of Proposition \ref{f5}.\\
                \begin{cor}
                          There exists a constant $C>0$ such that for any $s>0$, $\forall$ $ \mathcal{V}\in H^{-s},$ we have \\\hspace*{4.07cm} $\lVert (G_{r}^{S,I,N})^{*}\mathcal{V}\lVert_{H^{-s}}\leq C   \underset{y}{\sup}\lVert K(.,y)  \lVert_{H^{s}}\lVert \mathcal{V}  \lVert_{H^{-s}},$ \\ \hspace*{4.2cm}$\lVert (G_{r}^{I,N})^{*}\mathcal{V}\lVert_{H^{-s}}\leq C \underset{y}{\sup}\Arrowvert  K(.,y) \Arrowvert_{H^{s}}\lVert \mathcal{V}  \lVert_{H^{-s}},$ \\\hspace*{4.2cm}$\lVert (G_{r}^{S})^{*}\mathcal{V}\lVert_{H^{-s}}\leq C  \underset{x}{\sup}\Big\Arrowvert \displaystyle \frac{K(x,.)}{\int_{\mathbb{T}^{2}}K(x',.)\mu_{r}(dx')}\Big \Arrowvert_{H^{s}}\lVert \mathcal{V}  \lVert_{H^{-s}}.$\label{ff2}
                           \end{cor}
                            Let us prove the  following results which will be useful to prove the tightness of $  (U^{N})_{N\geq 1}$ and $(V^{N})_{N\geq 1}$ in $ D(\mathbb{R}_{+},H^{-s}).$
                                                                          
                                                                            \begin{lem}
                                             The pair of processes   $ (U^{N},V^{N})$ satisfies $\forall \hspace*{0.1cm}0\leq u<t$,\vspace*{0.2cm}\\
                                          $\displaystyle U_{t}^{N}=\displaystyle  \Upsilon(t-u)U_{u}^{N} + \beta \int_{u}^{t}\Upsilon(t-r)(G_{r}^{S,I,N})^{*}Z_{r}^{N}dr -\beta\int_{u}^{t}\Upsilon(t-r)(G_{r}^{I,N})^{*}U_{r}^{N}dr$\vspace*{-0.35cm}  \[\tag{6.17} \hspace*{-6.2cm}- \beta\int_{u}^{t} \Upsilon(t-r)(G_{r}^{S})^{*}V_{r}^{N}dr+\int_{u}^{t}\Upsilon(t-r)d\widetilde{M}_{r}^{N},\label{f71} \vspace*{-0.2cm}\]
                                           $V_{t}^{N}=\displaystyle\Upsilon(t-u)V_{u}^{N} -  \beta \int_{u}^{t}\Upsilon(t-r)(G_{r}^{S,I,N})^{*}Z_{r}^{N}dr +\beta\int_{u}^{t}\Upsilon(t-r)(G_{r}^{I,N})^{*}U_{r}^{N}dr$ \vspace*{-0.35cm}\[\tag{6.18}\hspace*{-5.2cm}+ \int_{u}^{t} \Upsilon(t-r)[\beta(G_{r}^{I,N})^{*}-\alpha ]V_{r}^{N}dr +\int_{u}^{t}\Upsilon(t-r)d\widetilde{L}_{r}^{N}. \label{f72}\]
                                                                    \end{lem}
                                     \begin{proof}
                                           Let us consider a function $\phi$ belonging to $ C^{1,2}(\mathbb{R}_{+}\times\mathbb{T}^{2})$. By the Itô formula applied   to $\phi(t,X_{t}^{i})$  and  using a similar computation as in subsections \ref{se5} and \ref{se6}, we obtain for $0\leq u< t,$ \\
                                           $ \displaystyle (U_{t}^{N},\phi_{t})=\displaystyle  (U_{u}^{N},\phi_{u}) + \gamma \int_{u}^{t}(U_{r}^{N},\bigtriangleup \phi_{r}) dr + \int_{u}^{t}(U_{r}^{N},\frac{\partial\phi_{r}}{\partial r})dr\\\hspace*{1.6cm}+\beta \int_{u}^{t}\left(Z_{r}^{N} , \left(\mu_{r}^{I,N}, K\frac{(\mu_{r}^{S,N},\phi_{r} K)}{(\mu^{N}_{r},K)(\mu_{r},K)}\right) \right)dr -\beta\int_{u}^{t}\left(U_{r}^{N} ,\phi_{r}\big(\mu_{r}^{I,N},\frac{K}{(\mu_{r},K)}\big) \right)dr \\\hspace*{1.6cm} - \beta\int_{u}^{t}\left( V_{r}^{N},\frac{( \mu_{r}^{S},\phi_{r} K)}{(\mu_{r},K)} \right)dr + \int_{u}^{t}(\phi_{r},d\widetilde{M}_{r}^{N}). $ \vspace*{0.2cm}\\Let $\varphi \in H^{s+1}$ and $ 0\leq u <t $, consider for $ r\in [u,t] $ the mapping $\psi_{r}(x)
                                           =\Upsilon(t-r)\varphi(x) $.\\ We have that $\psi_{\cdot}(\cdot)\in C^{1,2}([u,t]\times\mathbb{T}^{2})$. Indeed: \vspace*{0.12cm}\\
                                          \hspace*{0.5cm} - For any $ r\in [u,t]$, $\psi_{r}(\cdot)\in H^{s+1}(\mathbb{T}^{2})\subset C^{2}(\mathbb{T}^{2}).$\\\hspace*{0.5cm} -  $\forall x\in \mathbb{T}^{2}$, the map  $r\in [u,t]\mapsto  \psi^{'}_{r}(x)=-\gamma\bigtriangleup( \Upsilon(t-r)\varphi(x))$ is continuous  since  $\Upsilon(t)$ is a strongly continuous semi-group and  $-\gamma\bigtriangleup( \Upsilon(t-r)\varphi(x))=\Upsilon(t-r)(-\gamma\bigtriangleup\varphi(x))$, (see Proposition \ref{b5}). Thus  replacing  $ \phi $ by $ \psi $ in the above equation we obtain, \vspace*{0.12cm}\\$ \displaystyle (U_{t}^{N},\varphi)=\displaystyle  (U_{u}^{N},\Upsilon(t-u)\varphi) +\beta \int_{u}^{t}\left(Z_{r}^{N} , \left(\mu_{r}^{I,N}, K\frac{(\mu_{r}^{S,N},\Upsilon(t-r)\varphi K)}{(\mu^{N}_{r},K)(\mu_{r},K)}\right) \right)dr \\\hspace*{1.6cm}-\beta\int_{u}^{t}\left(U_{r}^{N} ,\Upsilon(t-r)\varphi\big(\mu_{r}^{I,N},\frac{K}{(\mu_{r},K)}\big) \right)dr - \beta\int_{u}^{t}\left( V_{r}^{N},\frac{( \mu_{r}^{S},\Upsilon(t-r)\varphi K)}{(\mu_{r},K)} \right)dr\\\hspace*{1.6cm} + \int_{u}^{t}(\Upsilon(t-r)\varphi,d\widetilde{M}_{r}^{N}). $\vspace*{0.12cm}\\ We  obtain (\ref{f72}) by similar argument.
                                                                                                        \end{proof}
                  \begin{prp}
                    There exists $ C>0$ such that for  any stopping times $\overline{\tau}<\infty$ a.s and $\theta>0$, \vspace*{0.2cm}$\hspace*{4.6cm} \mathbb{E}\left(\Big\lVert \displaystyle\int_{\overline{\tau}}^{\overline{\tau}+\theta}\Upsilon(\overline{\tau}+\theta-r)d\widetilde{M}_{r}^{N} \Big\lVert_{H^{-s}}^{2}\right)\leq C\theta,$\hspace*{3.87cm}$(6.19)$\vspace*{0.1cm}\\$\hspace*{4.6cm} \mathbb{E}\left(\Big\lVert \displaystyle\int_{\overline{\tau}}^{\overline{\tau}+\theta}\Upsilon(\overline{\tau}+\theta-r)d\widetilde{L}_{r}^{N} \Big\lVert_{H^{-s}}^{2}\right)\leq C\theta.\hspace*{4.05cm}(6.20)\label{f6}$
                  \end{prp}
                  \begin{proof} Proof of $(6.19)$.\vspace*{0.1cm} Let us recall that $A_{r}^{N}(\varphi)=$ $\beta\left(\mu_{r}^{S,N},\varphi^{2}(\mu_{r}^{I,N},\frac{K}{(\mu_{r}^{N},K)})\right)+2\gamma(\mu_{r}^{S,N},(\bigtriangledown\varphi)^{2})$ and  $\displaystyle\int_{\tau^{N}}^{\tau^{N}+\theta}(\Upsilon(\tau^{N}+\theta-r)\varphi,d\widetilde{M}_{r}^{N})=\sqrt{\frac{2\gamma}{N}} \sum\limits_{i=1}^{N} \int_{\tau^{N}}^{\tau^{N}+\theta}1_{\{E_{r}^{i}=S\}}\bigtriangledown\Upsilon(\tau^{N}+\theta-r)\varphi(X_{r}^{i})dB_{r}^{i} \\\hspace*{1.5cm}- \displaystyle\sqrt{\frac{1}{N}} \sum_{i=1}^{N} \int_{\tau^{N}}^{\tau^{N}+\theta} \int_{0}^{\infty}1_{\{E_{r^{-}}^{i}=S\}}\Upsilon(\tau^{N}+\theta-r)\varphi(X_{r}^{i})1_{\{u\leq \beta \sum\limits_{j=1}^{N}\frac{K(X_{r}^{i},X_{r}^{j})}{\sum\limits_{l=1}^{N}K(X_{r}^{l},X_{r}^{j}))} 1_{\{E_{r}^{j}=I\}}\}} \overline{M}^{i}(dr,du).$\vspace*{0.1cm}\\Now from Lemma \ref{ap} below, we have\vspace*{0.2cm}\\
                                  $\displaystyle\mathbb{E}\left(\left\lVert \int_{\overline{\tau}}^{\overline{\tau}+\theta}\Upsilon(\overline{\tau}+\theta-r)d\widetilde{M}_{r}^{N} \right\lVert_{H^{-s}}^{2}\right)\leq \sum\limits_{i,n_{1},n_{2}}\mathbb{E}\left( \left(\displaystyle\int_{\overline{\tau}}^{\overline{\tau}+\theta}\Upsilon(\overline{\tau}+\theta-r)\rho^{i,s}_{n_{1},n_{2}},d\widetilde{M}_{r}^{N} \right)^{2}\right)\\\hspace*{6.395cm} =\sum\limits_{i,n_{1},n_{2}}\mathbb{E}\left( \displaystyle\int_{0}^{\theta} A_{r+\overline{\tau}}^{N}(\Upsilon(\theta-r)\rho^{i,s}_{n_{1},n_{2}})dr \right)\\\hspace*{6.395cm} =\sum\limits_{i,n_{1},n_{2}}\mathbb{E}\left( \displaystyle\int_{0}^{\theta}A_{r+\overline{\tau}}^{N}(e^{-(\theta-r)\lambda_{n_{1},n_{2}}}\rho^{i,s}_{n_{1},n_{2}})dr \right)
                                  \\\hspace*{6.395cm} \leq\mathbb{E}\left( \displaystyle\int_{0}^{\theta}A_{r+\overline{\tau}}^{N}(\sum\limits_{i,n_{1},n_{2}}\rho^{i,s}_{n_{1},n_{2}}) dr\right) \\\hspace*{6.395cm}\leq\theta C.\\
                                  $A similar argument yields $(6.20)$.
                  \end{proof}                               
            \begin{prp}
                   For all  $T>0$, \hspace*{0.12cm} \\\hspace*{5cm} $\displaystyle  \underset{N\geq 1}{\sup}\hspace*{0.04cm} \underset{0\leq t\leq T}{\sup} \mathbb{E}(\lVert U_{t}^{N} \lVert^{2}_{H^{-s}})<\infty, $ \vspace*{0.06cm}\\\hspace*{5cm}  $\displaystyle   \underset{N\geq 1}{\sup}\hspace*{0.04cm} \underset{0\leq t\leq T}{\sup} \mathbb{E}(\lVert V_{t}^{N} \lVert^{2}_{H^{-s}})<\infty. \hspace*{2.7cm}$\label{f8}.
                 \end{prp}        
                 \begin{proof} Choosing $u=0$ in equation (\ref{f71}) and (\ref{f72}), we get the estimates \hspace*{0.12cm}\\
                $ \displaystyle \lVert U_{t}^{N} \lVert_{H^{-s}}^{2}\leq \displaystyle 5 \lVert \Upsilon(t)U_{0}^{N}\lVert_{H^{-s}}^{2} + 5\beta^{2}t \int_{0}^{t}\lVert \Upsilon(t-r)(G_{r}^{S,I,N})^{*}Z_{r}^{N}\lVert_{H^{-s}}^{2}dr\\\hspace*{2cm}+ 5\beta^{2}t\int_{0}^{t}\lVert\Upsilon(t-r)(G_{r}^{I,N})^{*}U_{r}^{N}\lVert_{H^{-s}}^{2}dr +5 \beta^{2}t\int_{0}^{t} \lVert\Upsilon(t-r)(G_{r}^{S})^{*}V_{r}^{N}\lVert_{H^{-s}}^{2}dr\vspace*{0.12cm}\\\hspace*{2cm}+5\Big\lVert\int_{0}^{t}\Upsilon(t-r)d\widetilde{M}_{r}^{N}\Big\lVert_{H^{-s}}^{2},$ \hspace*{0.12cm}\\
                                                            $\lVert V_{t}^{N} \lVert_{H^{-s}}^{2}\leq\displaystyle6\lVert\Upsilon(t)V_{0}^{N} \lVert_{H^{-s}}^{2} +6 \beta^{2}t \int_{0}^{t}\lVert \Upsilon(t-r)(G_{r}^{S,I,N})^{*}Z_{r}^{N} \lVert_{H^{-s}}^{2}dr\\\hspace*{2cm} +6\beta^{2}t\int_{0}^{t}\lVert\Upsilon(t-r)(G_{r}^{I,N})^{*}U_{r}^{N} \lVert_{H^{-s}}^{2}dr +6t\beta^{2}\int_{0}^{t} \lVert\Upsilon(t-r)G_{r}^{S}V_{r}^{N} \lVert_{H^{-s}}^{2}dr\vspace*{0.12cm}\\\hspace*{2cm}+6t\alpha^{2}\int_{0}^{t} \lVert V_{r}^{N} \lVert_{H^{-s}}^{2}dr +6 \Big\lVert\int_{0}^{t}\Upsilon(t-r)d\widetilde{L}_{r}^{N} \Big\lVert_{H^{-s}}^{2}.$\vspace*{0.17cm}\\
           From Corollary \ref{ff2}, we have\\
                              $\displaystyle \lVert U_{t}^{N}\lVert^{2}_{H^{-s}}\leq 5\lVert U_{0}^{N} \lVert^{2}_{H^{-s}}+5t\beta^{2} C\underset{x}{\sup}\lVert K(x,.)\lVert_{H^{s}}^{2} \int_{0}^{t} \lVert U_{r}^{N} \lVert_{H^{-s}}dr \\\hspace*{0.5cm}+5Ct\beta^{2}\underset{x}{\sup}\Big\Arrowvert \frac{K(x,.)}{\int_{\mathbb{T}^{2}}K(x',.)\mu_{r}(dx')}\Big \Arrowvert_{H^{s}}^{2} \int_{0}^{t}\lVert V_{r}^{N} \lVert^{2}_{H^{-s}} dr+5t\beta^{2} C \underset{x}{\sup}\lVert K(x,.)\lVert_{H^{s}}^{2}\int_{0}^{t}\lVert Z_{r}^{N}\lVert^{2}_{H^{-s}}dr$ \hspace*{0.17cm}\\$\hspace*{0.5cm}+\Big\lVert\int_{0}^{t}\Upsilon(t-r)d\widetilde{M}_{r}^{N} \Big\lVert_{H^{-s}}^{2},$\\\\\\\\
                              $\displaystyle \lVert V_{t}^{N}\lVert^{2}_{H^{-s}}\leq 6\lVert V_{0}^{N} \lVert^{2}_{H^{-s}}+6t\beta^{2} C \underset{x}{\sup}\lVert K(x,.)\lVert_{H^{s}}^{2} \int_{0}^{t} \lVert U_{r}^{N} \lVert_{H^{-s}}dr \\\hspace*{2cm}+6t\Big(\beta^{2}C \underset{x}{\sup}\Big\Arrowvert \frac{K(x,.)}{\int_{\mathbb{T}^{2}}K(x',.)\mu_{r}(dx')}\Big \Arrowvert_{H^{s}}^{2}+\alpha^{2}
                                   \Big)\int_{0}^{t}\lVert V_{r}^{N} \lVert^{2}_{H^{-s}} dr\\\hspace*{2cm}+6t\beta^{2}C \underset{x}{\sup}\lVert K(x,.)\lVert_{H^{s}}^{2} \int_{0}^{t}\lVert Z_{r}^{N}\lVert^{2}_{H^{-s}}dr+\Big\lVert\int_{0}^{t}\Upsilon(t-r)d\widetilde{L}_{r}^{N} \Big\lVert_{H^{-s}}^{2}$.\vspace*{0.1cm}\\Thus  from Lemmas \ref{ape} in the Appendix below, \ref{ff3} and  Proposition \ref{f6}, we have \vspace*{0.2cm}\\
                                   $\displaystyle \underset{0\leq t\leq T}{\sup} \mathbb{E}(\lVert U_{t}^{N}\lVert^{2}_{H^{-s}})\leq 5\mathbb{E}(\lVert U_{0}^{N} \lVert_{H^{-s}}^{2})+5T\beta^{2}C\int_{0}^{T} \Big\{\underset{0\leq r\leq t}{\sup} \mathbb{E}(\lVert U_{r}^{N} \lVert^{2}_{H^{-s}})+\underset{0\leq r\leq t}{\sup}\mathbb{E}(\lVert V_{r}^{N} \lVert^{2}_{H^{-s}})\Big\}dt \\\hspace*{3.4cm}+5T^{2}\beta^{2} C\underset{0\leq t\leq T}{\sup}\mathbb{E}(\lVert Z_{t}^{N}\lVert_{H^{-s}})+CT,\hspace*{6.501cm}(6.21)$\\
                            $\displaystyle \underset{0\leq t\leq T}{\sup}\mathbb{E}(\lVert V_{t}^{N}\lVert^{2}_{H^{-s}})\leq 6\mathbb{E}(\lVert V_{0}^{N} \lVert^{2}_{H^{-s}})+6T(\beta^{2} C+\alpha^{2}) \int_{0}^{T} \Big\{\underset{0\leq r\leq t}{\sup} \mathbb{E}(\lVert U_{r}^{N} \lVert^{2}_{H^{-s}})+\underset{0\leq r\leq t}{\sup}\mathbb{E}(\lVert V_{r}^{N} \lVert^{2}_{H^{-s}})\Big\}dt\\\hspace*{3.4cm}+6T^{2}\beta^{2} C  \underset{0\leq t\leq T}{\sup}\mathbb{E}(\lVert Z_{t}^{N}\lVert^{2}_{H^{-s}})+CT. \hspace*{6.501cm}
                            (6.22) $ \vspace*{0.1cm}\\Hence
                             summing $(6.21)$ and $ (6.22)$ and applying Gronwall's lemma we deduce  the result from Proposition \ref{uo} in section 4 and Proposition \ref{d11}.
                 \end{proof}
                 Now we can establish the tightness of the sequences $(U^{N})_{N}$ and $(V^{N})_{N}.$                                                                       \begin{prp}
                                           Both sequences of processes $U^{N}$ and $V^{N}$ are tight in $D(\mathbb{R}_{+},H^{-s})$. \label{f10} 
                                                                             \end{prp} 
                                         \begin{proof}
       We only establish the tighness of $U^{N}$ by showing that the conditions of the Proposition \ref{d7} are satisfied.\vspace*{0.1cm}\\
      $-$ (T1)  is obtained by using  the Proposition \ref{f8} and applying an  argument similar to that of the  proof of (T1) in Theorem \ref{d6}. \vspace*{0.2cm}\\
      $-$ Proof of (T2). Let T>0, $\varepsilon_{1},\varepsilon_{2}$ >0, $(\tau^{N})_{N}$ a family of stopping times with $\tau^{N}\leq T$. \\ From equation (\ref{f71}), we have\\ $\displaystyle U_{\tau^{N}+\theta}^{N}-U_{\tau^{N}}^{N}=\displaystyle  (\Upsilon(\theta)-I_{d})U_{\tau^{N}}^{N} + \beta \int_{\tau^{N}}^{\tau^{N}+\theta}\Upsilon(\tau^{N}+\theta-r)(G_{r}^{S,I,N})^{*}Z_{r}^{N}dr \\\hspace*{2.3cm}-\beta\int_{\tau^{N}}^{\tau^{N}+\theta}\Upsilon(\tau^{N}+\theta-r)(G_{r}^{I,N})^{*}U_{r}^{N}dr  - \beta\int_{\tau^{N}}^{\tau^{N}+\theta} \Upsilon(\tau^{N}+\theta-r)(G_{r}^{S})^{*}V_{r}^{N}dr\\\hspace*{2.3cm}+\int_{\tau^{N}}^{\tau^{N}+\theta}\Upsilon(\tau^{N}+\theta-r)d\widetilde{M}_{r}^{N}$
     \\\hspace*{2.2cm}= $(\Upsilon(\theta)-I_{d})U_{\tau^{N}}^{N} + \displaystyle\beta\displaystyle \int_{\tau^{N}}^{\tau^{N}+\theta}\Upsilon(\tau^{N}+\theta-r)J_{r}^{S,I,N}(Z^{N},U^{N},V^{N})dr\\\hspace*{2.3cm}+\int_{\tau^{N}}^{\tau^{N}+\theta}\Upsilon(\tau^{N}+\theta-r)d\widetilde{M}_{r}^{N},$ \vspace*{0.3cm}\\where $ J_{r}^{S,I,N}(Z^{N},U^{N},V^{N})=(G_{r}^{S,I,N})^{*}Z_{r}^{N}-(G_{r}^{I,N})^{*}U_{r}^{N}-(G_{r}^{S})^{*}V_{r}^{N}$. \vspace*{0.2cm}\\
      We find $\delta>0$ and $N_{0}\geq 1$ such that  \vspace*{0.3cm}\\\hspace*{1cm} 
                      $\underset{N\geq N_{0}}{\sup}\underset{\delta\geq\theta}{\sup}\mathbb{P}\left(\left\lVert (\Upsilon(\theta)-I_{d})U_{\tau^{N}}^{N} \right\lVert_{H^{-s}}\geq\varepsilon_{1}\right)\leq \varepsilon_{2},\hspace*{6.9cm}(6.23)$ \\\hspace*{1cm} $\underset{N\geq N_{0}}{\sup}\underset{\delta\geq\theta}{\sup}\mathbb{P}\left(\left\lVert \displaystyle\beta \int_{\tau^{N}}^{\tau^{N}+\theta}\Upsilon(\tau^{N}+\theta-r)J_{r}^{S,I,N}(Z^{N},U^{N},V^{N})dr \right\lVert_{H^{-s}}\geq\varepsilon_{1}\right)\leq \varepsilon_{2},\hspace*{1.08cm}(6.24)$ \\\hspace*{1cm} $ \underset{N\geq N_{0}}{\sup}\underset{\delta\geq\theta}{\sup}\mathbb{P}\left(\left\lVert\displaystyle \int_{\tau^{N}}^{\tau^{N}+\theta}\Upsilon(\tau^{N}+\theta-r)d\widetilde{M}_{r}^{N}\right\lVert_{H^{-s}}\geq\varepsilon_{1}\right)\leq \varepsilon_{2}.\hspace*{4.57cm}(6.25)$\\$- $  $(6.23)$ is proved by similar reasoning to that of the proof of $(6.5)$ in  Proposition \ref{ff4}.\vspace*{0.2cm} \\$-$  Proof of $(6.24)$.  Let $l\in \mathbb{R}_{+}\backslash[0,1]$, we find $\delta>0$ such that $\tau^{N}+\delta\leq lT$ and such that $(6.24)$ is satisfied. Since $\forall \varphi \in H^{s}$, $\lVert \Upsilon(\tau^{N}+\theta-r)\varphi \lVert_{H^{s}}\leq \lVert \varphi \lVert_{H^{s}}$, from  Corollary \ref{ff2}, Lemmas \ref{ape} below, \ref{ff3} and Propositions \ref{d11} and  \ref{f8}, we have \\
                     
                    $\mathbb{P}\left(\left\lVert \displaystyle\beta \int_{\tau^{N}}^{\tau^{N}+\theta}\Upsilon(\tau^{N}+\theta-r)J_{r}^{S,I,N}(Z^{N},U^{N},V^{N})dr \right\lVert_{H^{-s}}\geq\varepsilon_{1}\right)$\\ \vspace*{-0.2cm}\begin{align*}
                    \hspace*{2cm}&\leq \frac{\beta^{2}}{\varepsilon_{1}^{2}}\mathbb{E}\left(\left\lVert  \displaystyle\int_{\tau^{N}}^{\tau^{N}+\theta}\Upsilon(\tau^{N}+\theta-r)J_{r}^{S,I,N}(Z^{N},U^{N},V^{N})dr\right\lVert_{H^{-s}}^{2}\right)
                     \\&\leq \frac{\beta^{2}\theta}{\varepsilon_{1}^{2}}\displaystyle\mathbb{E}\left(\int_{\tau^{N}}^{\tau^{N}+\theta}\lVert  \Upsilon(\tau^{N}+\theta-r)J_{r}^{S,I,N}(Z^{N},U^{N},V^{N})\lVert_{H^{-s}}^{2}dr\right)\\&
                     \leq\frac{\beta^{2}\theta C}{\varepsilon_{1}^{2}} \underset{y}{\sup}\lVert K(.,y)\lVert_{H^{s}}\displaystyle\mathbb{E}\left(\int_{\tau^{N}}^{\tau^{N}+\theta}\{\lVert Z^{N}_{r} \lVert_{H^{-s}}^{2}+ \lVert U^{N}_{r} \lVert_{H^{-s}}^{2}\}dr\right) \\&
                                          +\frac{\beta^{2}\theta C}{\varepsilon_{1}^{2}} \underset{x}{\sup}\Big\Arrowvert \frac{K(x,.)}{\int_{\mathbb{T}^{2}}K(x',.)\mu_{r}(dx')}\Big \Arrowvert_{H^{s}}\displaystyle\mathbb{E}\left(\int_{\tau^{N}}^{\tau^{N}+\theta}\lVert V^{N}_{r} \lVert_{H^{-s}}^{2} dr\right) \\&\leq \frac{\beta^{2}\delta^{2} C}{\varepsilon_{1}^{2}} \underset{N\geq1}{\sup}\hspace*{0.05cm}\underset{0\leq t\leq lT}{\sup} \mathbb{E}(\{\lVert Z^{N}_{t}\lVert_{H^{-s}}^{2}+ \lVert U^{N}_{t} \lVert_{H^{-s}}^{2}+ \lVert V^{N}_{t} \lVert_{H^{-s}}^{2}\}) \\&\leq \frac{\beta^{2}\delta^{2} C}{\varepsilon_{1}^{2}}.                   
                       \end{align*}
                    So $(6.24)$ is proved. \vspace*{0.2cm}\\
                   $-$ Proof of $(6.25).$ From result (6.19) in Proposition \ref{f6} we have,\vspace*{0.2cm}\\
                   $\displaystyle\mathbb{P}\left(\left\lVert \int_{\tau^{N}}^{\tau^{N}+\theta}\Upsilon(\tau^{N}+\theta-r)d\widetilde{M}_{r}^{N}\right\lVert_{H^{-s}}\geq\varepsilon_{1}\right)$ 
                   $\leq \frac{1}{\varepsilon_{1}^{2}}\displaystyle\mathbb{E}\left(\left\lVert \int_{\tau^{N}}^{\tau^{N}+\theta}\Upsilon(\tau^{N}+\theta-r)d\widetilde{M}_{r}^{N} \right\lVert_{H^{-s}}^{2}\right) \\\hspace*{7.88cm}\leq \frac{1}{\varepsilon_{1}^{2}}\delta C,
                   $
                    \\hence $(6.25)$ is proved.
                                                                             \end{proof}
  To  show that all converging subsequences of $(U^{N},V^{N})_{N\geq 1} $ have the same limit we will need the next two Lemmas.
 \begin{lem}
  For any $ t\geq0$, $\varphi\in H^{3}(\mathbb{T}^{2})$,  as $N\longrightarrow \infty$, \\\hspace*{5cm} $ \displaystyle\int_{0}^{t}\mathbb{E}\Big(\lVert [G_{r}^{I,N}-G_{r}^{I}]\Upsilon(t-r)\varphi \lVert_{H^{s}}^{2}\Big)^{\frac{1}{2}}dt\longrightarrow 0.$\label{f11}
 \end{lem}
\begin{proof}
Since $s>3$, $H^{3}\hookrightarrow H^{s}$,  thus\\
   \hspace*{2cm} $\displaystyle\int_{0}^{t}\mathbb{E}\Big(\lVert [G_{r}^{I,N}-G_{r}^{I}]\Upsilon(t-r)\varphi \lVert_{H^{s}}^{2}\Big)^{\frac{1}{2}}dt\leq C\displaystyle\int_{0}^{t}\mathbb{E}\Big(\lVert [G_{r}^{I,N}-G_{r}^{I}]\Upsilon(t-r)\varphi \lVert_{H^{3}}^{2}\Big)^{\frac{1}{2}}dt.$\\ Furtemore as $H^{3}$ is a Banach algebra  (see Proposition \ref{b4})  and  $\lVert \Upsilon(t)\varphi\lVert_{H^{3}}\leq C \lVert \varphi \lVert_{H^{3}}$, \vspace*{0.2cm} \\
     \hspace*{2cm}$\lVert [G_{r}^{I,N}-G_{r}^{I}]\Upsilon(t-r)\varphi \lVert_{H^{3}}=\Big\lVert \Upsilon(t-r)\varphi \displaystyle\Big(\mu_{r}^{I,N}-\mu_{r}^{I},\frac{K}{(\mu_{r},K)}\Big)\Big\lVert_{H^{3}}\\\hspace*{6.59cm}\leq C\lVert \varphi \lVert_{H^{3}}\Big\lVert \displaystyle \Big(\mu_{r}^{I,N}-\mu_{r}^{I},\frac{K}{(\mu_{r},K)}\Big) \Big\lVert_{H^{3}}.$ \vspace*{0.2cm}\\ If we let 
      support$\{K(x,\cdot)\}=A(x)$, we will have \\
      $\Big\lVert \displaystyle\Big(\mu_{r}^{I,N}-\mu_{r}^{I},\frac{K}{(\mu_{r},K)}\Big) \Big\lVert_{H^{3}}^{2}=\displaystyle\sum\limits_{\lvert\eta\lvert\leq 3}\int_{\mathbb{T}^{2}}\Big\lvert\int_{A(x)} \frac{D^{\eta}K(x,y)}{\int_{\mathbb{T}^{2}}K(x',y)\mu_{r}(dx')}(\mu_{r}^{I,N}-\mu_{r}^{I})(dy)\Big\lvert^{2}dx,$\vspace*{0.2cm}\\
      furthemore from Remark \ref{f2f}  the map $y\in A(x)\mapsto D^{\eta}K(x,y)$ is continuous and bounded by $C\underset{0\leq \lvert\eta\lvert\leq 3}{\max}\lVert k^{(\lvert \eta\lvert)}\lVert_{\infty}$. 
   Thus as the map $y\in \mathbb{T}^{2}\mapsto\displaystyle \int_{\mathbb{T}^{2}}K(x',y)\mu_{r}(dx')$ is also continuous and lower bounded by a positive constant, the map $y\in A(x)\mapsto \displaystyle \frac{D^{\eta}K(x,y)}{\int_{\mathbb{T}^{2}}K(x',y)\mu_{r}(dx')}$ is continuous and bounded by $C\underset{0\leq \lvert\eta\lvert\leq 3}{\max}\lVert k^{(\lvert \eta\lvert)}\lVert_{\infty}$. So we deduce from  Theorem \ref{eeee} that \vspace*{0.1cm}\\\hspace*{3.2cm} $\displaystyle\Big\lvert\int_{A(x)} \frac{D^{\eta}K(x,y)}{\int_{\mathbb{T}^{2}}K(x',y)\mu_{r}(dx')}(\mu_{r}^{I,N}-\mu_{r}^{I})(dy)\Big\lvert^{2}\xrightarrow{P}0$.\\ According to Lebesgue's dominated convergence theorem, $\mathbb{E}\Big(\Big\lVert \displaystyle\Big(\mu_{r}^{I,N}-\mu_{r}^{I},\frac{K}{(\mu_{r},K)}\Big) \Big\lVert_{H^{3}}^{2}\Big)\longrightarrow 0,$ as $N\longrightarrow\infty$. Hence\vspace*{0.2cm}\\ $\hspace*{2.5cm}\displaystyle\int_{0}^{t}\mathbb{E}\Big(\lVert [G_{r}^{I,N}-G_{r}^{I}]\Upsilon(t-r)\varphi \lVert_{H^{3}}^{2}\Big)^{\frac{1}{2}}dt\\\hspace*{5cm}\leq C\lVert  \varphi \lVert_{H^{3}}\displaystyle\int_{0}^{t}\mathbb{E}\Big(\Big\lVert \displaystyle\Big(\mu_{r}^{I,N}-\mu_{r}^{I},\frac{K}{(\mu_{r},K)}\Big) \Big\lVert_{H^{3}}^{2}\Big)^{\frac{1}{2}}dr\vspace*{0.1cm}\\\hspace*{5cm}\longrightarrow0,$ as $N\longrightarrow\infty.$\\
   Hence the result.
\end{proof}
\begin{lem}
  For any $ t\geq0$, $\varphi\in H^{s}(\mathbb{T}^{2})$,  as $N\longrightarrow \infty$, \\\hspace*{5cm} $ \displaystyle\int_{0}^{t}\mathbb{E}\Big(\lVert [G_{r}^{S,I,N}-G_{r}^{S,I}]\Upsilon(t-r)\varphi \lVert_{H^{s}}^{2}\Big)^\frac{1}{2}dt\longrightarrow 0.$\label{f12}
 \end{lem}
\begin{proof}
We  have \vspace*{0.15cm}\\ $G_{r}^{S,I,N}(\Upsilon(t-r)\varphi)(x)-G_{r}^{S,I}(\Upsilon(t-r)\varphi)(x)$\vspace*{0.2cm}\\\hspace*{1.3cm}=$\displaystyle \left(\mu_{r}^{I,N}, K(x,.)\frac{(\mu_{r}^{S,N},K \Upsilon(t-r)\varphi )}{(\mu^{N}_{r},K)(\mu_{r},K)}\right)-\left(\mu_{r}^{I}, K(x,.)\frac{(\mu_{r}^{S},K \Upsilon(t-r)\varphi )}{(\mu_{r},K)(\mu_{r},K)} \right)\vspace*{0.1cm}\\\hspace*{1.3cm}=\displaystyle \left(\mu_{r}^{I,N}-\mu_{r}^{I}, K(x,.)\frac{(\mu_{r}^{S},K \Upsilon(t-r)\varphi )}{(\mu_{r},K)^{2}} \right)+\displaystyle \left(\mu_{r}^{I,N}, K(x,.)\frac{(\mu_{r}^{S,N}-\mu_{r}^{S},K \Upsilon(t-r)\varphi)}{(\mu_{r},K)^{2}} \right)\vspace*{0.1cm}\\\hspace*{1.3cm}- \displaystyle \left(\mu_{r}^{I,N}, K(x,.)\frac{(\mu_{r}^{S,N},K \Upsilon(t-r)\varphi )}{(\mu^{N}_{r},K)(\mu_{r},K)^{2}}(\mu^{N}_{r}-\mu_{r},K) \right). $\vspace*{0.2cm}\\
Furthemore:\vspace*{0.2cm}\\
 a) Since :\\  $-$ $\Big\lvert \displaystyle\frac{\int_{\mathbb{T}^{2}}K(x',y)\Upsilon(t-r)\varphi(x')\mu_{r}^{S}(dx')}{\int_{\mathbb{T}^{2}}K(x'',y)\mu_{r}(dx'')} \Big\lvert \leq \lVert \varphi \lVert_{\infty}$\\ $-$   The map $y\in \mathbb{T}^{2}\mapsto\displaystyle \int_{\mathbb{T}^{2}}K(x',y)\mu_{r}(dx')$ is  continuous and lower bounded by a positive \\\hspace*{0.5cm} constant.  \\ $-$ Under (H2), The map $y\in \mathbb{T}^{2}\mapsto\displaystyle\int_{\mathbb{T}^{2}}K(x',y)\Upsilon(t-r)\varphi(x')\mu_{r}^{S}(dx')$ is continuous \\ $-$ From Remark \ref{f2f}, for any $ x\in \mathbb{T}^{2}$, $ \lvert \eta\lvert\leq 3,$   the map   $y\in A(x)\mapsto \displaystyle D^{\eta}K(x,y)$ is continuous \hspace*{0.5cm} and bounded by $C\underset{0\leq \lvert\eta\lvert\leq 3}{\max}\lVert k^{(\lvert \eta\lvert)}\lVert_{\infty}$, \\the map $y\in A(x)\mapsto \displaystyle D^{\eta}K(x,y)\times\displaystyle\frac{\int_{\mathbb{T}^{2}}K(x',y)\Upsilon(t-r)\varphi(x')\mu_{r}^{S}(dx')}{\Big(\int_{\mathbb{T}^{2}}K(x'',y)\mu_{r}(dx'')\Big)^{2}}$ is continous and bounded  by $C\lVert \varphi\lVert_{\infty}\underset{0\leq \lvert\eta\lvert\leq 3}{\max}\lVert k^{(\lvert \eta\lvert)}\lVert_{\infty}$, hence\\\hspace*{2cm}$\Big \lvert\displaystyle\int_{A(x)}D^{\eta}K(x,y)\times\displaystyle\frac{\int_{\mathbb{T}^{2}}K(x',y)\Upsilon(t-r)\varphi(x')\mu_{r}^{S}(dx')}{\Big(\int_{\mathbb{T}^{2}}K(x'',y)\mu_{r}(dx'')\Big)^{2}}(\mu_{r}^{I,N}-\mu_{r}^{I})(dy)\Big\lvert^{2}\xrightarrow{P}0.$ \\ Furthemore, according to Lebesgue's  dominated convergence theorem, we have \vspace*{0.1cm}\\
   $\hspace*{1cm}\displaystyle\int_{0}^{t}\mathbb{E}\Big(\Big\lVert \left(\mu_{r}^{I,N}-\mu_{r}^{I}, K\frac{(\mu_{r}^{S},K \Upsilon(t-r)\varphi )}{(\mu_{r},K)^{2}} \right) \Big\lVert_{H^{s}}^{2}\Big)^\frac{1}{2}dr\\\hspace*{2cm}\leq C \displaystyle \int_{0}^{t} \mathbb{E}\Big(\Big\lVert \left(\mu_{r}^{I,N}-\mu_{r}^{I}, K\frac{(\mu_{r}^{S},K \Upsilon(t-r)\varphi )}{(\mu_{r},K)^{2}} \right) \Big\lVert_{H^{3}}^{2}\Big)^{\frac{1}{2}}dr\\\hspace*{2cm}\leq C \displaystyle\int_{0}^{t}\Big\{\sum\limits_{\lvert\eta\lvert\leq3}\int_{\mathbb{T}^{2}}\mathbb{E}\Big(\Big\lvert\int_{A(x)}D^{\eta}K(x,y)\times\\\hspace*{5cm}\times\displaystyle\frac{\int_{\mathbb{T}^{2}}K(x',y)\Upsilon(t-r)\varphi(x')\mu_{r}^{S}(dx')}{\Big(\int_{\mathbb{T}^{2}}K(x'',y)\mu_{r}(dx'')\Big)^{2}}(\mu_{r}^{I,N}-\mu_{r}^{I})(dy)\Big\lvert^{2}\Big)dx\Big\}^{\frac{1}{2}}dr\vspace*{0.1cm}\\\hspace*{2cm}\longrightarrow0,$ as $N\longrightarrow\infty.$ \vspace*{0.17cm}\\b) Let us prove that $\lim\limits_{N\longrightarrow \infty}\displaystyle\int_{0}^{t}\mathbb{E}\Big(\Big\lVert \displaystyle \left(\mu_{r}^{I,N}, K\frac{(\mu_{r}^{S,N}-\mu_{r}^{S},K \Upsilon(t-r)\varphi)}{(\mu_{r},K)^{2}}\right) \Big\lVert_{H^{s}}^{2}\Big)^{\frac{1}{2}}dt=0.$
                     \\Under (H2):\vspace*{0.1cm}\\\hspace*{0.3cm}$-$ $\forall y\in \mathbb{T}^{2}$,  the map $x\in \mathbb{T}^{2} \mapsto \displaystyle K(x,y)\Upsilon(t-r)\varphi(x)$ is Lipschitz and bounded and its $\lVert \cdot\lVert_{\infty}$ \\\hspace*{0.7cm} and $\lVert  \cdot \lVert_{L}$ norms are bounded by $C\lVert k \lVert_{\infty}\lVert \varphi \lVert_{\infty} $ and $C(\lVert k \lVert_{\infty}+\lVert \varphi \lVert_{\infty} )$ respectively. \\\hspace*{0.3cm}$-$  For any $ x\in \mathbb{T}^{2}$, $ \lvert \eta\lvert\leq 3,$   the map   $y\in A(x)\mapsto \displaystyle D^{\eta}K(x,y)$ is continuous  and bounded by \\\hspace*{0.7cm} $C\underset{0\leq \lvert\eta\lvert\leq 3}{\max}\lVert k^{(\lvert \eta\lvert)}\lVert_{\infty}$. \\\hspace*{0.3cm}$-$ $\forall y \in \mathbb{T}^{2},$ $\displaystyle\frac{1}{(\mu_{r},K(.,y))}$ is bounded by a positive contant independent of $y$. \vspace*{0.2cm} \\Thus\\
           $ \Big\lVert \displaystyle \left(\mu_{r}^{I,N}, K\frac{(\mu_{r}^{S,N}-\mu_{r}^{S},K \Upsilon(t-r)\varphi)}{(\mu_{r},K)^{2}} \right) \Big\lVert_{H^{s}}^{2}\leq C\Big\lVert \displaystyle \left(\mu_{r}^{I,N}, K\frac{(\mu_{r}^{S,N}-\mu_{r}^{S},K \Upsilon(t-r)\varphi)}{(\mu_{r},K)^{2}}\right) \Big\lVert_{H^{3}}^{2}$\vspace*{0.2cm}\\\hspace*{1cm} $=C \displaystyle\sum\limits_{\lvert\eta\lvert\leq 3}\int_{\mathbb{T}^{2}}\Big\lvert\int_{A(x)}D^{\eta}K(x,y)\times\displaystyle\frac{\int_{\mathbb{T}^{2}}K(x',y)\Upsilon(t-r)\varphi(x')(\mu_{r}^{S,N}-\mu_{r}^{S})(dx')}{\Big(\int_{\mathbb{T}^{2}}K(x'',y)\mu_{r}(dx'')\Big)^{2}}\mu^{I,N}_{r}(dy)\Big\lvert^{2}dx \vspace*{0.2cm}\\\hspace*{1cm} \leq C \displaystyle\sum\limits_{\lvert\eta\lvert\leq 3}\int_{\mathbb{T}^{2}}\int_{A(x)}\Big\lvert D^{\eta}K(x,y)\times\int_{\mathbb{T}^{2}}K(x',y)\Upsilon(t-r)\varphi(x') (\mu_{r}^{S,N}-\mu_{r}^{S})(dx')\Big\lvert^{2}\mu^{I,N}_{r}(dy)dx \vspace*{0.2cm}\\\hspace*{1cm}\leq C(
           \lVert k \lVert_{\infty},\lVert \varphi \lVert_{\infty},\underset{0\leq \lvert\eta\lvert\leq 3}{\max}\lVert k^{(\lvert \eta\lvert)}\lVert_{\infty})d_{F}^{2}(\mu_{r}^{S,N},\mu_{r}^{S}).$\vspace*{0.12cm}\\ Hence since $d_{F}^{2}(\mu_{r}^{S,N},\mu_{r}^{S})\xrightarrow{P}0$, and  $d_{F}^{2}(\mu_{r}^{S,N},\mu_{r}^{S})\leq 4$,  according to Lebesgue's  dominated convergence theorem, we have \vspace*{0.13cm}\\\hspace*{0.5cm}$\displaystyle\int_{0}^{t} \mathbb{E}\Big(\Big\lVert \displaystyle \left(\mu_{r}^{I,N}, K\frac{(\mu_{r}^{S,N}-\mu_{r}^{S},K \Upsilon(t-r)\varphi)}{(\mu_{r},K)^{2}}\right) \Big\lVert_{H^{s}}^{2}\Big)^{\frac{1}{2}}dt\\\hspace*{2.5cm}\leq C(\lVert \varphi \lVert_{\infty},\underset{0\leq \lvert\eta\lvert\leq 3}{\max}\lVert k^{(\lvert \eta\lvert)}\lVert_{\infty})\displaystyle\int_{0}^{t}\mathbb{E}(d_{F}^{2}(\mu_{r}^{S,N},\mu_{r}^{S}))^{\frac{1}{2}}dr\vspace*{0.1cm}\\\hspace*{2.5cm}\longrightarrow0,$ as $N\longrightarrow\infty.$\vspace*{0.18cm}\\
        c) Finaly we show that $\lim\limits_{N\longrightarrow\infty}\displaystyle\int_{0}^{t}\mathbb{E}\Big(\Big\lVert \displaystyle \left(\mu_{r}^{I,N}, K\frac{(\mu_{r}^{S,N},K \Upsilon(t-r)\varphi )}{(\mu^{N}_{r},K)(\mu_{r},K)^{2}}(\mu^{N}_{r}-\mu_{r},K) \right) \Big\lVert_{H^{s}}^{2}\Big)^{\frac{1}{2}}dt=0.$
       \vspace*{0.05cm} \\Since for all $ x\in \mathbb{T}^{2}$, $K(x,.)$ is Lipschitz and bounded by $\lVert k\lVert_{\infty}$  and  For any $ x\in \mathbb{T}^{2}$, $ \lvert \eta\lvert\leq 3,$   the map   $y\in A(x)\mapsto \displaystyle D^{\eta}K(x,y)$ is bounded  by  $C\underset{0\leq \lvert\eta\lvert\leq 3}{\max}\lVert k^{(\lvert \eta\lvert)}\lVert_{\infty}$, we have \\
       $ \Big\lVert \displaystyle \left(\mu_{r}^{I,N}, K\frac{(\mu_{r}^{S,N},K \Upsilon(t-r)\varphi )}{(\mu^{N}_{r},K)(\mu_{r},K)^{2}}(\mu^{N}_{r}-\mu_{r},K) \right) \Big\lVert_{H^{s}}^{2}\vspace*{0.13cm}$\\ $\hspace*{1cm}\leq C\Big\lVert \displaystyle \left(\mu_{r}^{I,N}, K\frac{(\mu_{r}^{S,N},K \Upsilon(t-r)\varphi )}{(\mu^{N}_{r},K)(\mu_{r},K)^{2}}(\mu^{N}_{r}-\mu_{r},K) \right)\Big \lVert_{H^{3}}^{2}$\vspace*{0.18cm}\\ $\hspace*{1cm}\leq C \displaystyle\sum\limits_{\lvert\eta\lvert\leq 3}\int_{\mathbb{T}^{2}}\Big(\int_{A(x)}\lvert D^{\eta}K(x,y)\vert\times\left\lvert\frac{\int_{\mathbb{T}^{2}}\Upsilon(t-r)\varphi(x')K(x',y)\mu_{r}^{S,N}(dx')}{\int_{\mathbb{T}^{2}}K(u,y)\mu_{r}^{N}(du)\Big(\int_{\mathbb{T}^{2}}K(u,y)\mu_{r}(du)\Big)^{2}}\right\lvert\vspace*{0.13cm}\\\hspace*{9cm} \times \left\lvert \int_{\mathbb{T}^{2}}K(u,y)(\mu_{r}^{N}-\mu_{r})(du)\right \lvert \mu_{r}^{I,N}(dy)\Big)^{2}dx \vspace*{0.13cm}\\\hspace*{1cm}\leq C\lVert \varphi \lVert_{\infty}^{2}  \displaystyle\sum\limits_{\lvert\eta\lvert\leq 3}\int_{\mathbb{T}^{2}}\Big(\int_{A(x)}\lvert D^{\eta}K(x,y)\lvert\times\left\lvert \int_{\mathbb{T}^{2}}K(u,y)(\mu_{r}^{N}-\mu_{r})(du)\right \lvert \mu_{r}^{I,N}(dy)\Big)^{2}dx \vspace*{0.13cm}\\\hspace*{1cm}\leq  C(\lVert k\lVert_{\infty},\lVert \varphi \lVert_{\infty},\underset{0\leq \lvert\eta\lvert\leq 3}{\max}\lVert k^{(\lvert \eta\lvert)}\lVert_{\infty})d_{F}^{2}(\mu_{r}^{N},\mu_{r}).$\\Thus\vspace*{0.1cm}\\\hspace*{1cm}$\displaystyle\int_{0}^{t}\mathbb{E}\Big(\Big\lVert \displaystyle \left(\mu_{r}^{I,N}, K\frac{(\mu_{r}^{S,N},K \Upsilon(t-r)\varphi )}{(\mu^{N}_{r},K)(\mu_{r},K)^{2}}(\mu^{N}_{r}-\mu_{r},K) \right) \Big\lVert_{H^{s-1}}^{2}\Big)^{\frac{1}{2}}dt\vspace*{0.13cm}\\\hspace*{4cm}\leq C(
        \lVert k\lVert_{\infty},\lVert \varphi \lVert_{\infty},\underset{0\leq \lvert\eta\lvert\leq 3}{\max}\lVert k^{(\lvert \eta\lvert)}\lVert_{\infty})\displaystyle\int_{0}^{t}\mathbb{E}(d_{F}^{2}(\mu_{r}^{N},\mu_{r}))^{\frac{1}{2}}dr\vspace*{0.1cm}\\\hspace*{4cm}\longrightarrow0,$ as $N\longrightarrow\infty.$
   \end{proof}
 The next Proposition establish the evolution equations of all limit points  $ (U,V)$  of the sequence $ (U^{N},V^{N})$.
      \begin{prp}
    Any limit point  $ (U,V)$  of the sequence $ (U^{N},V^{N})$  satisfies \vspace*{0.2cm}\\
           $\displaystyle U_{t}=\displaystyle  \Upsilon(t)U_{0} + \beta \int_{0}^{t}\Upsilon(t-r)(G_{r}^{S,I})^{*}Z_{r}dr -\beta\int_{0}^{t}\Upsilon(t-r)(G_{r}^{I})^{*}U_{r}dr - \beta\int_{0}^{t}\Upsilon(t-r)(G_{r}^{S})^{*} V_{r}dr\\\hspace*{1cm} + \int_{0}^{t}\Upsilon(t-r)W_{r}^{1},\hspace*{11.8cm}(6.26)$ \vspace*{0.1cm}\\ $ \displaystyle V_{t}=\displaystyle  \Upsilon(t)V_{0} -\beta \int_{0}^{t}\Upsilon(t-r)(G_{r}^{S,I})^{*}Z_{r}dr +\beta\int_{0}^{t}\Upsilon(t-r)(G_{r}^{I})^{*}U_{r} dr + \beta\int_{0}^{t}\Upsilon(t-r)(G_{r}^{S})^{*}V_{r}dr \\\hspace*{1cm}-\alpha\int_{0}^{t}\Upsilon(t-r)V_{r}dr+\int_{0}^{t}\Upsilon(t-r)dW_{r}^{2}. \hspace*{7.8cm}(6.27)$ \label{llln}
                                                          \end{prp}
                                                          \begin{proof}
   We prove this Proposition by taking the weak limit   in the equations  (\ref{f71}) and (\ref{f72}). \\Note first that from Propositions \ref{ff4} and \ref{f10} there exists a subsequence along which the sequences $(U^{N},V^{N})_{N}$,    and $(Z^{N})_{N}$ converge in law towards  $(U,V)$ and Z respectively. \\
   1- We first prove that $\displaystyle\int_{0}^{t}(U_{r}^{N},G_{r}^{I,N}\Upsilon(t-r)\varphi )dr\xrightarrow{L}\int_{0}^{t}(U_{r},G_{r}^{I}\Upsilon(t-r)\varphi )dr.$\\We have $\displaystyle\int_{0}^{t}\left(U_{r}^{N} ,G_{r}^{I,N}\Upsilon(t-r)\varphi \right)dr=\displaystyle\int_{0}^{t}\left(U_{r}^{N} ,G_{r}^{I}\Upsilon(t-r)\varphi \right)dr\\\hspace*{7cm}+\displaystyle\int_{0}^{t}\left(U_{r}^{N} ,[G_{r}^{I,N}-G_{r}^{I}]\Upsilon(t-r)\varphi \right)dr.$ \\$-$ Since  $ G_{r}^{I}\Upsilon(t-r)\varphi\in H^{s} $, \\\hspace*{4cm} $\displaystyle\int_{0}^{t}\left(U_{r}^{N} ,G_{r}^{I}\Upsilon(t-r)\varphi \right)dr\xrightarrow{L}\int_{0}^{t}\left(U_{r} ,G_{r}^{I}\Upsilon(t-r)\varphi \right)dr.$\\$-$ From Lemma \ref{f11} and  Proposition \ref{f8} $\displaystyle\int_{0}^{t}\left(U_{r}^{N} ,[G_{r}^{I,N}-G_{r}^{I}]\Upsilon(t-r)\varphi \right)dr\longrightarrow 0$ in $L^{1}(\mathbb{P})$. Indeed, $\mathbb{E}\left(\Big \lvert \displaystyle\int_{0}^{t}\left(U_{r}^{N} ,[G_{r}^{I,N}-G_{r}^{I}]\Upsilon(t-r)\varphi \right)dr \Big \lvert\right)\leq\\\hspace*{5cm}\leq \underset{N\geq 1}{\sup}\hspace*{0.04cm}\underset{0\leq r \leq t}{\sup}\mathbb{E}(\lVert U_{r}^{N}\lVert_{H^{-s}}^{2})^{\frac{1}{2}}\displaystyle\int_{0}^{t}\mathbb{E}(\lVert [G_{r}^{I,N}-G_{r}^{I}]\Upsilon(t-r)\varphi\lVert_{H^{s}}^{2})^{\frac{1}{2}}dr.
   $\vspace*{0.12cm}\\
2- Since $V^{N}$  converges in law 
                          in $D(\mathbb{R}_{+},H^{-s})$ towards $V,$ \\\hspace*{3cm}
                          $\displaystyle\int_{0}^{t}(\Upsilon(t-r)(G_{r}^{S})^{*}V_{r}^{N},\varphi )dr\xrightarrow{L}\int_{0}^{t}(\Upsilon(t-r)(G_{r}^{S})^{*}V_{r},\varphi )dr. $ \vspace*{0.2cm}\\
                          3- The convergence $\displaystyle\int_{0}^{t}(Z_{r}^{N},G_{r}^{S,I,N}\Upsilon(t-r)\varphi )dr\xrightarrow{L}\int_{0}^{t}(Z_{r},G_{r}^{S,I}\Upsilon(t-r)\varphi )dr $ follows from the same argument as the one in 1-, using this time Lemma \ref{f12} and Proposition \ref{d11}.\vspace*{0.2cm}\\
4- Since for any $\varphi\in H^{s+1}$, $\Upsilon(t)\varphi\in H^{s+1}$, 
 from  Theorem \ref{lo} and Proposition \ref{f4}, we have \\\hspace*{2cm} $(U_{0}^{N},\Upsilon(t)\varphi)\xrightarrow{L}(U_{0},\Upsilon(t)\varphi)$
 and $\displaystyle\int_{0}^{t}(\Upsilon(t-r)\varphi, d\widetilde{M}_{r}^{N})\xrightarrow{L}\displaystyle\int_{0}^{t}(\Upsilon(t-r)\varphi, dW^{1}_{r}).
$
\end{proof}
From Proposition \ref{f2} we deduce that  all limit points $(U,V)$ of $ (U^{N},V^{N})_{N\geq1} $ are elements of  $( C(\mathbb{R}_{+},H^{-s}))^{2},$ thus since it is so easy to see that  equations (6.26) and (6.27) are the same as equations (\ref{df1}) and (\ref{df2}) repectively, we end the proof of Theorem \ref{ff1} by showing that   the system formed by the equations (6.26) and (6.27) admits a unique solution $(U,V)\in( C(\mathbb{R}_{+},H^{-s}))^{2}$.

       \begin{prp}
       Suppose that $(U^{1},V^{1})$  and $ (U^{2},V^{2})$  both belong to $(C(\mathbb{R}_{+},H^{-s}))^{2}$ and  are solutions to equations (6.26) and (6.27) with $ (U_{0}^{1},V_{0}^{1})=(U_{0}^{2},V_{0}^{2})$  then  $(U^{1},V^{1})$ = $ (U^{2},V^{2})$ \label{rt1}
       \end{prp}
       \begin{proof}
      We have \\
                   $\hspace*{0.5cm}U_{t}^{1}-U_{t}^{2}=\displaystyle-\beta\displaystyle \int_{0}^{t} \Upsilon(t-r)(G_{r}^{I})^{*}(U_{r}^{1}-U_{r}^{2})dr-\beta\int_{0}^{t}\Upsilon(t-r)(G_{r}^{S})^{*}(V_{r}^{1}-V_{r}^{2})dr,$\\thus\\$\hspace*{0.5cm}\lVert U_{t}^{1}-U_{t}^{2}\lVert_{H^{-s}}\leq\displaystyle\beta\displaystyle \int_{0}^{t} \lVert\Upsilon(t-r)(G_{r}^{I})^{*}(U_{r}^{1}-U_{r}^{2})\lVert_{H^{-s}}dr+\beta\int_{0}^{t}\lVert\Upsilon(t-r)(G_{r}^{S})^{*}(V_{r}^{1}-V_{r}^{2})\lVert_{H^{-s}}dr.$\vspace*{0.1cm}\\ Moreover from Corollary \ref{ff2}, we deduce that \\
                                                 $\lVert U_{t}^{1}-U_{t}^{2}\lVert_{H^{-s}}\leq\beta C \underset{y}{\sup}\lVert K(.,y)\lVert_{H^{s}}\displaystyle \int_{0}^{t}\lVert U_{r}^{1}-U_{r}^{2}\lVert_{H^{-s}}  dr+\\\hspace*{7cm}\beta C\underset{x}{\sup}\Big\lVert \displaystyle\frac{K(x,.)}{\int_{\mathbb{T}^{2}}K(x',.)\mu_{r}(dx')}\Big\lVert_{H^{s}}\int_{0}^{t} \lVert V_{r}^{1}-V_{r}^{2} \lVert_{H^{-s}}dr,$ \vspace*{0.16cm}\\hence using  Lemmas \ref{ape} below and \ref{ff3},  we obtain  \vspace*{0.16cm}\\$\hspace*{3cm}\lVert U_{t}^{1}-U_{t}^{2}\lVert_{H^{-s}}\leq\beta C \displaystyle \int_{0}^{t}\Big\{\lVert U_{r}^{1}-U_{r}^{2}\lVert_{H^{-s}}  + \lVert V_{r}^{1}-V_{r}^{2} \lVert_{H^{-s}}\Big\}dr.$ \hspace*{2.5cm} ($6.28$)\\
                                                 On the  other hand  \\\hspace*{3cm}
                                                 $V_{t}^{1}-V_{t}^{2}=-(\displaystyle U_{t}^{1}-U_{t}^{2})-\alpha\int_{0}^{t}\Upsilon(t-r) (V_{r}^{1}-V_{r}^{2})dr, $\\
                                                 hence \\\hspace*{3cm}
                                                 $\lVert V_{t}^{1}-V_{t}^{2}\lVert_{H^{-s}} \leq\lVert U_{t}^{1}-U_{t}^{2}\lVert_{H^{-s}}+\alpha\displaystyle\int_{0}^{t}\lVert V_{r}^{1}-V_{r}^{2}\lVert_{H^{-s}}dr. \hspace*{3.1cm}  (6.29 $)\\
                                                 Summing $(6.28)$ and  $ (6.29) $ and applying Gronwall's lemma we obtain  $(U^{1},V^{1})=(U^{2},V^{2})$.
                  \end{proof} 
                    The next two Lemmas will be  useful to show an additional result about the regularity of the pair of Processes $(U,V).$
                  \begin{lem}
                 For any $x\in \mathbb{T}^{2},$ let us set $\mathcal{K}_{t}^{I}(x)=\displaystyle \Big(\mu_{t}^{I},\frac{K(x,.)}{(\mu_{t},K)}\Big).$ For any $t,t_{0}\in \mathbb{R}_{+}$, there exists $C>0$ such that \vspace*{0.1cm}$\\\hspace*{3.6cm}\lVert \mathcal{K}_{t}^{I}(.)-\mathcal{K}_{t_{0}}^{I}(.)\lVert_{H^{2}}\leq C\big( d_{F}(\mu_{t}, \mu_{t_{0}})+d_{F}(\mu_{t}^{I}, \mu_{t_{0}}^{I})\big).$ \label{ad}
                  \end{lem}
                  \begin{proof}
                 Let $t,t_{0}\in \mathbb{R}_{+},$ we have \\\hspace*{1.5cm} $\mathcal{K}_{t}^{I}(x)-\mathcal{K}_{t_{0}}^{I}(x)=\displaystyle \Big(\mu_{t}^{I}-\mu_{t_{0}}^{I},\frac{K(x,.)}{(\mu_{t},K)}\Big)+\Big(\mu_{t}^{I},\frac{K(x,.)\big(\mu_{t}-\mu_{t_{0}},K)}{(\mu_{t},K)(\mu_{t_{0}},K)}\Big).$\\Furthemore: \\ 1- Since the map $y\in A(x)\mapsto \displaystyle \int_{\mathbb{T}^{2}}K(x',y)\mu_{r}(dx')$ is Lipschitz (with the Lipschitz constant $2\sqrt{2}C_{k}\delta_{2}$ (see Lemma \ref{e5})) and  lower bounded by a positive constant and  from Remark \ref{f2f}, $\forall \lvert \eta\lvert\leq 2 $ the map \\ $y\in A(x)\mapsto D^{\eta}K(x,y)$ is Lipschitz (with the Lipschitz constant independent of $x$) and bounded by  $C \underset{0\leq  \lvert\eta\lvert\leq 2}{\max}\lVert k^{(\lvert \eta\lvert)}\lVert_{\infty},$ the map $y\in A(x)\mapsto \displaystyle \frac{D^{\eta}K(x,y)}{\int_{\mathbb{T}^{2}}K(x',y)\mu_{r}(dx')}$ is Lipschitz (with the constant Lipschitz independent of $x$) and bounded by $C \underset{\lvert\eta\lvert\leq 2}{\max}\lVert k^{(\lvert \eta\lvert)}\lVert_{\infty}.$\\ Hence  $ \Big\lVert \displaystyle\Big(\mu_{t}^{I}-\mu_{t_{0}}^{I},\frac{K(x,.)}{(\mu_{t},K)}\Big) \Big\lVert_{H^{2}}^{2}=\displaystyle\sum\limits_{\lvert\eta\lvert\leq 2}\int_{\mathbb{T}^{2}}\Big\lvert\int_{A(x)} \frac{D^{\eta}K(x,y)}{\int_{\mathbb{T}^{2}}K(x',y)\mu_{t}(dx')}(\mu_{t}^{I}-\mu_{t_{0}}^{I})(dy)\Big\lvert^{2}dx\\\hspace*{5.68cm}\leq Cd_{F}(\mu_{t}^{I}, \mu_{t_{0}}^{I}).$\vspace*{0.1cm}\\2-  Since  $K(x,.)$  is Lipschitz (with the Lipschit constant $2\sqrt{2}C_{k}$ (see Lemma \ref{e5})) and bounded by $\lVert k\lVert_{\infty}$ and  the map $y\in A(x)\mapsto \displaystyle \int_{\mathbb{T}^{2}}K(x',y)\mu_{r}(dx')$ is   lower bounded and $\forall \lvert \eta\lvert\leq 2 $ the map $y\in A(x)\mapsto D^{\eta}K(x,y)$ is bounded, \\$\hspace*{1.5cm} \Big\lVert \displaystyle\Big(\mu_{t}^{I},\frac{K(x,.)\big(\mu_{t}-\mu_{t_{0}},K)}{(\mu_{t},K)(\mu_{t_{0}},K)}\Big) \Big\lVert_{H^{2}}^{2}\\\hspace*{4.5cm}=\displaystyle\sum\limits_{\lvert\eta\lvert\leq 2}\int_{\mathbb{T}^{2}}\Big\lvert\int_{A(x)} \frac{D^{\eta}K(x,y)\big(\mu_{t}-\mu_{t_{0}},K(.,y))}{\int_{\mathbb{T}^{2}}K(x',y)\mu_{t}(dx')\int_{\mathbb{T}^{2}}K(x',y)\mu_{t_{0}}(dx')}\mu_{t_{0}}^{I}(dy)\Big\lvert^{2}dx,\\\hspace*{4.5cm}\leq C\displaystyle\sum\limits_{\lvert\eta\lvert\leq 2}\int_{\mathbb{T}^{2}}\int_{A(x)} \Big\lvert \frac{D^{\eta}K(x,y)\big(\mu_{t}-\mu_{t_{0}},K(.,y))}{\int_{\mathbb{T}^{2}}K(x',y)\mu_{t}(dx')\int_{\mathbb{T}^{2}}K(x',y)\mu_{t_{0}}(dx')}\Big\lvert^{2}\mu_{t_{0}}^{I}(dy)dx,\\\hspace*{4.5cm}\leq Cd^{2}_{F}(\mu_{t}, \mu_{t_{0}}).$\vspace*{0.13cm}\\So the result follows from 1- and  2-.
                  \end{proof}
                  \begin{lem}
                                    Let  $\varphi \in H^{2}(\mathbb{T}^{2})$, for any $x\in \mathbb{T}^{2},$ let us set  $\mathcal{K}_{t}^{S}(x,\varphi)=\displaystyle\frac{(\mu_{t}^{S},\varphi K(x,.))}{(\mu_{t},K(x,.))}.$\\ For any $t,t_{0}\in \mathbb{R}_{+}$, $\vspace*{0.1cm}\\\hspace*{2cm}\lVert \mathcal{K}_{t}^{S}(.,\varphi)-\mathcal{K}_{t_{0}}^{S}(.,\varphi)\lVert_{H^{2}}\leq C(\lVert  \varphi \lVert_{H^{2}})\big( d_{F}(\mu_{t}, \mu_{t_{0}})+d_{F}(\mu_{t}^{S}, \mu_{t_{0}}^{S})\big).$ \label{ad1}
                                    \end{lem}
                                    \begin{proof}
                                   Let $t,t_{0}\in \mathbb{R}_{+},$ we have \\ $\mathcal{K}_{t}^{S}(x,\varphi)-\mathcal{K}_{t_{0}}^{S}(x,\varphi)=\displaystyle \frac{(\mu_{t}^{S}-\mu_{t_{0}}^{S},\varphi K(x,.))}{(\mu_{t_{0}},K(x,.))}+\frac{(\mu_{t}^{S},\varphi K(x,.))}{(\mu_{t},K(x,.))(\mu_{t_{0}},K(x,.))}(\mu_{t}-\mu_{t_{0}},K(x,.))$.\\Furthemore: \\ 1- Since from Remark \ref{f2f}, $\forall \lvert \eta\lvert\leq 2,$ the map  $y\in A(x)\mapsto \varphi(y) D^{\eta} K(x,y)$ is Lipschitz ( with the Lipschitz constant independent of x) and bounded by $C\lVert \varphi \lVert_{\infty} \underset{0\leq  \lvert\eta\lvert\leq 2}{\max}\lVert k^{(\lvert \eta\lvert)}\lVert_{\infty}$ and $H^{2}$ is  Banach algebra and    an easy adaptation of the proof of Lemma \ref{ape} below yields that $\Big\lVert\displaystyle\frac{1}{(\mu_{t_{0}},K)}\Big\lVert_{H^{2}}<C$,\\$\hspace*{1.4cm} \Big\lVert \displaystyle\frac{(\mu_{t}^{S}-\mu_{t_{0}}^{S},\varphi K)}{(\mu_{t_{0}},K)}\Big\lVert_{H^{2}}^{2}\leq C \Big\lVert\frac{1}{(\mu_{t_{0}},K)}\Big\lVert_{H^{2}}^{2}\Big\lVert \displaystyle(\mu_{t}^{S}-\mu_{t_{0}}^{S},\varphi K)\Big\lVert_{H^{2}}^{2}\\\hspace*{5cm}=C\displaystyle\sum\limits_{\lvert\eta\lvert\leq 2}\int_{\mathbb{T}^{2}}\Big\lvert\int_{A(x)}\varphi(y) D^{\eta}K(x,y)(\mu_{t}^{S}-\mu_{t_{0}}^{S})(dy)\Big\lvert^{2}dx\\\hspace*{5cm}\leq C(\lVert \varphi \lVert_{\infty})d^{2}_{F}(\mu_{t}^{S}, \mu_{t_{0}}^{S}).$\\\\2- Using   an argument similar to    the proof of (6.14) in proposition \ref{f5}  and  an easy adaptation of Lemma \ref{ape} below,  we have\\ $\hspace*{2cm}\Big\lVert \displaystyle\frac{(\mu_{t}^{S},\varphi K)}{(\mu_{t},K)(\mu_{t_{0}},K)} \Big\lVert_{H^{2}}\leq C\lVert \varphi \lVert_{H^{2}} \underset{x}{\sup}\Big\lVert \displaystyle\frac{ K(x,.)}{\int_{\mathbb{T}^{2}}K(.,x'')\mu_{t}(dx'')\int_{\mathbb{T}^{2}K(.,x')}\mu_{t_{0}}(dx')} \Big\lVert_{H^{2}}\leq C.$  \vspace*{0.1cm}\\Now since from Remark \ref{f2f}, $\forall \lvert \eta\lvert\leq 2,$ the map  $y\in A(x)\mapsto D^{\eta} K(x,y)$ is Lipschitz  (with the Lipschitz constant independent of x) and bounded by $C \underset{0\leq  \lvert\eta\lvert\leq 2}{\max}\lVert k^{(\lvert \eta\lvert)}\lVert_{\infty}$ and   $H^{2}$ is  Banach algebra, one has\vspace*{0.1cm}\\ $\hspace*{0.5cm}\Big\lVert \displaystyle\frac{(\mu_{t}^{S},\varphi K)}{(\mu_{t},K)(\mu_{t_{0}},K)}(\mu_{t}-\mu_{t_{0}},K) \Big\lVert_{H^{2}}^{2}\leq C \Big\lVert \displaystyle\frac{(\mu_{t}^{S},\varphi K)}{(\mu_{t},K)(\mu_{t_{0}},K)} \Big\lVert_{H^{2}}^{2} \Big\lVert \displaystyle(\mu_{t}-\mu_{t_{0}}, K)\Big\lVert_{H^{2}}^{2}\\\hspace*{6.4cm}=C(\lVert \varphi \lVert_{H^{2}})\displaystyle\sum\limits_{\lvert\eta\lvert\leq 2}\int_{\mathbb{T}^{2}}\Big\lvert\int_{A(x)} D^{\eta}K(x,y)(\mu_{t}-\mu_{t_{0}})(dy)\Big\lvert^{2}dx\\\hspace*{6.4cm}\leq C(\lVert \varphi \lVert_{H^{2}})d^{2}_{F}(\mu_{t}, \mu_{t_{0}}).$\vspace*{0.13cm}\\So the result follows from 1- and 2-.
                                    \end{proof}
                                    Now we  state  an additional result about the regularity of the pair of processes $(U,V).$
                                     \begin{prp} The pair of processes $(U,V)$ belongs to $L^{2}_{loc}(\mathbb{R}_{+},(H^{1-s})^{2}).$
                                                       \label{ar}
                                                        \end{prp}
                                                        \begin{proof}

                                                         Let us set \\ \hspace*{2.5cm}$ U_{t}^{s}=(I_{d}-\gamma\bigtriangleup)^{\frac{-s}{2}} U_{t}$,\hspace*{0.3cm} $ V_{t}^{s}=(I_{d}-\gamma\bigtriangleup)^{\frac{-s}{2}} V_{t}$, \hspace*{0.3cm}$ Z_{t}^{s}=(I_{d}-\gamma\bigtriangleup)^{\frac{-s}{2}} Z_{t},$\vspace*{0.06cm}\\\hspace*{2.3cm} $ W_{t}^{1,s}=(I_{d}-\gamma\bigtriangleup)^{\frac{-s}{2}} W_{t}^{1}$,\hspace*{0.3cm} $ W_{t}^{2,s}=(I_{d}-\gamma\bigtriangleup)^{\frac{-s}{2}} W_{t}^{2}.$\vspace*{0.1cm}\\Given that $U_{t}\in H^{-s}$,  $U_{t}^{s}\in L^{2}(\mathbb{T}^{2}).$ The same conclusion remains true for $V_{t}^{s}$, $Z_{t}^{s}$, $W_{t}^{1,s}$ and $W_{t}^{2,s}$. Now
                                                                                                                               by noticing that $\bigtriangleup(I_{d}-\gamma\bigtriangleup)^{\frac{-s}{2}}$= $(I_{d}-\gamma\bigtriangleup)^{\frac{-s}{2}}\bigtriangleup,$ we deduce from the evolution equations (\ref{df1}) and (\ref{df2}) of $(U,V)$ given in Theorem \ref{ff1}  that \vspace*{0.16cm}\\
                      $\displaystyle U_{t}^{s} =U^{s}_{0} + \gamma\int_{0}^{t}\bigtriangleup U_{r}^{s}dr  + \displaystyle\beta \int_{0}^{t} (I_{d}-\bigtriangleup)^{\frac{-s}{2}}(G_{t}^{S,I})^{*}(I_{d}-\gamma\bigtriangleup)^{\frac{s}{2}}Z_{r}^{s}dr\\\hspace*{1.3cm}-\beta\int_{0}^{t} (I_{d}-\gamma\bigtriangleup)^{\frac{-s}{2}}(G_{r}^{I})^{*}(I_{d}-\gamma\bigtriangleup)^{\frac{s}{2}}U_{r}^{s}dr-\beta\int_{0}^{t} (I_{d}-\gamma\bigtriangleup)^{\frac{-s}{2}}(G_{r}^{S})^{*}(I_{d}-\gamma\bigtriangleup)^{\frac{s}{2}}V_{r}^{s}dr+ W_{t}^{1,s},$\\
                      $\displaystyle V_{t}^{s} =V^{s}_{0} + \gamma\int_{0}^{t}\bigtriangleup V_{r}^{s}dr   \displaystyle -\beta \int_{0}^{t} (I_{d}-\gamma\bigtriangleup)^{\frac{-s}{2}}(G_{r}^{S,I})^{*}(I_{d}-\gamma\bigtriangleup)^{\frac{s}{2}}Z_{r}^{s}dr\\\hspace*{1.3cm}+\beta\int_{0}^{t}(I_{d}-\gamma\bigtriangleup)^{\frac{-s}{2}} (G_{r}^{I})^{*}(I_{d}-\gamma\bigtriangleup)^{\frac{s}{2}}U_{r}^{s}dr+\beta\int_{0}^{t} (I_{d}-\gamma\bigtriangleup)^{\frac{-s}{2}}(G_{r}^{S})^{*}(I_{d}-\gamma\bigtriangleup)^{\frac{s}{2}}V_{r}^{s}dr\\\hspace*{1.3cm}-\alpha\int_{0}^{t}V_{r}^{s}dr+ W_{t}^{2,s}.$ \vspace*{0.12cm}\\Moreover the above system can be rewritten as one equation as follows\vspace*{0.2cm}\\ $\hspace*{3cm}(dU_{t}^{s},dV_{t}^{s})'+ A(t,(U_{t}^{s},V_{t}^{s}))dt=J(t)dt+(dW_{t}^{1,s},dW_{t}^{2,s})',$\\where\\
                                                                                                                                       $A(t,(U_{t}^{s},V_{t}^{s}))$=\vspace*{0.1cm}\\\begin{tabular}{cc}                                                                                                     $\begin{pmatrix}
                     -\gamma \bigtriangleup +\beta(I_{d}-\gamma\bigtriangleup)^{\frac{-s}{2}} (G_{t}^{I})^{*}(I_{d}-\gamma\bigtriangleup)^{\frac{s}{2}} & \beta(I_{d}-\gamma\bigtriangleup)^{\frac{-s}{2}} (G_{t}^{S})^{*}(I_{d}-\gamma\bigtriangleup)^{\frac{s}{2}} \\
                     -\beta(I_{d}-\gamma\bigtriangleup)^{\frac{-s}{2}} (G_{t}^{I})^{*}(I_{d}-\gamma\bigtriangleup)^{\frac{s}{2}} & -\gamma \bigtriangleup - \beta(I_{d}-\gamma\bigtriangleup)^{\frac{-s}{2}} (G_{t}^{S})^{*}(I_{d}-\gamma\bigtriangleup)^{\frac{s}{2}}+\alpha
                                                                                                                              \end{pmatrix}$\hspace*{-0.5cm}& $\begin{pmatrix}
                     U_{t}^{s}\\ V_{t}^{s}                                                                                                          \end{pmatrix}$ 
                        \vspace*{0.12cm}                                                                                                      \end{tabular}                                                                                         $J(t)=((I_{d}-\gamma\bigtriangleup)^{\frac{-s}{2}}(G_{t}^{S,I})^{*}(I_{d}-\gamma\bigtriangleup)^{\frac{s}{2}}Z_{t}^{s},\hspace*{0.12cm}-(I_{d}-\gamma\bigtriangleup)^{\frac{-s}{2}}(G_{t}^{S,I})^{*}(I_{d}-\gamma\bigtriangleup)^{\frac{s}{2}}Z_{t}^{s})'$\\ and  $(u,v)'$ denotes the transpose of the pair $(u,v)$.\vspace*{0.23cm}\\
      On the other hand given that the dual spaces of $(L^{2}(\mathbb{T}^{2}))^{2}$ and $(H^{1})^{2}$ are isomorphic to $(L^{2}(\mathbb{T}^{2}))'\times (L^{2}(\mathbb{T}^{2}))'$ and $H^{-1}\times H^{-1}$ respectively, by identifying $L^{2}(\mathbb{T}^{2})$ to its dual space $(L^{2}(\mathbb{T}^{2}))'$, we have $(H^{1})^{2}\subset(L^{2}(\mathbb{T}^{2}))^{2}\subset (H^{-1})^{2}.$\vspace*{0.17cm}\\
         Let us now prove that the family of operators $A(t,.)$  satisfies the assuptions given in Proposition \ref{b8} in section \ref{pr}, with  $H=(L^{2}(\mathbb{T}^{2}))^{2},$  $F=(H^{1})^{2}$ and  $ F'=(H^{-1})^{2}.$\\
        We  first Recall  that the family of eigenfunctions  $(f^{i,s}_{n_{1},n_{2}})_{i,n_{1},n_{2}}$ (as defined in Proposition \ref{b2}) associate to the family of eigenvalues  $(\lambda_{n_{1},n_{2}})_{n_{1},n_{2}}$ of the operator $-\gamma\bigtriangleup$ is an orthonormal basis of $L^{2}(\mathbb{T}^{2}).$\vspace*{0.1cm}\\ Now  we start by showing that $\forall (\mathcal{U},\mathcal{V})\in H^{1}\times H^{1},$ $A(t,(\mathcal{U},\mathcal{V}))\in H^{-1}\times H^{-1}.$ 
        We have \\ $A(t,(\mathcal{U},\mathcal{V}))=\vspace*{0.16cm}\\\begin{pmatrix}
                             -\gamma \bigtriangleup \mathcal{U} +\beta(I_{d}-\gamma\bigtriangleup)^{\frac{-s}{2}} (G_{t}^{I})^{*}(I_{d}-\gamma\bigtriangleup)^{\frac{s}{2}}\mathcal{U} +\beta(I_{d}-\gamma\bigtriangleup)^{\frac{-s}{2}} (G_{t}^{S})^{*}(I_{d}-\gamma\bigtriangleup)^{\frac{s}{2}}\mathcal{V} \\
                             -\beta(I_{d}-\gamma\bigtriangleup)^{\frac{-s}{2}} (G_{t}^{I})^{*}(I_{d}-\gamma\bigtriangleup)^{\frac{s}{2}}\mathcal{U} -\gamma \bigtriangleup \mathcal{V} - \beta(I_{d}-\gamma\bigtriangleup)^{\frac{-s}{2}} (G_{t}^{S})^{*}(I_{d}-\gamma\bigtriangleup)^{\frac{s}{2}} \mathcal{V}+\alpha \mathcal{V}\\
                                                                                                                                      \end{pmatrix}.$ \vspace*{0.1cm}\\ 
                                                                                                                                      Moreover: \vspace*{0.1cm}\\$(i).$ $ \lVert (I_{d}-\gamma\bigtriangleup)^{\frac{-s}{2}} (G_{t}^{I})^{*}(I_{d}-\gamma\bigtriangleup)^{\frac{s}{2}}\mathcal{U}\lVert_{H^{-1}}\leq C \lVert (I_{d}-\gamma\bigtriangleup)^{\frac{-s}{2}} (G_{t}^{I})^{*}(I_{d}-\gamma\bigtriangleup)^{\frac{s}{2}}\mathcal{U}\lVert_{L^{2}} \\\hspace*{7.2cm}=C\lVert  (G_{t}^{I})^{*}(I_{d}-\gamma\bigtriangleup)^{\frac{s}{2}}\mathcal{U}\lVert_{H^{-s}}\vspace*{0.1cm}\\\hspace*{7.2cm}\leq C\lVert  (I_{d}-\gamma\bigtriangleup)^{\frac{s}{2}}\mathcal{U}\lVert_{H^{-s}}\vspace*{0.1cm}\\\hspace*{7.2cm}=C\lVert  \mathcal{U}\lVert_{L^{2}},$ \\ where the second inequality follows from Corollary \ref{ff2}.\vspace*{0.1cm}\\$(ii).$ Similarly form Corollary \ref{ff2}, we have\vspace*{0.21cm}\\ $\lVert \beta(I_{d}-\gamma\bigtriangleup)^{\frac{-s}{2}} (G_{t}^{S})^{*}(I_{d}-\gamma\bigtriangleup)^{\frac{s}{2}}\mathcal{V} \lVert_{H^{-1}}\leq C \lVert (I_{d}-\gamma\bigtriangleup)^{\frac{-s}{2}} (G_{t}^{S})^{*}(I_{d}-\gamma\bigtriangleup)^{\frac{s}{2}}\mathcal{V}\lVert_{L^{2}} \vspace*{0.1cm}\\\hspace*{6.8cm}=C\lVert  (G_{t}^{S})^{*}(I_{d}-\gamma\bigtriangleup)^{\frac{s}{2}}\mathcal{V}\lVert_{H^{-s}}\vspace*{0.1cm}\\\hspace*{6.8cm}\leq C\lVert  (I_{d}-\gamma\bigtriangleup)^{\frac{s}{2}}\mathcal{V}\lVert_{H^{-s}}\\\hspace*{6.8cm}=C\lVert  \mathcal{V}\lVert_{L^{2}}.$ \vspace*{0.2cm}\\$(iii).$ $\lVert -\gamma\bigtriangleup \mathcal{U}\lVert_{H^{-1}}^{2}= \sum\limits_{i,n_{1},n_{2}}(1+\gamma\pi^{2}(n_{1}^{2}+n_{2}^{2}))^{-1}(-\gamma\bigtriangleup \mathcal{U},f_{n_{1},n_{2}}^{i})_{L^{2}}^{2}\\\hspace*{3.36cm}= \sum\limits_{i\neq 0,n_{1},n_{2}}\gamma\pi^{2}(n_{1}^{2}+n_{2}^{2})(1+\gamma\pi^{2}(n_{1}^{2}+n_{2}^{2}))^{-1}( \mathcal{U},f_{n_{1},n_{2}}^{i})_{L^{2}}^{2}\\\hspace*{3.36cm}\leq \sum\limits_{i,n_{1},n_{2}}(1+\gamma\pi^{2}(n_{1}^{2}+n_{2}^{2}))( \mathcal{U},f_{n_{1},n_{2}}^{i})_{L^{2}}^{2}=\lVert \mathcal{U}\lVert_{H^{1}}^{2}.$\vspace*{0.12cm}\\So from $(i),$ $(ii),$ $(iii)$  we conclude that $A(t,(\mathcal{U},\mathcal{V}))\in H^{-1}\times H^{-1}.$\vspace*{0.15cm}\\ Let us now prove that the four assuptions in Proposition \ref{b8} are satisfied. In the rest of this proof we will note by $ <.,.>$ the duality product between  $H^{-1}\times H^{-1}$ and $H^{1}\times H^{1}.$\vspace*{0.2cm}\\(1). Given that $A(t,.)$ is linear, to prove the first point, it is enough to show that $\forall (\mathcal{U},\mathcal{V})$, $(\mathcal{U}_{1},\mathcal{V}_{1})\in H^{1}\times H^{1}$ the map $\theta\mapsto <A(t,\theta(\mathcal{U},\mathcal{V})), (\mathcal{U}_{1},\mathcal{V}_{1})>$ is continuous on $\mathbb{R},$  which in turn is so easy.\vspace*{0.12cm}\\(2). The fact that there exists  $\delta>0$ such that $\lVert A(t,(\mathcal{U},\mathcal{V})) \lVert_{H^{-1}\times H^{-1}}\leq \delta \lVert (\mathcal{U},\mathcal{V}) \lVert_{H^{1}\times H^{1}}  $ follows from $(i),$ $(ii),$ and  $(iii).$\vspace*{0.12cm}\\(3). Let us prove that there  exists $\sigma_{1}>0$, $\sigma_{2}\in \mathbb{R},$ such that \vspace*{0.15cm}\\\hspace*{3cm} $<A(t,(\mathcal{U},\mathcal{V})),(\mathcal{U},\hspace*{0.1cm}\mathcal{V})>+ \sigma_{2}\lVert (\mathcal{U},\mathcal{V})\lVert_{L^{2}\times L^{2}}^{2}\geq \sigma_{1}\lVert (\mathcal{U},\mathcal{V})\lVert_{H^{1}\times H^{1}}^{2}$.\vspace*{0.13cm}\\We have $<A(t,(\mathcal{U},\mathcal{V})),(\mathcal{U},\mathcal{V})>\vspace*{0.12cm}\\\hspace*{0.5cm}=< -\gamma \bigtriangleup \mathcal{U} +\beta(I_{d}-\gamma\bigtriangleup)^{\frac{-s}{2}} (G_{t}^{I})^{*}(I_{d}-\gamma\bigtriangleup)^{\frac{s}{2}}\mathcal{U} +\beta(I_{d}-\gamma\bigtriangleup)^{\frac{-s}{2}} (G_{t}^{S})^{*}(I_{d}-\gamma\bigtriangleup)^{\frac{s}{2}}\mathcal{V},\hspace*{0.1cm}\mathcal{U}> \vspace*{0.12cm}\\\hspace*{0.5cm}+                                                                                      <-\beta(I_{d}-\gamma\bigtriangleup)^{\frac{-s}{2}} (G_{t}^{I})^{*}(I_{d}-\gamma\bigtriangleup)^{\frac{s}{2}}\mathcal{U} -\gamma \bigtriangleup \mathcal{V} - \beta(I_{d}-\gamma\bigtriangleup)^{\frac{-s}{2}} (G_{t}^{S})^{*}(I_{d}-\gamma\bigtriangleup)^{\frac{s}{2}} \mathcal{V}+\alpha \mathcal{V},\hspace*{0.12cm}\mathcal{V}>.$ \vspace*{0.01cm}\\Furthemore: \vspace*{0.1cm}\\(3.1). $< -\gamma \bigtriangleup \mathcal{U},\hspace*{0.12cm}\mathcal{U}>= \gamma <\bigtriangledown\mathcal{U},\hspace*{0.1cm}\bigtriangledown\mathcal{U}>=\gamma\lVert \bigtriangledown\mathcal{U}\lVert_{L^{2}}^{2}.$\vspace*{0.13cm}\\(3.2).  From $(i)$ and $(ii),$ we note that \vspace*{0.1cm}\\\hspace*{2cm}$\beta(I_{d}-\gamma\bigtriangleup)^{\frac{-s}{2}} (G_{t}^{I})^{*}(I_{d}-\gamma\bigtriangleup)^{\frac{s}{2}}\mathcal{U} +\beta(I_{d}-\gamma\bigtriangleup)^{\frac{-s}{2}} (G_{t}^{S})^{*}(I_{d}-\gamma\bigtriangleup)^{\frac{s}{2}}\mathcal{V}\in L^{2}(\mathbb{T}^{2}),$\vspace*{0.1cm}\\ thus as $H^{1}\subset L^{2}(\mathbb{T}^{2}),$ from $(i)$ and $(ii)$ and from Young's inequality we have \\ $\lvert<\beta(I_{d}-\gamma\bigtriangleup)^{\frac{-s}{2}} (G_{t}^{I})^{*}(I_{d}-\gamma\bigtriangleup)^{\frac{s}{2}}\mathcal{U} +\beta(I_{d}-\gamma\bigtriangleup)^{\frac{-s}{2}} (G_{t}^{S})^{*}(I_{d}-\gamma\bigtriangleup)^{\frac{s}{2}}\mathcal{V},\hspace*{0.12cm}\mathcal{U}>\lvert\vspace*{0.1cm}\\\leq \lVert \beta(I_{d}-\gamma\bigtriangleup)^{\frac{-s}{2}} (G_{t}^{I})^{*}(I_{d}-\gamma\bigtriangleup)^{\frac{s}{2}}\mathcal{U}\lVert_{L^{2}} \lVert  \mathcal{U}\lVert_{L^{2}}+\lVert \beta(I_{d}-\gamma\bigtriangleup)^{\frac{-s}{2}} (G_{t}^{S})^{*}(I_{d}-\gamma\bigtriangleup)^{\frac{s}{2}}\mathcal{V}\lVert_{L^{2}} \lVert  \mathcal{U}\lVert_{L^{2}}\vspace*{0.19cm}\\\leq C\beta(\lVert\mathcal{U}\lVert_{L^{2}}^{2} +\lVert \mathcal{V}\lVert_{L^{2}} \lVert  \mathcal{U}\lVert_{L^{2}})\leq C_{1}\beta(\lVert\mathcal{U}\lVert_{L^{2}}^{2} +\lVert \mathcal{V}\lVert_{L^{2}}^{2} ).$\vspace*{0.1cm} \\Hence\\ $<\beta(I_{d}-\gamma\bigtriangleup)^{\frac{-s}{2}} (G_{t}^{I})^{*}(I_{d}-\gamma\bigtriangleup)^{\frac{s}{2}}\mathcal{U} +\beta(I_{d}-\gamma\bigtriangleup)^{\frac{-s}{2}} (G_{t}^{S})^{*}(I_{d}-\gamma\bigtriangleup)^{\frac{s}{2}}\mathcal{V},\hspace*{0.12cm}\mathcal{U}>\vspace*{0.12cm}\\+<-\beta(I_{d}-\gamma\bigtriangleup)^{\frac{-s}{2}} (G_{t}^{I})^{*}(I_{d}-\gamma\bigtriangleup)^{\frac{s}{2}}\mathcal{U} -\beta(I_{d}-\gamma\bigtriangleup)^{\frac{-s}{2}} (G_{t}^{S})^{*}(I_{d}-\gamma\bigtriangleup)^{\frac{s}{2}}\mathcal{V},\hspace*{0.12cm}\mathcal{V}>$ belongs in the \vspace*{0.17cm}\\interval  $\big[-C_{2}\beta(\lVert\mathcal{U}\lVert_{L^{2}}^{2} +\lVert \mathcal{V}\lVert_{L^{2}}^{2} ),\hspace*{0.12cm}C_{2}\beta(\lVert\mathcal{U}\lVert_{L^{2}}^{2} +\lVert \mathcal{V}\lVert_{L^{2}}^{2} )\big].$\vspace*{0.1cm}\\(3.3). $<\mathcal{V},\mathcal{V}>=\lVert\mathcal{V}\lVert_{L^{2}}^{2}$.\vspace*{0.12cm}\\So   from (3.1), (3.2) and (3.3) we see that it is enough to choose $\sigma_{2}=C_{2}\beta+\gamma$  and $\sigma_{1}=\gamma.$\\(4). We finish by showing that for any $(\mathcal{U},\mathcal{V})\in H^{1}\times H^{1}$ the map $t\mapsto A(t, (\mathcal{U},\mathcal{V}))$ is Lebesgue measurable from $]0,T[$ to $ H^{-1}\times H^{-1}$.  However since $H^{-1}$ is a separable Hilbert space then according to the Pettis theorem (see Theorem 1.1 page 2 in [\ref{tc}]) it is enough to prove that $\forall (\mathcal{U}_{1},\mathcal{V}_{1})\in H^{1}\times H^{1},$ the map $t\mapsto < A(t, (\mathcal{U},\mathcal{V})), (\mathcal{U}_{1},\mathcal{V}_{1})>$ is Lebesgue measurable from $]0,T[$ to $\mathbb{R}.$\\Let us prove that the map $t\mapsto < A(t, (\mathcal{U},\mathcal{V})), (\mathcal{U}_{1},\mathcal{V}_{1})>$ is continuous from $]0,T[$ to $\mathbb{R}.$\\ Let $t, t_{0}\in ]0,T[ $, by noticing that $\forall x \in \mathbb{T}^{2},$ $ G_{t}^{I}\varphi(x)=\varphi(x)\displaystyle \Big(\mu_{t}^{I},\frac{K(x,.)}{(\mu_{t},K)}\Big)=\varphi(x)\mathcal{K}_{t}^{I}(x)$ and  $H^{s-1}$ is a Banach algebra,  from Lemma \ref{ad}, we have\vspace*{0.12cm}\\
                                         $\lvert<(I_{d}-\gamma\bigtriangleup)^{\frac{-s}{2}} \big((G_{t}^{I})^{*}-(G_{t_{0}}^{I})^{*}\big)(I_{d}-\gamma\bigtriangleup)^{\frac{s}{2}}\mathcal{U} ,\hspace*{0.12cm}\mathcal{U}_{1}>\lvert\vspace*{0.12cm}\\\hspace*{2.5cm}=\lvert< (I_{d}-\gamma\bigtriangleup)^{\frac{s}{2}}\mathcal{U} ,\hspace*{0.12cm}\big(
                                         G_{t}^{I}-G_{t_{0}}^{I}\big)(I_{d}-\gamma\bigtriangleup)^{\frac{-s}{2}}\mathcal{U}_{1}>\lvert\vspace*{0.12cm}\\\hspace*{2.5cm}\leq\lVert (I_{d}-\gamma\bigtriangleup)^{\frac{s}{2}}\mathcal{U}\lVert_{H^{1-s}}\lVert \big(
                        G_{t}^{I}-G_{t_{0}}^{I}\big)(I_{d}-\gamma\bigtriangleup)^{\frac{-s}{2}}\mathcal{U}_{1}\lVert_{H^{s-1}}\vspace*{0.12cm}\\\hspace*{2.5cm}\leq\lVert (I_{d}-\gamma\bigtriangleup)^{\frac{s}{2}}\mathcal{U}\lVert_{H^{1-s}}\lVert (I_{d}-\gamma\bigtriangleup)^{\frac{-s}{2}}\mathcal{U}_{1}\big( \mathcal{K}_{t}^{I}(.)-\mathcal{K}_{t_{0}}^{I}(.)\big)\lVert_{H^{s-1}}\vspace*{0.12cm}\\\hspace*{2.5cm}\leq C\lVert \mathcal{U}\lVert_{H^{1}}\lVert (I_{d}-\gamma\bigtriangleup)^{\frac{-s}{2}}\mathcal{U}_{1}\lVert_{H^{s-1}}\lVert \mathcal{K}_{t}^{I}(.)-\mathcal{K}_{t_{0}}^{I}(.)\lVert_{H^{s-1}}\vspace*{0.12cm}\\\hspace*{2.5cm}\leq C\lVert \mathcal{U}\lVert_{H^{1}}\lVert \mathcal{U}_{1}\lVert_{H^{-1}}\lVert  \mathcal{K}_{t}^{I}(.)-\mathcal{K}_{t_{0}}^{I}(.)\lVert_{H^{2}}\vspace*{0.12cm}\\\hspace*{2.5cm}\leq C\lVert \mathcal{U}\lVert_{H^{1}}\lVert \mathcal{U}_{1}\lVert_{H^{1}}\big(d_{F}(\mu_{t}^{I}, \mu_{t_{0}}^{I})+d_{F}(\mu_{t}^{I}, \mu_{t_{0}}^{I})\big).$  \vspace*{0.12cm}\\Similarly by noticing that $\forall \mathcal{U}_{1}\in H^{1},$ $(I_{d}-\gamma\bigtriangleup)^{\frac{-s}{2}}\mathcal{U}_{1}\in H^{1+s}\subset H^{2}$ and \vspace*{0.12cm}\\ $ G_{t}^{S}(I_{d}-\gamma\bigtriangleup)^{\frac{-s}{2}}\mathcal{U}_{1}(x)=\displaystyle \frac{\big(\mu_{t}^{S},(I_{d}-\gamma\bigtriangleup)^{\frac{-s}{2}}\mathcal{U}_{1}K(x,.)\big)}{(\mu_{t},K)},$  from Lemma \ref{ad1}, we have \vspace*{0.12cm}\\$\lvert<(I_{d}-\gamma\bigtriangleup)^{\frac{-s}{2}} \big((G_{t}^{S})^{*}-(G_{t_{0}}^{S})^{*}\big)(I_{d}-\gamma\bigtriangleup)^{\frac{s}{2}}\mathcal{V} ,\hspace*{0.12cm}\mathcal{U}_{1}>\lvert\vspace*{0.12cm}\\\hspace*{2.5cm}=\lvert< (I_{d}-\gamma\bigtriangleup)^{\frac{s}{2}}\mathcal{V} ,\hspace*{0.12cm}\big(
                                                                 G_{t}^{S}-G_{t_{0}}^{S}\big)(I_{d}-\gamma\bigtriangleup)^{\frac{-s}{2}}\mathcal{U}_{1}>\lvert\vspace*{0.12cm}\\\hspace*{2.5cm}\leq\lVert (I_{d}-\gamma\bigtriangleup)^{\frac{s}{2}}\mathcal{V}\lVert_{H^{1-s}}\lVert \big(
                                                G_{t}^{S}-G_{t_{0}}^{S}\big)(I_{d}-\gamma\bigtriangleup)^{\frac{-s}{2}}\mathcal{U}_{1}\lVert_{H^{s-1}}\vspace*{0.12cm}\\\hspace*{2.5cm}\leq\lVert \mathcal{V}\lVert_{H^{1}}\lVert \lVert \mathcal{K}_{t}^{S}(.,(I_{d}-\gamma\bigtriangleup)^{\frac{-s}{2}}\mathcal{U}_{1})-\mathcal{K}_{t_{0}}^{S}(.,(I_{d}-\gamma\bigtriangleup)^{\frac{-s}{2}}\mathcal{U}_{1})\lVert_{H^{2}}\vspace*{0.12cm}\vspace*{0.12cm}\\\hspace*{2.5cm}\leq C\big(\lVert \mathcal{V}\lVert_{H^{1}},\lVert \mathcal{U}_{1}\lVert_{H^{1}}\big)\big(d_{F}(\mu_{t}^{S}, \mu_{t_{0}}^{S})+ d_{F}(\mu_{t}, \mu_{t_{0}})\big).$\vspace*{0.16cm}\\ So \\$\lvert  < A(t, (\mathcal{U},\mathcal{V})), (\mathcal{U}_{1},\mathcal{V}_{1})> - < A(t_{t_{0}}, (\mathcal{U},\mathcal{V})), (\mathcal{U}_{1},\mathcal{V}_{1})>\lvert\vspace*{0.15cm}\\\hspace*{1.5cm}\leq C\big(\lVert \mathcal{V}\lVert_{H^{1}},\lVert \mathcal{U}_{1}\lVert_{H^{1}},\lVert \mathcal{U}\lVert_{H^{1}},\lVert \mathcal{V}_{1}\lVert_{H^{1}}\big)\big(d_{F}(\mu_{t}^{S}, \mu_{t_{0}}^{S})+d_{F}(\mu_{t}^{I}, \mu_{t_{0}}^{I})+ d_{F}(\mu_{t}, \mu_{t_{0}})\big),$ \vspace*{0.13cm}\\thus the result follows from the fact that $\mu^{S},$ $\mu^{I}$ and $\mu$ belongs to $C\big(\mathbb{R}_{+},(\mathcal{M}_{F}(\mathbb{T}^{2}),d_{F})\big).$\vspace*{0.12cm}\\ Now in addition to the fact that  the family of operators $A(t,.)$  satisfies the  four assumptions given in Proposition \ref{b8} in section \ref{pr}, with  $H=(L^{2}(\mathbb{T}^{2}))^{2},$  $F=(H^{1})^{2}$ and  $ F'=(H^{-1})^{2}$, we  note  that:   \vspace*{0.18cm}\\$-$ $J(t)=((I_{d}-\gamma\bigtriangleup)^{\frac{-s}{2}}(G_{t}^{S,I})^{*}(I_{d}-\gamma\bigtriangleup)^{\frac{s}{2}}Z_{t}^{s},\hspace*{0.13cm}-(I_{d}-\gamma\bigtriangleup)^{\frac{-s}{2}}(G_{t}^{S,I})^{*}(I_{d}-\gamma\bigtriangleup)^{\frac{s}{2}}Z_{t}^{s})'$ \vspace*{0.1cm}\\\hspace*{0.4cm} belongs to  $L^{2}_{loc}(\mathbb{R}_{+},(L^{2}(\mathbb{T}^{2}))^{2}).$ Indeed, \\\hspace*{0.5cm}$\lVert (I_{d}-\gamma\bigtriangleup)^{\frac{-s}{2}}(G_{r}^{S,I})^{*}(I_{d}-\gamma\bigtriangleup)^{\frac{s}{2}}Z_{r}^{s} \lVert_{L^{2}}= \lVert (G_{t}^{S,I})^{*}Z_{t} \lVert_{H^{-s}}\leq C\lVert Z_{t} \lVert_{H^{-s}}$  and $Z\in C(\mathbb{R}_{+},H^{-s})$.\vspace*{0.1cm} \\$-$ $(U_{0}^{s},V_{0}^{s})\in (L^{2}(\mathbb{T}^{2}))^{2}$. \vspace*{0.16cm}\\ Thus we deduce from Proposition \ref{b8} that for $(A(t,.))_{\{t\in]0,T[\}}$ being  the above family of operators, the stochastic differential equation \\
                                               $\hspace*{3cm}(d\mathcal{U}(t),d\mathcal{V}(t))'+ A(t,(\mathcal{U}(t),\mathcal{V}(t))dt=J(t)dt+(dW_{t}^{1,s},dW_{t}^{2,s})',$ \hspace*{0.1cm}\vspace*{-0.17cm}\[\tag{6.28}  (\mathcal{U}(0),\mathcal{V}(0))=(U_{0}^{s},V_{0}^{s}),\vspace*{-0.08cm}\label{ere}\]   admits a unique solution $(\mathcal{U},\mathcal{V})\in L^{2}_{loc}\big(\mathbb{R}_{+},(H^{1}(\mathbb{T}^{2}))^{2}\big)\cap  C(\mathbb{R}_{+},(L^{2}(\mathbb{T}^{2}))^{2})$ a.s.\vspace*{0.15cm}\\ On the other hand, it has been shwon in Proposition \ref{rt1}, that the pair $(U^{s},V^{s})$ is the unique solution of the equation (\ref{ere}) in   $C(\mathbb{R}_{+},(L^{2}(\mathbb{T}^{2}))^{2})$(since equation (\ref{ere}) is equivalent to the system formed by  equations (\ref{df1}) and (\ref{df2}), which in turn is the same as the system formed by equations (6.26) and (6.27)). Thus we conclude that the pair of  processes\\ $(U^{s},V^{s})=(\mathcal{U},\mathcal{V})$ which belongs to  $L^{2}_{loc}\big(\mathbb{R}_{+},(H^{1}(\mathbb{T}^{2}))^{2}\big) \cap  C(\mathbb{R}_{+},(L^{2}(\mathbb{T}^{2}))^{2}),$ consequently the pair of processes $(U,V)$ belongs to  $ L^{2}_{loc}\big(\mathbb{R}_{+},(H^{1-s})^{2}\big).$
                                              \end{proof}      
                                         \begin{rmq}
           Given that $dZ^{N}_{t}=\gamma\bigtriangleup Z^{N}_{t}dt + d\widetilde{\mathcal{H}}_{t}^{N}$ (see  equation (\ref{d1})), by proceeding as in the proof of Proposition \ref{ar}, it is easy to see that $\forall N\in \mathbb{N}^{*},$ $Z^{N}\in L^{2}_{loc}(\mathbb{R}_{+},H^{1-s})$. Furthemore by Using the Itô formula establish in [\ref{tc} page 62] we can see that the sequence $(Z^{N})_{N}$   is bounded  in $L^{2}_{loc}(\mathbb{R}_{+},H^{1-s})$ (although we do not establish this result in this paper). Consequently since in a  Hilbert space the bounded sets are relatively compact for the weak topology, the  convergence in law  of  $(Z^{N})_{N}$ in $L^{2}_{loc}(\mathbb{R}_{+},H^{1-s})$ equipped with its weak topology can be easily deduced.\\                    
        The convergence in law of the sequence  $(U^{N},V^{N})_{N}$ in $(L^{2}_{loc}(\mathbb{R}_{+},H^{1-s}))^{2}$ where \\ $L^{2}_{loc}(\mathbb{R}_{+},H^{1-s})$ is equipped with its weak topology follows also by the fact it is bounded in   $(L^{2}_{loc}(\mathbb{R}_{+},H^{1-s}))^{2}$, which in turn is not difficult to  prove by using the Proposition \ref{ar}. However we do not establish that result in this paper.\vspace*{-0.3cm}
                                                                             \end{rmq}
                                                                                                                          \section{Law of Large Numbers, Central Limit Theorem of the sequence $(\mu^{S,N},\mu^{I,N})_{N\geq 1}:$ the case $\gamma=0$}
\fancyhf{}
\lhead{\textit{7. LAW OF LARGE NUMBERS AND CENTRAL LIMIT THEOREM}}
\rhead{\thepage}
 \par In this section we consider the second model presented  in the Introduction ie when the diffusion coefficient $\gamma=0$. More precisely we consider a compartmental SIR stochastic epidemic model for a population distributed on the two dimentional torus such that:\vspace*{0.05cm}\\
\hspace*{0.2cm} $\bullet$ The position of an individual $i$ is independent of time and is represented by    $X^{i}$ defined as  \\\hspace*{0.7cm} in the Introduction. \\
\hspace*{0.2cm} $\bullet$ A susceptible $i$ become infected at time t at the rate $\beta 1_{\{E_{t}^{i}=S\}}\sum_{j=1}^{N}\frac{K(X^{i},X^{j})}{\sum_{l=1}^{N}K(X^{l},X^{j})}1_{\{E_{t}^{j}=I\}}, $ \\
 where  $\alpha,\beta $,  $E_{t}^{i}$ and the function K are defined as in the Introduction. \\
 The temporal evolution of $S(\cdot)$, $I(\cdot)$ and $R(\cdot)$ are defined   as in the Introduction.\\The assumptions made at time $t=0,$ are presented in the Introduction, in other words, the sequence of  empirical measures $(\mu_{0}^{S,N},\mu_{0}^{I,N},\mu_{0}^{N})_{N}$ is defined as in the Introduction. \vspace*{0.2cm}\\Thus  the renormalized processes $\mu^{S,N}$, $\mu^{I,N}$, $\mu^{R,N}$ and  $\mu^{N}$ are defined as follows. $\forall t>0,$\\ $\displaystyle  \hspace*{5cm}\mu_{t}^{S,N}=\frac{1}{N}\sum\limits_{i=1}^{N}1_{\{E_{t}^{i}=S\}}\delta_{X^{i}} $ \\\hspace*{5cm}
 $\displaystyle  \mu_{t}^{I,N}=\frac{1}{N}\sum\limits_{i=1}^{N}1_{\{E_{t}^{i}=I\}}\delta_{X^{i}}$ \\\hspace*{5cm} $\mu_{t}^{R,N}=\frac{1}{N}\sum\limits_{i=1}^{N}1_{\{E_{t}^{i}=R\}}\delta_{X^{i}}$\\\hspace*{5cm}  $ \mu^{N}=\mu_{t}^{S,N}+\mu_{t}^{I,N}+\mu_{t}^{R,N}=\frac{1}{N}\sum\limits_{i=1}^{N}\delta_{X^{i}}.$ \\We first recall that Assumption (H0) in section \ref{se8} remains true. Now since we have already shown in  section \ref{se8} and section \ref{se10} respectively that $ (\mu^{S,N}_{0},\mu^{I,N}_{0},\mu^{N})\xrightarrow{a.s}(\mu^{S}_{0},\mu^{I}_{0},\mu) $  and that $\big(U_{0}^{N}=\sqrt{N}(\mu^{S,N}_{0}-\mu^{S}_{0}),V_{0}^{N}=\sqrt{N}(\mu^{I,N}_{0}-\mu^{I}_{0}),Z^{N}=\sqrt{N}(\mu^{N}-\mu)\big)$ converges in law in $(H^{-s})^{3}$ (for any s>1) towards the Gaussian vector $(U_{0},V_{0},Z),$   then the  aim in this section  is to study the law of large numbers  and the central limit theorem of the sequence  $\{(\mu^{S,N}_{t},\mu^{I,N}_{t}),t\geq0, N\geq 1\}.$
\subsection{Law of Large Numbers}
The following  is assumed to hold throughout subsection 7.1.\vspace*{0.18cm}\\
 \textbf{Assumption (H3): \hspace*{0.2cm} k is Lipschitz.}
\subsubsection{System of evolution equations of the pair $(\mu^{S,N},\mu^{I,N})$}
For  $\varphi \in C(\mathbb{T}^{2})$, since $\gamma=0$, the evolution equations of $\mu^{S,N}$ and $\mu^{I,N}$ simplify as follows.\vspace*{0.13cm}\\
$(\mu_{t}^{S,N},\varphi)=\displaystyle(\mu_{0}^{S,N},\varphi)- \beta \int_{0}^{t}\left(\mu_{r}^{S,N}, \varphi (\mu_{r}^{I,N},\frac{K}{(\mu^{N},K)})\right) dr + M_{t}^{' N,\varphi},$\vspace*{0.15cm}\\ $(\mu_{t}^{I,N} ,\varphi)=\displaystyle (\mu_{0}^{I,N},\varphi) + \int_{0}^{t}\left(\mu_{r}^{S,N}, \varphi (\mu_{r}^{I,N},\frac{K}{(\mu^{N},K)})\right) dr-\alpha \int_{0}^{t}(\mu_{r}^{I,N},\varphi)dr  +L_{t}^{' N,\varphi},$\vspace*{0.1cm}\\
                  where\vspace*{0.1cm}\\
                          $ M_{t}^{' N,\varphi}=\displaystyle - \frac{1}{N} \sum_{i=1}^{N} \int_{0}^{t} \int_{0}^{\infty}1_{\{E_{r^{-}}^{i}=S\}}\varphi(X^{i})1_{\{u\leq \frac{\beta}{N} \sum_{j=1}^{N}\frac{K(X^{i},X^{j})}{\sum\limits_{l=1}^{N}K(X^{l},X^{j})} 1_{\{E_{r}^{j}=I\}}\}} \overline{M}^{i}(dr,du), $ \\
                            $ L_{t}^{' N,\varphi} = \displaystyle\frac{1}{N} \sum_{i=1}^{N} \int_{0}^{t} \int_{0}^{\infty}1_{\{E_{r^{-}}^{i}=S\}}\varphi(X^{i})1_{\{u\leq \frac{\beta}{N} \sum_{j=1}^{N}\frac{K(X^{i},X^{j})}{\sum\limits_{l=1}^{N}K(X^{l},X^{j})} 1_{\{E_{r}^{j}=I\}}\}} \overline{M}^{i}(dr,du)) \\\hspace*{1cm}- \frac{1}{N} \sum_{i=1}^{N} \int_{0}^{t} \int_{0}^{\alpha}1_{\{E_{r^{-}}^{i}=I\}}\varphi(X^{i})\overline{Q}^{i}(dr,du).$ \vspace*{0.16cm}\\
                            Let us state the main result of this subsection. 
                            \begin{thm}
                           The sequence $(\mu^{S,N},\mu^{I,N})_{N \geq 1}$  converges in probability  in $ (D(\mathbb{R}_{+},\mathcal{M}(\mathbb{T}^{2})))^{2} $ towards $ (\mu^{S},\mu^{I})$ $\in (C(\mathbb{R}_{+},\mathcal{M}(\mathbb{T}^{2})))^{2}$ where $\{(\mu_{t}^{S},\mu_{t}^{I}), t\geq 0 \}$ satisfies, $\forall \varphi \in C(\mathbb{T}^{2}),$\vspace*{0.2cm}\\
     $\hspace*{2cm}\displaystyle   (\mu_{t}^{S},\varphi)=(\mu_{0}^{S},\varphi) \displaystyle   - \beta\int_{0}^{t}\left(\mu_{r}^{S}, \varphi (\mu_{r}^{I},\frac{K}{(\mu,K)})\right) dr,\hspace*{5.5cm}(7.1) 
                                                                       \vspace*{0.17cm}\\
   \hspace*{2cm}\displaystyle  \displaystyle (\mu_{t}^{I},\varphi)=(\mu_{0}^{I},\varphi)  +\beta \int_{0}^{t}\left(\mu_{r}^{S}, \varphi (\mu_{r}^{I},\frac{K}{(\mu,K)})\right) dr -\alpha\int_{0}^{t}(\mu_{r}^{I},\varphi)dr.\hspace*{2.7cm}(7.2)$ \label{r1}
                
                                     \end{thm}
\begin{proof}
 We obtain  the tightness of the sequence $ (\mu^{S,N},\mu^{I,N})_{N \geq 1 }$ by an adaptation of the proof of Proposition \ref{eee}, thus  by Prokhorov's theorem we deduce the existence of a  subsequence which converges in law towards  $\{ (\mu_{t}^{S},\mu_{t}^{I}), t\geq 0  \}.$ Furthemore, adapting the proof of Theorem \ref{eeee} we prove that  $\{(\mu_{t}^{S},\mu_{t}^{I}), t\geq 0  \}$ is continuous  and verifies  the system formed by the equations (7.1) and (7.2), and  $\forall t\geq 0$,  $\mu_{t}^{S}$ and $\mu_{t}^{I}$ are absolutely continuous with respect to the Lebesgue measue with densities $f_{S}(t,.)$ and $f_{S}(t,.)$ bounded by $\delta_{2}$($\delta_{2}$ is defined in section \ref{se8}). \\On the  other hand   $(f_{S}(t,.),f_{I}(t,.))$ 
  satisfies the following system.\vspace*{0.18cm}\\
   $ \hspace*{2.5cm}\displaystyle  f_{S}(t) =f_{S}(0)-\beta \int_{0}^{t} f_{S}(r) \int_{\mathbb{T}_{2}}\frac{K(.,y) }{\int_{\mathbb{T}^{2}}K(x',y)g(x')dx'} f_{I}(r,y)dydr,
                      \vspace*{0.15cm}\\
   \hspace*{2.5cm}\displaystyle   f_{I}(t) =f_{I}(0)+\beta \int_{0}^{t} f_{S}(r) \int_{\mathbb{T}_{2}}\frac{K(.,y) }{\int_{\mathbb{T}^{2}}K(x',y)g(x')dx'} f_{I}(r,y)dydr-\alpha \int_{0}^{t}f_{I}(r)dr,
    $\\
                   where $g$ is the density of $\mu$.  Moreover it is  easy to prove that the above system admits a unique solution  in the set $\Lambda=\{(f_{1},f_{2})/0\leq f_{i}\leq\delta_{2},i\in\{1,2\}\}$. Thus we conclude that the sequence $ (\mu^{S,N},\mu^{I,N})_{N \geq 1 }$ weakly converges in  $(D(\mathbb{R}_{+},\mathcal{M}(\mathbb{T}^{2})))^{2}$ towards $(\mu^{S},\mu^{I}).$ However as the initial measures $\mu_{0}^{S}$ and $\mu_{0}^{I}$ are deterministic (see Theorem 3.1 in section \ref{se8}), the measures $\mu^{S}_{t}$ and $\mu^{I}_{t}$ also have this property, consequently  we have the convergence in probability.
\end{proof}

\subsection{Central Limit Theorem}
The aim here is to study the convergence of the  sequence  
                        $ (U^{N}=\sqrt{N}(\mu^{S,N}-\mu^{S}),V^{N}=\sqrt{N}(\mu^{I,N}-\mu^{I})), $ as $N\rightarrow \infty$ in  $ D(\mathbb{R}_{+},H^{-s})\times   D(\mathbb{R}_{+},H^{-s})$ with s>1 (the choice of such a 's' is justified as in subsection \ref{ss1}). To this end  we make the following assumptions.\vspace*{0.2cm}\\
\textbf{Assuptions (H4)}\vspace*{0.2cm}: $\hspace*{0.1cm}  k \in C^{2}(\mathbb{R}_{+}).$ \\
Following the same argument as that used in the proof of lemma \ref{ff3}, we note that:
\begin{rmq} Under the assumption $(H4)$,    $\underset{x}{\sup}\Arrowvert K(x,.) \Arrowvert_{H^{s}}<\infty$ \label{ape1}\vspace*{-0.18cm}
\end{rmq}
\subsubsection{Equations verified by the pair $ (U^{N},V^{N})$ } 
 Let $ \varphi \in C(\mathbb{T}^{2})$, by a similar reasoning as in subsection \ref{se6}, we see that \vspace*{-0.2cm} \[\tag{7.3}\displaystyle (U_{t}^{N},\varphi)=\displaystyle  (U_{0}^{N},\varphi)  + \beta \int_{0}^{t}(Z^{N} , G_{r}^{S,I,N}\varphi)dr -\beta\int_{0}^{t}(U_{r}^{N} ,G_{r}^{I,N}\varphi )dr - \beta\int_{0}^{t}( V_{r}^{N},G_{r}^{S}\varphi )dr + \widetilde{M}_{t}^{' N,\varphi},\label{rr1}\vspace*{-0.25cm}\]
 $\displaystyle (V_{t}^{N},\varphi)=(V_{0}^{N},\varphi)- \beta \int_{0}^{t}(Z^{N} , G_{r}^{S,I,N}\varphi)dr +\beta\int_{0}^{t}(U_{r}^{N} ,G_{r}^{I,N}\varphi )dr +\beta\int_{0}^{t}( V_{r}^{N},G_{r}^{S}\varphi )dr  $\vspace*{-0.3cm}\[\tag{7.4}\hspace*{-8.5cm}-\alpha \int_{0}^{t}(V_{r}^{N},\varphi)dr +\widetilde{L}_{t}^{' N,\varphi},\label{rr2}\vspace*{-0.2cm}\] where  \vspace*{0.1cm}\\\hspace*{0.6cm} - $ \widetilde{M}_{t}^{' N,\varphi}=\sqrt{N}M_{t}^{' N,\varphi}$ and  $ \widetilde{L}_{t}^{' N,\varphi}=\sqrt{N}L_{t}^{' N,\varphi},$\vspace*{0.17cm} \\\hspace*{0.6cm} - $\forall x,y,x' \in \mathbb{T}_{2}$,\vspace*{0.18cm} \\\hspace*{1cm} $G_{r}^{S,I,N}\varphi(x')=\displaystyle\left(\mu_{r}^{I,N}, K(x',.)\frac{(\mu_{r}^{S,N},\varphi K)}{(\mu^{N},K)(\mu,K)}\right)\vspace*{0.15cm}\\\hspace*{3.18cm}=\displaystyle\int_{\mathbb{T}^{2}}K(x',y)\frac{\int_{\mathbb{T}^{2}}\varphi(x)K(x,y)\mu_{r}^{S,N}(dx)}{\int_{\mathbb{T}^{2}} K(y',y)\mu^{N}(dy')\int_{\mathbb{T}^{2}} K(y',y)\mu(dy')}\mu_{r}^{I,N}(dy),$\vspace*{0.18cm}\\ \hspace*{1cm} $G_{r}^{I,N}\varphi(x)= \displaystyle\varphi(x)\Big(\mu_{r}^{I,N}, \frac{K(x,.)}{(\mu,K)}\Big)=\displaystyle\varphi(x) \int_{\mathbb{T}^{2}}\frac{K(x,y)}{\int_{\mathbb{T}^{2}} K(y',y)\mu(dy')}\mu_{r}^{I,N}(dy),$  \vspace*{0.19cm} \\\hspace*{1cm} $G_{r}^{S}\varphi(y)=\displaystyle \frac{(\mu_{r}^{S}, \varphi K(.,y))}{(\mu,K(.,y))}=\displaystyle\frac{\int_{\mathbb{T}^{2}}\varphi(x) K(x,y)\mu_{r}^{S}(dx)}{\int_{\mathbb{T}^{2}} K(y',y)\mu(dy')}.$\vspace*{0.35cm}\\\textbf{In the rest of this section we arbitrarily choose $1<s<2$,} and we equipped $D(\mathbb{R}_{+},H^{-s})$ with the Skorhokod topology. \newpage
  \begin{thm}
        Under  (H4),  the sequence  $(U^{N},V^{N})_{N\geq 1}$ converges in law   in $(D(\mathbb{R}_{+},H^{-s}))^{2}$ to the pair of process  $(U,V)$ which belongs  to $ (C(\mathbb{R}_{+},H^{-s}))^{2}$
             and  satisfies: \vspace*{0.12cm}\\
          $ U_{t}=\displaystyle U_{0} + \beta \int_{0}^{t} (G_{r}^{S,I})^{*}Zdr- \beta \int_{0}^{t} (G_{r}^{I})^{*}U_{r}dr  -\beta \int_{0}^{t} (G_{r}^{S})^{*} V_{r}dr + W_{t}^{' 1},\hspace*{3.8cm}(7.5) $\vspace*{0.1cm}\\
           $ V_{t}=\displaystyle V_{0}  - \beta \int_{0}^{t}(G_{r}^{S,I})^{*}Zdr+ \beta \int_{0}^{t} (G_{r}^{I})^{*}U_{r}dr  + \beta \int_{0}^{t} (G_{r}^{S})^{*}V_{r}dr -\alpha \int_{0}^{t}V_{r}dr+W_{t}^{' 2},\hspace*{1.8cm}(7.6)$\vspace*{0.1cm} \\
                     where $\forall \varphi,\psi \in H^{s}, ((W^{' 1},\varphi) ,(W^{' 2},\psi)) $ is a centered Gaussian martingale   verifying \\
                                 $\hspace*{2cm}< (W^{' 1},\varphi)>_{t}=\displaystyle\beta\int_{0}^{t}\Big(\mu_{r}^{S},\varphi^{2}(\mu_{r}^{I},\frac{K}{(\mu,K)})\Big)dr,$\vspace*{0.1cm}\\
                          $ \hspace*{2cm} < (W^{' 2},\psi)>_{t}=\displaystyle\beta\int_{0}^{t}\Big(\mu_{r}^{S},\psi^{2}(\mu_{r}^{I},\frac{K}{(\mu,K)})\Big)dr+\alpha \int_{0}^{t}(\mu_{r}^{I},\psi^{2})dr,$\vspace*{0.1cm}\\
                                                                 $\hspace*{2cm} < (W^{' 1},\varphi),(W^{' 2},\psi)>_{t}=-\displaystyle\beta\int_{0}^{t}\Big(\mu_{r}^{S},\varphi\psi(\mu_{r}^{I},\frac{K}{(\mu,K)})\Big)dr.$ \label{r2}
                                                       \end{thm}
                                                 The proof of   the next Proposition, which will be useful for the proof of Theorem \ref{r2} is an easy adaptation of the proof of Proposition \ref{f4}.
                                                           \begin{prp}
      The  sequence $ \{(\widetilde{M}_{t}^{' N},\widetilde{L}_{t}^{' N}), t\geq0 \}_{N\geq1} $ converges in law  in $ (D(\mathbb{R}_{+},H^{-s}))^{2}$ towards  the process $(W^{' 1},W^{' 2})\in (C(\mathbb{R}_{+},H^{-s}))^{2} $  and    $ \forall \varphi,\psi \in H^{s}, ((W^{' 1},\varphi),(W^{' 2},\psi))$ is a centered Gaussian martingale of the form\vspace*{0.18cm}\\
                $ (W_{t}^{' 1},\varphi)=-\displaystyle\displaystyle\int_{0}^{t}\int_{\mathbb{T}^{2}}\sqrt{\beta f_{S}(r,x)\int_{\mathbb{T}^{2}}\frac{K(x,y)}{\int_{\mathbb{T}^{2}}K(x',y)g(x')dx'}f_{I}(r,y)dy}\varphi(x)\mathcal{W}_{1}(dr,dx)$\vspace*{0.18cm}\\
                $(W_{t}^{' 2},\psi)=\displaystyle\displaystyle\int_{0}^{t}\int_{\mathbb{T}^{2}}\sqrt{\beta f_{S}(r,x)\int_{\mathbb{T}^{2}}\frac{K(x,y)}{\int_{\mathbb{T}^{2}}K(x',y)g(x')dx'}f_{I}(r,y)dy}\psi(x)\mathcal{W}_{1}(dr,dx)\\\hspace*{2cm}-\int_{0}^{t}\int_{\mathbb{T}^{2}}\psi(x)\sqrt{\alpha f_{I}(r,x)}\mathcal{W}_{2}(dr,dx),$\vspace*{0.1cm}\\where $g$ is the density of $\mu$ and  $\mathcal{W}_{1},\mathcal{W}_{2}$ are independent spatio-temporal white noises.\vspace*{-0.25cm}\label{rrt}
                \end{prp} 
              \subsubsection{\textit{Proof of Theorem \ref{r2}}}
             We establish the tightness of the sequence $(U^{N},V^{N})$ first, then we  show that all converging subsequences of $(U^{N},V^{N})_{N\geq 1} $ have the same limit which we shall identify.\vspace*{0.1cm}\\The next proposition is useful to  prove the tightness of the sequence $(U^{N},V^{N})_{N\geq 1}.$
             \begin{prp}
              $\forall$ $ T>0$, \\\hspace*{4cm} $\underset{N\geq 1}{\sup}\mathbb{E}(\underset{0\leq t \leq T}{\sup}\lVert U_{t}^{N}\lVert_{H^{-s}}^{2})<\infty,  $ \\ \hspace*{4cm} $\underset{N\geq 1}{\sup}\mathbb{E}(\underset{0\leq t \leq T}{\sup}\lVert V_{t}^{N}\lVert_{H^{-s}}^{2})<\infty. $\label{r3} 
             \end{prp}
             \begin{proof}
             From equations (\ref{rr1}) and (\ref{rr2}),  we have\vspace*{0.18cm} \\
            $ \lVert U_{t}^{N} \lVert_{H^{-s}}^{2}\leq\displaystyle  5\lVert U_{0}^{N} \lVert_{H^{-s}}^{2}  + 5\beta^{2} t \int_{0}^{t}\lVert (G_{r}^{S,I,N})^{*}Z^{N} \lVert_{H^{-s}}^{2}dr +5\beta^{2} t\displaystyle\int_{0}^{t} \lVert (G_{r}^{I,N})^{*}U_{r}^{N}\lVert_{H^{-s}}^{2}dr$ \\$\hspace*{2cm}+\displaystyle5\beta^{2} t\displaystyle\int_{0}^{t} \lVert (G_{r}^{S})^{*}V_{r}^{N}\lVert_{H^{-s}}^{2}dr  + 5\lVert\widetilde{M}_{t}^{' N,\varphi}\lVert_{H^{-s}}^{2},$\vspace*{0.15cm}\\
              $ \lVert V_{t}^{N} \lVert_{H^{-s}}^{2}\leq\displaystyle  6\lVert V_{0}^{N} \lVert_{H^{-s}}^{2}  + 6\beta^{2} t \int_{0}^{t}\lVert (G_{r}^{S,I,N})^{*}Z^{N} \lVert_{H^{-s}}^{2}dr +6\beta^{2} t\int_{0}^{t} \lVert (G_{r}^{I,N})^{*}U_{r}^{N}\lVert_{H^{-s}}^{2}dr \\\hspace*{2cm}+6\beta^{2} t\int_{0}^{t} \lVert (G_{r}^{S})^{*}V_{r}^{N}\lVert_{H^{-s}}^{2}dr +6\alpha^{2} t\int_{0}^{t}\lVert V_{r}^{N}\lVert_{H^{-s}}^{2}dr + 6\lVert\widetilde{L}_{t}^{' N,\varphi}\lVert_{H^{-s}}^{2}.$
             \\ So  since  Proposition \ref{f5} and Corollary \ref{ff2} remain true for $\mu_{t}=\mu$, we have \vspace*{0.1cm}\\
              $ \lVert U_{t}^{N} \lVert_{H^{-s}}^{2}\leq\displaystyle  5\lVert U_{0}^{N} \lVert_{H^{-s}}^{2}  + 5\beta^{2} t^{2}C \underset{y}{\sup}\lVert K(.,y)\lVert_{H^{s}}^{2}\lVert Z^{N} \lVert_{H^{-s}}^{2} +5\beta^{2} t C\underset{y}{\sup}\lVert K(.,y)\lVert_{H^{s}}^{2}\int_{0}^{t} \lVert U_{r}^{N}\lVert_{H^{-s}}^{2}dr \\\hspace*{2cm}+5\beta^{2} tC\underset{x}{\sup}\Big\Arrowvert \frac{K(x,.)}{\int_{\mathbb{T}^{2}}K(x',.)\mu(dx')}\Big \Arrowvert_{H^{s}}^{2}\int_{0}^{t} \lVert V_{r}^{N}\lVert_{H^{-s}}^{2}dr  + 5\lVert\widetilde{M}_{t}^{' N,\varphi}\lVert_{H^{-s}}^{2},$\\
                            $ \lVert V_{t}^{N} \lVert_{H^{-s}}^{2}\leq\displaystyle  6\lVert V_{0}^{N} \lVert_{H^{-s}}^{2}  + 6\beta^{2} t^{2}C\underset{y}{\sup}\lVert K(.,y)\lVert_{H^{s}}^{2} \lVert Z^{N} \lVert_{H^{-s}}^{2} +6\beta^{2} tC\underset{y}{\sup}\lVert K(.,y)\lVert_{H^{s}}^{2}\int_{0}^{t} \lVert U_{r}^{N}\lVert_{H^{-s}}^{2}dr \\\hspace*{1cm}+6t\Big(C\beta^{2} \underset{x}{\sup}\Big\Arrowvert \frac{K(x,.)}{\int_{\mathbb{T}^{2}}K(x',.)\mu(dx')}\Big \Arrowvert_{H^{s}}^{2}+\alpha^{2}\Big) \int_{0}^{t} \lVert V_{r}^{N}\lVert_{H^{-s}}^{2}dr  + 6\lVert\widetilde{L}_{t}^{' N,\varphi}\lVert_{H^{-s}}^{2}.$\vspace*{0.1cm}\\Thus  from Remark \ref{ape1} and  Lemma \ref{ape} below, one has\\
                             $\displaystyle  \mathbb{E}(\underset{0\leq l\leq t}{\sup}\lVert U_{l}^{N}\lVert^{2}_{H^{-s}})\leq 5\underset{N\geq 1}{\sup}\mathbb{E}(\lVert U_{0}^{N} \lVert_{H^{-s}})+5t\beta^{2} C\int_{0}^{t} \Big\{\mathbb{E}(\underset{0\leq u\leq r}{\sup}\lVert U_{u}^{N} \lVert^{2}_{H^{-s}})+  \mathbb{E}(\underset{0\leq u\leq r}{\sup}\lVert V_{u}^{N} \lVert^{2}_{H^{-s}})\Big\} dr\\\hspace*{3.9cm}+5t^{2}\beta^{2} C \underset{N\geq 1}{\sup}\mathbb{E}( \lVert Z^{N}\lVert_{H^{-s}})+\underset{N\geq 1}{\sup}\mathbb{E}(\underset{0\leq r\leq t}{\sup}\lVert \widetilde{M}^{'N}_{r} \lVert_{H^{-s}}),\hspace*{2.97cm}(7.7)$\vspace*{0.19cm}\\
                                         $\displaystyle \mathbb{E}(\underset{0\leq l\leq t}{\sup}\lVert V_{l}^{N}\lVert^{2}_{H^{-s}})\leq 6\underset{N\geq 1}{\sup}\mathbb{E}(\lVert V_{0}^{N} \lVert^{2}_{H^{-s}})\\\hspace*{4.9cm}+6t(\beta^{2} C+\alpha^{2}) \int_{0}^{t}  \Big\{\mathbb{E}(\underset{0\leq u\leq r}{\sup}\lVert U_{u}^{N} \lVert^{2}_{H^{-s}})+\mathbb{E}(\underset{0\leq u\leq r}{\sup}\lVert V_{u}^{N} \lVert^{2}_{H^{-s}})\Big\} dr\\\hspace*{4.9cm}+6t^{2}\beta^{2} C \underset{N\geq 1}{\sup} \mathbb{E}( \lVert Z^{N}\lVert^{2}_{H^{-s}})+ \underset{N\geq 1}{\sup}\mathbb{E}(\underset{0\leq r\leq t}{\sup}\lVert \widetilde{L}^{'N}_{r} \lVert_{H^{-s}}).\hspace*{2.2cm}
                                         (7.8) $ \\
                                         Hence summing (7.7) and (7.8) and applying Gronwall's lemma we obtain the result from Corollary \ref{d4} and Proposition \ref{uo} in section 4.
             \end{proof}
            Now we prove  the tightness of the sequence $(U^{N},V^{N})_{N\geq 1}$ in $(D(\mathbb{R}_{+},H^{-s}))^{2}$.
             \begin{prp}
         Both sequences $(U^{N})_{N\geq 1}$ and  $(V^{N})_{N\geq 1}$  are tight in  $D(\mathbb{R}_{+},H^{-s})$. \label{r4}
            \end{prp}
            \begin{proof} We only  establish  the tightness of 
             $(U^{N})_{N\geq 1}$  by showing that the conditions (T1) and (T2) of Proposition \ref{d7} are satisfied. \\ 
             The condition (T1) is obtained by using the  Proposition \ref{r3} and applying an argument similar to that of the proof of (T1)  in Theorem \ref{d6}.\vspace*{0.2cm}\\
                             - Proof of (T2).
                             we have \\
                          $\displaystyle U_{t}^{N}=\displaystyle  U_{0}^{N}  + \beta \int_{0}^{t}(G_{r}^{S,I,N})^{*}Z^{N} dr -\beta\int_{0}^{t} (G_{r}^{I,N})^{*}U_{r}^{N} dr - \beta\int_{0}^{t}(G_{r}^{S})^{*} V_{r}^{N}dr + \widetilde{M}_{t}^{' N,\varphi}\\\hspace*{0.6cm}=\displaystyle  U_{0}^{N} +\beta\int_{0}^{t}\Gamma_{r}^{N}dr+\widetilde{M}_{t}^{' N,\varphi}$ with  $\Gamma_{r}^{N}=(G_{r}^{S,I,N})^{*}Z^{N}-(G_{r}^{I,N})^{*}U_{r}^{N}-(G_{r}^{S})^{*} V_{r}^{N}.$ \\
                          We  want to prove that \\
        $  \forall T>0,\hspace*{0.1cm}\forall \varepsilon_{1},\varepsilon_{2}>0 ,\exists\delta>0, N_{0}\geq 1$ such that for any family of  stopping times $(\tau^{N})_{N}$ with  $\tau^{N} \leq T, \vspace*{0.3cm}\\\hspace*{4cm}  \underset{\underset{\delta\geq\theta}{N\geq N_{0}}}{\sup} $ $\mathbb{P}\left(\Big\lVert\displaystyle \int_{\tau^{N}}^{\tau^{N}+\theta}\Gamma_{r}^{N}dr\Big\lVert_{H^{-s}}>\varepsilon_{1}\right)<\varepsilon_{2},\hspace*{4.92cm}(7.9) $  \\\hspace*{4cm} $\underset{\underset{\delta\geq\theta}{N\geq N_{0}}}{\sup} $ $\mathbb{P}\left(\Big\lVert \widetilde{M}^{' N}_{(\tau^{N}+\theta)}-\widetilde{M}^{' N}_{\tau^{N}}\Big\lVert_{H^{-s}}>\varepsilon_{1}\right)<\varepsilon_{2}.\hspace*{4.3cm}(7.10)$\\$- (7.10)$ is proved by using an argument similar to that of the proof of (T2) in  Proposition \ref{f1}\vspace*{0.2cm} \\
        $- $ Proof of $(7.9).$ Let  $T>0$, $l\in \mathbb{R}_{+}\backslash[0,1]$,  $\varepsilon_{1},\varepsilon_{2}>0 $, we find $\delta>0, N_{0}\geq 1$ such that $\delta+\tau^{N}\leq lT$ and \\ $ \hspace*{3cm}\underset{\underset{\delta\geq\theta}{N\geq N_{0}}}{\sup} $ $\mathbb{P}\left(\Big\lVert \displaystyle\int_{\tau^{N}}^{\tau^{N}+\theta}\Gamma^{N}_{r}dr\Big\lVert_{H^{-s}}>\varepsilon_{1}\right)<\varepsilon_{2}. $ \\ We have  $\Big\lVert\displaystyle \int_{\tau^{N}}^{\tau^{N}+\theta}\Gamma^{N}_{r}dr\Big\lVert_{H^{-s}}\leq \displaystyle\int_{\tau^{N}}^{(\tau^{N}+\theta)}\lVert \Gamma^{N}_{r}\lVert_{H^{-s}} dr\leq \theta \underset{0\leq r\leq lT}{\sup}\lVert \Gamma^{N}_{r}\lVert_{H^{-s}}$.\\ Thus $(7.9)$ will follows from  $\displaystyle  \underset{N\geq1}{\sup} \mathbb{E}(\underset{0\leq r\leq lT}{\sup}\lVert \Gamma_{r}^{N} \lVert_{H^{-s}})<C$ in view of the last inegality, which in turn follows readilly form Lemma \ref{ape} below,  Remark \ref{ape1} and Propositions \ref{uo} and \ref{r3}, combined with the fact that Corollary \ref{ff2} remains true for $\mu_{t}=\mu$.
         So (T2) is established.  \end{proof} 
         An easy adaptation of the proof of Lemmas \ref{f11} and \ref{f12} yields respectively the next two Lemmas.
         \begin{lem}
           For any $ t\geq0$, $\varphi\in H^{2}(\mathbb{T}^{2})$,  as $N\longrightarrow \infty$, \\\hspace*{5cm} $ \displaystyle\int_{0}^{t}\mathbb{E}\Big(\lVert (G_{r}^{I,N}-G_{r}^{I})\varphi \lVert_{H^{s}}^{2}\Big)^{\frac{1}{2}}dt\longrightarrow 0.$\label{fmm}
          \end{lem}
          \begin{lem}
                     For any $ t\geq0$, $\varphi\in H^{s}(\mathbb{T}^{2})$,  as $N\longrightarrow \infty$, \\\hspace*{5cm} $ \displaystyle\int_{0}^{t}\mathbb{E}\Big(\lVert (G_{r}^{S,I,N}-G_{r}^{S,I})\varphi \lVert_{H^{s}}^{2}\Big)^{\frac{1}{2}}dt\longrightarrow 0.$\label{fmmm}
                    \end{lem}
         
       From Proposition  \ref{r4}, we deduce that  the sequence $(U^{N},V^{N})_{N}$ is tight in $(D(\mathbb{R}_{+},H^{-s}))^{2},$ so there exists a subsequence still denoted $(U^{N},V^{N})_{N}$, which converges in law in $(D(\mathbb{R}_{+},H^{-s}))^{2}$ towards $(U,V)$. Moreover from Proposition \ref{rrt}, we deduce that $(U,V)\in$ $(C(\mathbb{R}_{+},H^{-s}))^{2}$, thus we end the proof of Theorem \ref{r2} as follows.  
     \begin{prp}
     The pair $(U,V)$ is the unique solution in $(C(\mathbb{R}_{+},H^{-s}))^{2}$ of  the system formed by equations (7.5) and (7.6).
     \end{prp}
     \begin{proof}By adapting the proof of Proposition \ref{llln} we obtain the following results\vspace*{0.16cm}\\
    $\hspace*{3.8cm}\displaystyle\int_{0}^{t}(U_{r}^{N},G_{r}^{I,N}\varphi )dr\xrightarrow{L}\int_{0}^{t}(U_{r},G_{r}^{I}\varphi )dr,$\\
      $\hspace*{3.8cm}\displaystyle\int_{0}^{t}(Z_{r}^{N},G_{r}^{S,I,N}\varphi )dr\xrightarrow{L}\int_{0}^{t}(Z_{r},G_{r}^{S,I}\varphi )dr. $\vspace*{0.2cm}\\
Therefore, with the convergence of the other terms of equations (7.3) and (7.4), we see that the pair $ (U, V) $ satisfies the  equations (7.5) and (7.6). By adapting the proof of Proposition \ref{rt1} we also see that the system formed by  equations (7.5) and (7.6) admits a unique solution in $(C(\mathbb{R}_{+},H^{-s}))^{2}$.\vspace*{-0.3cm}
           \end{proof}  
           \section{Appendix}
            We first recall that for any $s>0$, the family  $(\rho^{i,s}_{n_{1},n_{2}})_{i,n_{1},n_{2}}$ (as defined in Proposition \ref{b2}) is an orthonormal basis of $H^{s}(\mathbb{T}^{2}).$\vspace*{0.1cm}\\ In this appendix we prove the next two Lemmas.
                   \begin{lem} \label{ap} $\forall x\in \mathbb{T}_{2}$, we have,\vspace*{0.2cm}\\\hspace*{5cm}
                                         $\displaystyle\sum\limits_{i,n_{1},n_{2}}(\rho_{n_{1},n_{2}}^{i,s}(x))^{2}<\infty$ iff s>1,\\
                                             $\hspace*{5cm}\displaystyle\sum\limits_{i,n_{1},n_{2}}(\bigtriangledown\rho_{n_{1},n_{2}}^{i,s}(x))^{2}<\infty$ iff s>2.          
                   \end{lem}
                   \begin{proof}
                   As for any $x\in \mathbb{T}^{2}$, $i \in [\lvert 1,8\lvert]$,  $0\leq(f_{n_{1},n_{2}}^{i}(x))^{2}\leq 4,$ \vspace*{0.2cm}\\
                  $\displaystyle\sum\limits_{i,n_{1},n_{2}}(\rho_{n_{1},n_{2}}^{i,s}(x))^{2}$ 
                      $=\displaystyle
                      \sum\limits_{i,n_{1},n_{2}}\frac{(f_{n_{1},n_{2}}^{i})^{2}(x)}{(1+\gamma\pi^{2}(n_{1}^{2}+n_{2}^{2}))^{s}} $  and, \\ $ \displaystyle\sum\limits_{i,n_{1},n_{2}}(\bigtriangledown\rho_{n_{1},n_{2}}^{i,s}(x))^{2}=\displaystyle \pi^{2} \sum\limits_{i=1}^{4}\sum\limits_{n_{1}>0,n_{2}>0,even}\frac{n_{1}^{2}+n_{2}^{2}}{(1+\gamma\pi^{2}(n_{1}^{2}+n_{2}^{2}))^{s}}(f_{n_{1},n_{2}}^{i})^{2}(x)$\newpage $ \hspace*{2.5cm} +\sum\limits_{n_{1}>0,even}\frac{\pi^{2}n_{1}^{2}[(f_{n_{1},0}^{6})^{2}(x)+(f_{n_{1},0}^{5})^{2}(x)]}{(1+\gamma\pi^{2}n_{1}^{2})^{s}}+\sum\limits_{n_{2}>0,even}\frac{\pi^{2}n_{2}^{2}[(f_{n_{1},0}^{8})^{2}(x)+(f_{n_{1},0}^{7})^{2}(x)]}{(1+\gamma\pi^{2}n_{2}^{2})^{s}}$ \\So\\
                      $\displaystyle\sum\limits_{i,n_{1},n_{2}}\{(\rho_{n_{1},n_{2}}^{i,s}(x))^{2}\leq 1+ 16  \sum\limits_{n_{1}>0,n_{2}>0,even}\frac{1}{(1+\gamma\pi^{2}(n_{1}^{2}+n_{2}^{2}))^{s}}+8\sum\limits_{i=1}^{2}\sum\limits_{n_{i}>0,even}\frac{1}{(1+\gamma\pi^{2}n_{i}^{2})^{s}}$ and, $\\\displaystyle\sum\limits_{i,n_{1},n_{2}}(\bigtriangledown\rho_{n_{1},n_{2}}^{i,s}(x))^{2}\leq 16\pi^{2}\sum\limits_{n_{1}>0,n_{2}>0,even}\frac{n_{1}^{2}+n_{2}^{2}}{(1+\gamma\pi^{2}(n_{1}^{2}+n_{2}^{2}))^{s}}+8\pi^{2}\sum\limits_{i=1}^{2}\sum\limits_{n_{i}>0,even}\frac{n_{i}^{2}}{(1+\gamma\pi^{2}n_{i}^{2})^{s}}$   \\
                      \\Hence we see that:\vspace*{0.16cm}\\
                      $-$ $\displaystyle\sum\limits_{i,n_{1},n_{2}}(\rho_{n_{1},n_{2}}^{i,s}(x))^{2}<\infty$  provided the series $ \sum\limits_{n_{1}>0,n_{2}>0}\frac{1}{(1+\gamma\pi^{2}(n_{1}^{2}+n_{2}^{2}))^{s}}$; $ \sum\limits_{n_{1}>0}\frac{1}{(1+\gamma\pi^{2}n_{1}^{2})^{s}}$ \\and $ \sum\limits_{n_{2}>0}\frac{1}{(1+\gamma\pi^{2}n_{2}^{2})^{s}}$ converge.\vspace*{0.1cm}\\
                      $-$ $ \displaystyle\sum\limits_{i,n_{1},n_{2}}(\bigtriangledown\rho_{n_{1},n_{2}}^{i,s}(x))^{2}<\infty  $ provided the series $\sum\limits_{n_{1}>0,n_{2}>0}\frac{n_{1}^{2}+n_{2}^{2}}{(1+\gamma\pi^{2}(n_{1}^{2}+n_{2}^{2}))^{s}}$;  $\sum\limits_{n_{1}>0}\frac{n_{1}^{2}}{(1+\gamma\pi^{2}n_{1}^{2})^{s}}$\\ and $ \sum\limits_{n_{2}>0}\frac{n_{2}^{2}}{(1+\gamma\pi^{2}n_{2}^{2})^{s}}$  converge.\vspace*{0.3cm}\\
                     \textbf{- Convergence of the series }\\ 1) Convergence of  $\sum\limits_{n_{1}>0,n_{2}>0}\frac{n_{1}^{2}+n_{2}^{2}}{(1+\gamma\pi^{2}(n_{1}^{2}+n_{2}^{2}))^{s}}$ \\ It is so easy  to see that  $ \sum\limits_{n_{1}\geq1,n_{2}\geq 1}\frac{n_{1}^{2}+n_{2}^{2}}{(1+\gamma\pi^{2}(n_{1}^{2}+n_{2}^{2}))^{s}} $  and $\displaystyle\int_{1}^{+\infty}\int_{1}^{+\infty}\frac{x^{2}+y^{2}}{(1+\gamma\pi^{2}(x^{2}+y^{2}))^{s}}dxdy$ 
                                                                                     are of the same type (either convergent or divergent), 
                          and the latter is of the same type as\vspace*{0.12cm}\\\hspace*{1cm} $\displaystyle \int_{1}^{+\infty}\frac{r^{3}}{(1+\gamma\pi^{2}r^{2})^{s}}dr \leq \frac{1}{\gamma^{s}\pi^{2s}}\int_{1}^{+\infty}r^{3-2s}dr $
                                      and    $\int_{1}^{+\infty}r^{3-2s}dr<\infty$ iff s>2.\vspace*{0.1cm}\\
                           Thus  $\sum\limits_{n_{1}\geq1,n_{2}\geq1}\frac{n_{1}^{2}+n_{2}^{2}}{(1+\gamma\pi^{2}(n_{1}^{2}+n_{2}^{2}))^{s}} $ converges
                                 iff s>2.\\2) By the same argument as previously $ \sum\limits_{n_{1}>0,n_{2}>0}\frac{1}{(1+\gamma\pi^{2}(n_{1}^{2}+n_{2}^{2}))^{s}} $ converges for s>1.\\3) By the comparison criterion the series $\sum\limits_{n_{1}>0}\frac{1}{(1+\gamma\pi^{2}n_{1}^{2})^{s}}$ and $\sum\limits_{n_{2}>0}\frac{1}{(1+\gamma\pi^{2}n_{2}^{2})^{s}}$ converge for $s>\frac{1}{2}.$\\4)   By the comparison criterion the series  $\sum\limits_{n_{1}>0}\frac{n_{1}^{2}}{(1+\gamma\pi^{2}n_{1}^{2})^{s}}$ and $\sum\limits_{n_{2}>0}\frac{n_{2}^{2}}{(1+\gamma\pi^{2}n_{2}^{2})^{s}}$ converge for $s>\frac{3}{2}$
                                 \end{proof}
                                 \begin{lem}
                               Under the assumption (H2),  for any $t\geq0,$ we have\vspace*{0.1cm} \\\hspace*{5cm}$ \underset{x}{\sup}\Big\Arrowvert \displaystyle\frac{K(x,.)}{\int_{\mathbb{T}^{2}}K(x',.)\mu_{t}(dx')}\Big \Arrowvert^{2}_{H^{3}}<\infty$ \label{ape}
                                 \end{lem}
                               \begin{proof}
                             We refer to Proposition \ref{b4} above.  $H^{3}(\mathbb{T}^{2}) \subset C^{1}(\mathbb{T}^{2}).$\vspace*{0.1cm}\\If we let support$\{K(x,\cdot)\}=A(x),$ we will have\vspace*{0.1cm}\\
                             $\Big\Arrowvert\displaystyle \frac{K(x,.)}{\int_{\mathbb{T}^{2}}K(x',.)\mu_{t}(dx')}\Big \Arrowvert^{2}_{H^{3}}=\displaystyle\sum\limits_{\lvert\eta \lvert\leq3}\int_{A(x)}\Big\vert D^{\eta} \frac{K(x,y)}{\int_{A(y)}K(x',y)\mu_{t}(dx')}\Big\lvert^{2} dy. $\vspace*{0.1cm}\\Now  letting $\displaystyle w_{t}(x,y)=\frac{K(x,y)}{\int_{A(y)}K(x',y)\mu_{t}(dx')},$ \\since from Lemma \ref{ff3}, for any $x\in \mathbb{T}^{2},  K(x,.)\in H^{3}\subset C^{1}(\mathbb{T}^{2}),$ one has\\\\$\displaystyle\frac{\partial w_{t}}{\partial{y_{1}}}(x,y)=\frac{\frac{\partial K}{\partial{y{1}}}(x,y)}{\int_{\mathbb{T}_{2}}K(u,y)\mu_{t}(du)}-\frac{K(x,y)\int_{A(y)}\frac{\partial K}{\partial{y_{1}}}(u,y)\mu_{t}(du)}{(\int_{A(y)}K(u,y)\mu_{t}(du))^{2}}.$\\ $\displaystyle\frac{\partial^{2} w_{t}}{\partial{y_{2}y_{1}}}(x,y)=\frac{\frac{\partial^{2} K}{\partial{y_{2}y_{1}}}(x,y)}{\int_{A(y)}K(u,y)\mu_{t}(du)}-\frac{\frac{\partial K}{\partial{y_{1}}}(x,y)\int_{A(y)}\frac{\partial K}{\partial{y_{2}}}(u,y)\mu_{t}(du)}{(\int_{A(y)}K(u,y)\mu_{t}(du))^{2}}\\\hspace*{2cm}-\displaystyle\frac{\frac{\partial K}{\partial{y_{2}}}(x,y)\int_{A(y)}\frac{\partial K}{\partial{y_{1}}}(u,y)\mu_{t}(du)+K(x,y)\int_{A(y)}\frac{\partial^{2} K}{\partial{y_{2}y_{1}}}(u,y)\mu_{t}(du)}{(\int_{A(y)}K(u,y)\mu_{t}(du))^{2}}\\\hspace*{2cm}+2\frac{K(x,y)\int_{A(y)}\frac{\partial K}{\partial{y_{1}}}(u,y)\mu_{t}(du)\int_{A(y)}\frac{\partial K}{\partial{y_{2}}}(u,y)\mu_{t}(du)}{(\int_{A(y)}K(u,y)\mu_{t}(du))^{3}}.$ \\\\$\displaystyle\frac{\partial^{3} w_{t}}{\partial{y_{1}y_{2}y_{1}}}(x,y)=\frac{\frac{\partial^{3} K}{\partial{y_{1}y_{2}y_{1}}}(x,y)}{\int_{A(y)}K(u,y)\mu_{t}(du)}-2\frac{\frac{\partial^{2} K}{\partial{y_{2}y_{1}}}(x,y)\int_{A(y)}\frac{\partial K}{\partial{y_{1}}}(u,y)\mu_{t}(du)}{(\int_{\mathbb{T}_{2}}K(u,y)\mu_{t}(du))^{2}}\\\hspace*{2cm}-\displaystyle\frac{\frac{\partial^{2} K}{\partial{y_{1}y_{1}}}(x,y)\int_{A(y)}\frac{\partial K}{\partial{y_{2}}}(u,y)\mu_{t}(du)+2\frac{\partial K(x,y)}{\partial  y_{1}}\int_{A(y)}\frac{\partial^{2} K}{\partial{y_{1}y_{2}}}(u,y)\mu_{t}(du)}{(\int_{A(y)}K(u,y)\mu_{t}(du))^{2}}\\\hspace*{2cm}+4\frac{\frac{\partial K}{\partial y_{1}}(x,y)\int_{A(y)}\frac{\partial K}{\partial{y_{2}}}(u,y)\mu_{t}(du)\int_{A(y)}\frac{\partial K}{\partial{y_{1}}}(z,y)\mu_{t}(du)}{(\int_{A(y)}K(u,y)\mu_{t}(du))^{3}}$\\$\hspace*{2cm}-\displaystyle\frac{\frac{\partial K}{\partial y_{2}}(x,y)\int_{A(y)}\frac{\partial^{2} K}{\partial{y_{2}y_{1}}}(u,y)\mu_{t}(du)+K(x,y)\int_{A(y)}\frac{\partial^{3} K}{\partial{y_{1}y_{2}y_{1}}}(u,y)\mu_{t}(du)}{(\int_{A(y)}K(u,y)\mu_{t}(du))^{2}}$\\\\$\hspace*{1.5cm}\displaystyle+2\frac{\int_{A(y)}\frac{\partial K}{\partial{y_{1}}}(u,y)\mu_{t}(du)\Big[\frac{\partial K}{\partial{y_{2}}}(x,y)\int_{\mathbb{T}_{2}}\frac{\partial K}{\partial{y_{1}}}(u,y)\mu_{t}(du)+2K(x,y)\int_{A(y)}\frac{\partial^{2} K}{\partial{y_{2}y_{1}}}(u,y)\mu_{t}(du)\Big]}{(\int_{A(y)}K(u,y)\mu_{t}(du))^{3}}\\\\\hspace*{1.5cm}+2\frac{K(x,y)\int_{A(y)}\frac{\partial^{2} K}{\partial{y_{1}y_{1}}}(u,y)\mu_{t}(du)\int_{A(y)}\frac{\partial K}{\partial{y_{2}}}(u,y)\mu_{t}(du)}{(\int_{A(y)}K(u,y)\mu_{t}(du))^{3}}\\\hspace*{1.5cm}+6\frac{K(x,y)\Big(\int_{A(y)}\frac{\partial K}{\partial{y_{1}}}(u,y)\mu_{t}(du)\Big)^{2}\int_{A(y)}\frac{\partial K}{\partial{y_{2}}}(u,y)\mu_{t}(du)}{(\int_{A(y)}K(u,y)\mu_{t}(du))^{4}}.$\\On the other hand if we let \hspace*{0.1cm}$\mathcal{C}_{1}=\{2k'(\lVert x-y\lVert^{2}),4k''(\lVert x-y\lVert^{2}), 8k^{(3)}(\lVert x-y\lVert^{2})\}$ and \\$\mathcal{C}_{2}=\{(x_{1}-y_{1})^{n}(x_{2}-y_{2})^{m},(n,m)\in\{0,1,2,3\}^{2}\},$
                                                    $\forall \lvert \eta\lvert\leq 3$, $D^{\eta}K(x,y)=D^{\eta}k(\lVert x-y\lVert^{2})$  is written as the sum of the products of elements of  $\mathcal{C}_{1}$ and  $\mathcal{C}_{2}$. Thus as from assumption (H2), 
                                                  $\forall \lvert \eta \lvert\leq 3$, $\lvert k^{(\lvert \eta \lvert)}\lvert$ is bounded in $\mathbb{R}_{+}$, $\forall \lvert \eta\lvert\leq 3,$ $x\in\mathbb{T}^{2},$  $D^{\eta}K(x,y)$ is bounded by a constant independent of $x$. Now since $\forall y\in \mathbb{T}^{2},$  $\int_{\mathbb{T}^{2}}K(u,y)\mu_{t}^{N}(du)=\int_{\mathbb{T}^{2}}K(u,y)f(t,u)(du)$ is lower  bounded  by a constant independent of y and $f(t,.)\leq \delta_{2}$ then we deduce from the above calculations that $\displaystyle\sum\limits_{\lvert\eta \lvert\leq3}\int_{A(x)}\Big\vert D^{\eta} \frac{K(x,y)}{\int_{A(y)}K(x',y)\mu_{r}(dx')}\Big\lvert^{2} dy$ is bounded by a constant independent of $x.$ Hence the result.\vspace*{0.1cm}
                               \end{proof}
                                 Funding: Alphonse Emakoua was supported by EMS Simons For Africa and an Eiffel scholarship program of excellence, and the two other authors by their respective university.
                             \fancyhf{}
                                                                                           \lhead{\textit{REFERENCES}}
                                                                                           \rhead{\thepage}
         \small{  }
                                                  

\begin{thebibliography}{9}
                                                           \bibitem{REFERENCES}
                                                            \label{ac} R.A. Adams. \textit{Sobolev Spaces}. Academic Press, 1975.
                                                              \bibitem{autrereference} \label{bc}
                                                              D. Aldous . Stopping times and tightness. \textit{The Annals of Probability} 6(2), 335-340, 1978
                                                              \bibitem{autrereference} \label{cc}
                                                              L.J.S. Allen, B.M Botlker, Y. lou and A.L. Nevai Asymptotic profiles of the steady states for an SIS epidemic reaction diffusion model, \textit{Discrete Contin. Dyn. Syst.} Ser. A 21, 1-20, 2008.
                                                              \bibitem{autrereference} \label{dc}
                                                             H. Andersson, T. Britton. stochastic epidemic models and their statistical analysis. \textit{Springer Lecture
                                                              Notes in Statistics}. Springer, New York (2000)
                                                              \bibitem{autrereference} \label{ec}
                                                              Hajer Bahouri, Jean-Yves Chemin and Raphaël Danchin. 
                                                              \textit{Fourier Analysis and Nonlinear Partial Differential Equations}, Springer (2011).
                                                              \bibitem{autrereference} \label{fc}
                                                              P. Billingsley, \textit{Convergence of Probability Measures}, 2nd edn. (Wiley, New York, 1999)
                                                               \bibitem{autrereference}  \label{gc}
                                                               Douglas Blount. Density-dependant limits for a nonlinear reaction diffusion models,\textit{The Anals of Probability}, vol. 22, N0.4 2040-2070, 1994
                                                               \bibitem{autrereference} \label{hc}
                                                          Tom. Britton, Etienne Pardoux. Stochastic epidemic in a homogeneous community, Part I of \textit{stochastic epidemic models with inference.} T. Britton, E. Pardoux. eds, Lecture Notes in Mathematics \textbf{2225}, pp. 1-120 Springer 2019
                                                               \bibitem{autrereference} \label{ic}
                                                               Axel Bücher and Ivan Kojadinovic, A note on conditional versus joint unconditional weak convergence in bootstrap consistency results, arXiv:1706.01031v4  [math.ST]  1 Mar 2018
                                                               \bibitem{autrereferences} \label{jc}
                                                              Jie Yen Fan, Kais Hamza, Peter Jagers, and Fima Klebaner.  Spde limits for the age  structure of  a population, arxiv1702.08592v1  
                                                               \bibitem{autrereferences} \label{kc}
                                                                          Loukas Grafakos. \textit{Classical fourier analysis} 2nd. ed Springer. 2008.  
                                                               \bibitem{autrereference} \label{lc}
                                                               J. Jacod and A.N. Shiryaev. \textit{Limit Theorems for Stochastic Processes}. Springer-Verlag, Berlin, 1987.
                                                               \bibitem{autrereference} \label{mc}
                                                               Joffe and  M. Metivier. Weak Convergence of Sequences of Semimartingales with Applications to Multitype Branching Processes,  \textit{Advances in Applied Probability}, Vol. 18, No. 1 , pp. 20-65, 1986 
                                                               \bibitem{autrereference}  \label{nc}
                                                                Olav Kallenberg Lectures on random measures, \textit{Institute of Statistics Mimeo Series} No. 963
                                                                November, 1974
                                                               \bibitem{autrereference} \label{oc}
                                                               P. Kotelenez .High density limit théorems for non linear chemical reactions with diffusion. \textit{Proba Th.Rel.Fields} 78,11-37, 1998 
                                                               \bibitem{autrereference} \label{pc}
                                                               S. Méléard. Convergence of the fluctuations for interacting diffusions with jumps associated with Boltzmann equation. \textit{Stochastics and Stochastics Reports,} 63 :195-225, 1998. 
                                                               \bibitem{autrereference} \label{qc}
                                                                Sylvie Méléard et Sylvie Roelly: Sur les convergences étroite ou vague de processus
                                                               à valeurs mesures. \textit{Comptes rendus de l'académie des Sciences de Paris} Sér.
                                                               1, 317:785-788, 1993.
                                                               \bibitem{autrereference} \label{rc}
                                                                M. Métivier. Convergence faible et principe d'invariance pour des martingales à valeurs dans des espaces de Sobolev. \textit{Annales de l'IHP,} 20(4) :329-348, 1984. 
                                                                \bibitem{autrereference} \label{sc}
                                                                      M. Metivier and S. Nakao, Equivalent conditions for the tightness of a sequence of continuous Hilbert
                                                                      valued martingales, \textit{Nagoya Math. J.} 106, 113-119, 1987.
                                                                      \bibitem{autrereference} \label{tc}
                                                                       Étienne Pardoux. Equations aux dérivées Partielles stochastiques non liiéaires monotones. Etudes des solutions fortes de type Itô, thèse 1975.
                                                          \bibitem{autrereference} \label{vc}
                                              Étienne Pardoux. Moderate Deviations and Extinction of an  Epidemic. \textit{Election. J. Probab} \textbf{25}, paper N° 25, 1-27, 2020
                                                            \bibitem{autrereference} \label{wc}
                                                                     Étienne Pardoux. \textit{Probabilistic Models of Population Evolution, Scaling Limits, Genealogies and Interactions,} Springer 2016.
                                                                     \bibitem{autrereference} \label{zc}
                                                                     Étienne Pardoux. Stochastic Partial Differential Equations and Filtering of Diffusion Processes.    \textit{Stochastics} Vol. 3. pp 127-167, 1979.
                                                               \bibitem{autrereference} \label{uc}
                                                               Sylvie Roelly-Coppoletta. A criterion of convergence of measure-valued processes: application to measure branching processes \textit{Stochastics}, 17 :43-65, 1985.
                                                               \bibitem{autre reference} \label{yc}
                                          M.E Taylor. \textit{Partial differential equations III Nonlinear equations}. Springer 1991
                                          \bibitem{autrereference}
                                                               Viet Chi Tran Modèles particulaires stochastiques pour des problèmes d'évolution adaptative et pour l'approximation de solutions statistiques, thèse 2006 
                                                               
                                                                    
                                                          \end{thebibliography}
\end{document}